\documentclass[12pt]{amsart}
\usepackage{amsmath}
\usepackage{amscd}
\usepackage{amssymb}
\usepackage{amsfonts}
\usepackage{amsthm}
\usepackage{bbm}
\usepackage{cancel}
\usepackage{color}
\usepackage{eucal}
\usepackage{enumerate,yfonts}
\usepackage{enumitem}
\usepackage{relsize}
\usepackage{float}
\usepackage{caption}

 \allowdisplaybreaks

\usepackage{pdfsync}
 \usepackage[all,cmtip]{xy}
\usepackage{graphicx}
\usepackage{graphics}
\usepackage{hyperref}
\usepackage{latexsym}
\usepackage{mathrsfs}
 \usepackage{placeins}
\usepackage{pstricks}
\usepackage{bookmark}

\usepackage{multirow}
\usepackage{array}
\newcolumntype{P}[1]{>{\centering\arraybackslash}p{#1}}
\usepackage{longtable}

\usepackage{stmaryrd}
\usepackage{url}

\DeclareFontFamily{U}{mathx}{\hyphenchar\font45}
\DeclareFontShape{U}{mathx}{m}{n}{
      <5> <6> <7> <8> <9> <10>
     <10.95> <12> <14.4> <17.28> <20.74> <24.88>
    mathx10
      }{}
\DeclareSymbolFont{mathx}{U}{mathx}{m}{n}
\DeclareMathAccent{\widecheck}{\mathalpha}{mathx}{"71}

\newtheorem{thm}{Theorem}[section]
\newtheorem{corollary}[thm]{Corollary}
\newtheorem{lemma}[thm]{Lemma}

\newtheorem{proposition}[thm]{Proposition}

\newtheorem{thm-dfn}[thm]{Theorem-Definition}

\newtheorem{remark}[thm]{Remark}

\newtheorem{theorem}{Theorem}[section]

\numberwithin{equation}{section}

\newcommand{\fa}{{\mathfrak a}}

\newcommand{\fF}{{\mathfrak{F}}}

\newcommand{\Lt}{{\mathfrak{t}}}
\newcommand{\Lg}{{\mathfrak g}}

\newcommand{\rf}{{\mathrm f}}

\newcommand{\rs}{\mathrm s}

\newcommand{\bC}{{\mathbb C}}

\newcommand{\bG}{{\mathbb G}}
\newcommand{\bZ}{{\mathbb Z}}

\newcommand{\bQ}{{\mathbb Q}}

\newcommand{\calF}{{\mathcal F}}

\newcommand{\cO}{{\mathcal O}}

\newcommand{\cN}{{\mathcal N}}
\newcommand{\cH}{{\mathcal H}}

\newcommand{\cP}{{\mathcal P}}
\newcommand{\cE}{{\mathcal E}}

\newcommand{\cZ}{{\mathcal Z}}

\newcommand{\cM}{{\mathcal{M}}}

\newcommand{\cL}{{\mathcal{L}}}

\newcommand{\bft}{\mathbf{t}}

\newcommand{\on}{\operatorname}

\newcommand{\Loc}{\on{LocSys}}

\newcommand{\en}{\on{en}}

\newcommand{\nc}{\newcommand}

\nc{\al}{{\alpha}} \nc{\be}{{\beta}} \nc{\ga}{{\gamma}}
\nc{\ve}{{\varepsilon}} \nc{\Ga}{{\Gamma}} 
\nc{\La}{{\fa}}

\nc{\ad }{{\on{ad }}}

\nc{\aff}{{\on{aff}}} \nc{\Aff}{{\mathbf{Aff}}}

\nc{\der}{{\on{der}}}

\nc{\diag}{{\on{diag}}}

\nc{\Fl}{{\calF\ell}}

\nc{\Hg}{{\on{Higgs}}}

\nc{\Id}{{\on{Id}}}

\nc{\Ind}{{\on{Ind}}}

\nc{\Op}{{\on{Op}}}

\nc{\res}{{\on{res}}}

\nc{\tr}{{\on{tr}}}

\nc{\GSp}{{\on{GSp}}} \nc{\GU}{{\on{GU}}} \nc{\SL}{{\on{SL}}}
\nc{\SU}{{\on{SU}}} \nc{\SO}{{\on{SO}}}

\nc{\nh}{{\Loc_{J^p}(\tau')}}
\nc{\bnh}{{\Loc_{\breve J^p}(\tau')}}

\nc{\bU}{{\overline{U}}} 
\nc{\IC}{{\on{IC}}}

\nc{\op}{{\operatorname{P}}}

\newcommand{\br}{\begin{rouge}}
\newcommand{\er}{\end{rouge}}

\newcommand{\bb}{\begin{bluet}}
\newcommand{\eb}{\end{bluet}}

\newcommand{\triv}{{\mathbf 1}}

\newcommand{\rc}{{\mathrm c}}

\nc{\ot}{\otimes}

\nc{\oh}{{\operatorname{H}}}
\nc{\gr}{{\operatorname{gr}}}
\nc{\rk}{{\operatorname{rank}}}
\nc{\codim}{{\operatorname{codim}}}
\nc{\img}{{\operatorname{Im}}}
\nc{\Span}{{\operatorname{Span}}}
\nc{\Img}{\operatorname{Im}}
\nc{\Char}{\operatorname{Char}}

\newcommand{\beqn}{\begin{equation*}}
\newcommand{\eeqn}{\end{equation*}}

\newcommand{\beq}{\begin{equation}}
\newcommand{\eeq}{\end{equation}}

\newcommand{\bega}{\begin{gathered}}
\newcommand{\eega}{\end{gathered}}

\newcommand{\bern}{\begin{eqnarray*}}
\newcommand{\eern}{\end{eqnarray*}}

\newcommand{\ber}{\begin{eqnarray}}
\newcommand{\eer}{\end{eqnarray}}

\newcommand{\inv}{{\mathbin{/\mkern-4mu/}}}

\begin{document}
\title[Cuspidal character sheaves on graded Lie algebras]{Cuspidal character sheaves on graded exceptional Lie algebras: stable gradings}

         \author{Ting Xue}
         \address{ School of Mathematics and Statistics, University of Melbourne, VIC 3010, Australia, and Department of Mathematics and Statistics, University of Helsinki, Helsinki, 00014, Finland}
         \email{ting.xue@unimelb.edu.au}
\thanks{TX was supported in part by the ARC grant DP250100824.}

\dedicatory{dedicated to George Lusztig on the occasion of his 80th birthday}

\begin{abstract}
We determine cuspidal character sheaves explicitly for all (GIT) stably graded exceptional Lie algebras.
\end{abstract}

\maketitle
\setcounter{tocdepth}{1}
\tableofcontents

\section{Introduction}

In this paper we continue our study, initiated in~\cite{VX2}, of character sheaves for $\bZ/m\bZ$-graded Lie algebras, analogue of Lusztig's character sheaves on (ungraded) Lie algebras. To classify the character sheaves in the graded setting, a first and most important step is to give a complete classification of cuspidal character sheaves, i.e., those that do not arise from smaller groups by parabolic induction. The explicit construction of character sheaves in~\cite{VX2,VX3,X}, together with the uniform (but not explicit) construction of cuspidal character sheaves in~\cite{LTVX}, results in a classification of cuspidal character sheaves for  classical graded Lie algebras (and a few cases in exceptional types). In this paper we work with graded Lie algebras in exceptional types and determine the cuspidal character sheaves explicitly for all GIT stable gradings. More precisely, we determine the cyclotomic Hecke algebras associated to complex reflection groups whose irreducible representations give rise to cuspidal character sheaves. In~\cite{LVX} we have classified all gradings that afford cuspidal character sheaves. We will deal with the remaining (unstable) gradings in a sequel to this paper.

Throughout this paper let $G$ be a simply connected almost simple algebraic group over $\bC$ and let $\Lg=\on{Lie}G$. We fix a primitive $m$-th root of unity $\zeta_m$. Let $\theta:G\to G$ be an automorphism of order $m$. It induces a $\bZ/m\bZ$-grading $
\Lg=\oplus_{i\in\bZ/m\bZ}\Lg_i
$ on $\Lg$, where $\Lg_i\subset\Lg$ is the $\zeta_m^i$-eigenspace of $d\theta$. Recall that $G_0:=G^\theta$ is a connected reductive group by a theorem of Steinberg~\cite{St} (since $G$ is simply connected). It acts on $\Lg_i$ by adjoint action. We say that the automorphism $\theta$ on $G$ and the $\bZ/m\bZ$-grading on $\Lg$ induced by $\theta$  are GIT stable, or stable (for simplicity), if there exists a semisimple element in $\Lg_1$ with finite stabiliser in $G_0$. Such gradings have been classified in~\cite{RLYG}. We will focus on GIT stable gradings in this paper.

In analogous to Lusztig's character sheaves on $\Lg$, the character sheaves on $\Lg_1$ are defined~\cite{VX2} as the simple $G_0$-equivariant $\bC^*$-conic perverse sheaves on $\Lg_1$ with nilpotent singular support. Equivalently, they are Fourier transforms of the simple $G_0$-equivariant perverse sheaves on $\cN_{-1}:=\Lg_{-1}\cap\cN$, where $\cN$ is the nilpotent cone of $\Lg$. Here Fourier transform is the functor $\fF:\on{Perv}_{G_0}(\Lg_{-1})\to\on{Perv}_{G_0}(\Lg_{1})$ (where we identify $\Lg_1\cong\Lg_{-1}^*$). A character sheaf on $\Lg_1$ is called cuspidal if it does not arise as a direct summand (up to shifts)  via parabolic induction from $\theta$-stable Levi subgroups contained in $\theta$-stable proper parabolic subgroups of $G$.

 In~\cite{LTVX} we show that the cuspidal character sheaves for stably graded Lie algebras are precisely the full support character sheaves and they all arise as simple quotient (or sub) sheaves of $\cP_\chi^{\dag}$, where $\cP_\chi$ are  nearby cycle sheaves associated to nil-supercuspidal data $(\rc,\chi)$ and $\cP_\chi^{\dag}$ is essentially $\fF(\cP_\chi)$, the Fourier transform of $\cP_\chi$. We will follow the strategy in~\cite{GVX,VX2,VX3} to determine $\cP_\chi^\dag$, and consequently determine the cuspidal character sheaves explicitly for stably graded exceptional Lie algebras. 

In more detail, by~\cite{V}, $\Lg_1$ has a Cartan subspace $\fa$ consisting of commuting  semisimple elements such that $\Lg_1\inv G_0\cong\fa/W$, where $$W=N_{G_0}(\fa)/Z_{G_0}(\fa)$$ is the little Weyl group. The group $W$ is in general a complex reflection group.  We  refer to $\dim\fa$ as the rank of the graded Lie algebra.

Let $\Lg_1^{rs}$ denote the set of regular semisimple elements of $\Lg_1$. Let $I=Z_{G_0}(\fa)$. By our stable grading assumption, $\Lg_1^{rs}=\Lg_1\cap\Lg^{rs}$ and $I$ is a finite abelian group. Let $a_0\in\fa^{rs}\subset\Lg_1^{rs}$ and $\bar{a}_0\in\fa^{rs}/W$ the image of $a_0$ under   $\fa\to\fa/W$. We have
\beqn
\widetilde{B}_W:=\pi_1^{G_0}(\Lg_1^{rs},a_0)\cong B_W\ltimes I
\eeqn
where   $B_W:=\pi_1(\fa^{rs}/W,\bar{a}_0)$ is the braid group  (see~\cite{BMR}) associated to the complex reflection group $W$. Let $\hat I=\on{Hom}(I,\bG_m)$, the set of  characters of $I$. The braid group $B_W$ acts on $\hat I$ via the action of $W$.  
By~\cite[Theorem 2.23]{GVX}, we have that
\beqn
\cP_\chi^\dag\cong\on{IC}(\Lg_1^{rs},\cM_\chi)
\eeqn
where $\cM_\chi$ is the $G_0$-equivariant local system on $\Lg_1^{rs}$ given by a specific   representation of $\widetilde{B}_W$. An essential piece of data in describing $\cM_\chi$ is a certain cyclotomic Hecke algebra $\cH_{W_\chi^0}$ associated to a complex reflection subgroup $W_\chi^0$ of $W_\chi:=\on{Stab}_{W}\chi$. By~\cite[Theorem 3.8 and Corollary 3.10]{LTVX}, equivalent nil-supercuspidal data $(\rc,\chi)$ give rise to isomorphic $\cP_\chi^\dag$. In our stable setting, a set of representatives for the equivalence classes of nil-supercuspidal data can be taken as $\{(G_0\cdot a_0,\chi)\mid\chi\in\hat I/W\}$, where $\hat I/W$ is the set of $W$-orbits in $\hat I$. We determine $\hat I/W$ and $\cH_{W_\chi^0}$, $\chi\in\hat I$ explicitly in a case-by-case manner, see Theorem~\ref{mainthm}. As a by-product, we verify that the endoscopic interpretation of the Hecke algebras $\cH_{W_\chi^0}$ in~\cite{VX2} is valid in exceptional types, as expected, see Corollary~\ref{cor-end}.

When the rank  of the graded Lie algebra is at least $2$, as in~\cite{GVX,VX2}, to determine  $\cH_{W_\chi^0}$, we are reduced to study, for each distinguished reflection $s\in W$,  the rank one stable automorphism $\theta|_{(G_s)_{\der}}$ and the restriction of $\chi$ to $I_s=Z_{((G_s)_{\der})^\theta}(\fa)$, where $G_s=Z_G(\fa_s)$. For this we follow an approach of Reeder's~\cite{R} to describe the distinguished reflections  $s\in W$ explicitly, by considering $\theta$-orbits of roots, see Proposition~\ref{prop-distinguished}. For gradings of rank one, we follow the approach from~\cite{VX3} using invariant systems and determine the Hecke algebras via b-function calculations. We do so with the exception of the non-trivial characters in $\hat I$ for the stable $\bZ/3\bZ$-grading of $G_2$ (which is the only case with non-trivial $W$-action on $\hat I$).
In the latter case we determine the Hecke relation needed via an explicit calculation of Fourier transforms of orbital sheaves, see Proposition~\ref{prop-g2-3s}, where we determine the cuspidal orbital sheaves as well.

The paper is organised as follows. In~\S\ref{sec-rec} we collect preliminaries and fix notations. In particular, we introduce notations for the cyclotomic Hecke algebras that appear in our main theorem. In~\S\ref{sec-strategy} we recall the construction from~\cite{GVX} and describe our general strategy to determine the Hecke algebras $\cH_{W_\chi^0}$.  In~\S\ref{sec-main} we state our main results. 
In~\S\ref{sec-reduction} we describe the distinguished reflections in $W$,  determine the rank one stable automorphisms $\theta|_{(G_s)_{\der}}$, and calculate the groups $I_s$. In~\S\ref{sec-rankgt2} we make the results in~\S\ref{sec-reduction} more explicit in each case and use them to  prove the main theorem for stable gradings of rank at least 2. In~\S\ref{sec-rkone} we prove the main theorem for stable gradings of rank one by calculating the relevant b-functions in each case. As explained above, additional argument are provided in the case of non-trivial characters for the stable $\bZ/3\bZ$-grading of $G_2$.

\noindent{\bf Acknowledgements.} I would like to thank George Lusztig for his constant support and encouragement and for the math conversations. I would also like to thank my coauthors, Misha Grinberg, Wille Liu, Cheng-Chiang Tsai and Kari Vilonen  for their collaboration on this project, which originally grew out of a suggestion of Lusztig's. 

\section{Recollections and notations}\label{sec-rec}

From now on let $G$ be a simply connected almost simple exceptional group over $\bC$ and $\theta:G\to G$ a (GIT) stable automorphism of order $m$. Let $n=\on{rank }G$. 

 Let $\zeta_\ell:=\on{exp}(2\pi\mathbf{i}/\ell)$, a fixed primitive $\ell$-th root of unity. Let $\Phi_\ell$ denote the $\ell$-th cyclotomic polynomial and $\mu_\ell$ the cyclic group of order $\ell$.

\subsection{Stable gradings and $\theta$-orbits of roots}\label{ssec-thetaorbits}
Following~\cite{RLYG} we will use the following realisation of $\theta$  (see also~\cite[(2.2)]{VX2}). We fix a Cartan subspace $\fa\subset\Lg_1$ (recall that all Cartan subspaces of $\Lg_1$ are conjugate by $G_0$).  Let $T=Z_G(\fa)$ (which is a split $\theta$-stable maximal torus of $G$) and let $W_G=N_G(T)/T$ be the Weyl group of $G$. We assume that  $\theta=\on{Int}(n_w)\circ\vartheta$ for some lift $n_w\in N_G(T)$ of $w\in W_G$, where $\vartheta\in\on{Out}(G)$ is a pinned automorphism with respect to $T$. 

Suppose that $G$ is of type $X_n$ and that $\theta=\on{Int}(n_w)\circ\vartheta$  with $\on{ord}\vartheta=e$ and $\on{ord}\theta=m$. We will say that $(G,\theta)$ is of type $(^eX_n,m_\rs)$.

Let $\Lt=\on{Lie}T$.  Let $R\subset\Lt^*$ be the set of roots and $\check R=\{\check\alpha\mid\alpha\in R\}$ the set of coroots. Let $\Delta=\{\alpha_i\}$ be a set of simple roots and $\check\Delta=\{\check\alpha_i\}$ the set of simple coroots. Let $\Lg_\alpha\subset \Lg$ denote the root subspace corresponding to $\alpha\in R$. Since $T$ is $\theta$-stable, $\theta$ induces a bijection $R\to R,\alpha\mapsto\theta\alpha$ such that $\theta(\Lg_\alpha)=\Lg_{\theta\alpha}$. We have $\theta\alpha=w(\vartheta(\alpha))$ for $\alpha\in R$.  Since $\theta$ is stable, $w\vartheta$ is an elliptic and regular element of $W_G\vartheta$ by~\cite{RLYG}. It follows that each $\theta$-orbit in $R$ consists of $m$ roots (see~\cite[Proposition 4.10 and Lemma 6.8]{S}). Moreover, we have (\cite{V})
\beq\label{eqn-weyl}
W=W_G^\theta=\{u\in W_G\mid u\circ\theta|_{T}=\theta|_{T}u\}\,.
\eeq
We choose a set of Chevalley basis of $\Lg$ consisting of root vectors $X_\alpha\in\Lg_\alpha$, $\alpha\in R$, and $h_{\alpha_i}\in\Lt$, $\alpha_i\in\Delta$ (see for example~\cite[\S25]{H}). In particular, we have
$[X_{\alpha_i},X_{-\alpha_i}]=h_{\alpha_i}$ and $[X_\alpha,X_{-\alpha}]:=h_\alpha\in{\bZ}\text{-span of }h_{\alpha_i}\,.$

\subsection{Root systems and Coxeter elements}\label{ssec-dynkin} In this subsection we fix the labelling of simple roots in each exceptional type. We also fix a Coxeter element $w_h\in W_G$ which we use later. We write a positive root $\alpha=\sum_{i=1}^na_i\alpha_i$ as $a_1a_2\cdots a_n$ and a negative root $\alpha=-\sum_{i=1}^na_i\alpha_i$ as $-a_1a_2\cdots a_n$. Let $s_\alpha\in W_G$ denote the reflection corresponding to $\alpha\in R$ and let $s_i=s_{\alpha_i}$.

\subsubsection{Type $G_2$}  The extended Dykin diagram is $\xymatrix@C=1em{{\substack{\alpha\\\bullet\\\,}}&{{\substack{\beta\\\bullet\\\,}}}\ar@{-}[r]\ar@3{->}[l]&{{\substack{-\alpha_0\\\circ\\\,}}}}$ where $\alpha_0=3\alpha+2\beta$.

\subsubsection{Type $F_4$} 
The extended Dynkin diagram is
$\xymatrix{{\substack{-\alpha_0\\\circ\\\,}}\ar@{-}[r]&{\substack{\alpha_1\\\bullet\\\,}}\ar@{-}[r]&{\substack{\alpha_2\\\bullet\\\,}}\ar@2{=>}[r]&{\substack{\alpha_3\\\bullet\\\,}}\ar@{-}[r]&{\substack{\alpha_4\\\bullet\\\,}}}
$
where $\alpha_0=2342$. Let 
$
w_h:=s_{\alpha_1}s_{\alpha_2}s_{\alpha_3}s_{\alpha_4}
.$ Then
$
(w_h\alpha_i)_{i=1,\ldots,4}=0100,1120,0001,-1111.
$

\subsubsection{Type $E_6$} 
The extended Dynkin diagram is 
$
\xymatrix@R-1.5pc{{\substack{\alpha_1\\\bullet}}\ar@{-}[r]&{\substack{\alpha_3\\\bullet}}\ar@{-}[r]&{\substack{\alpha_4\\\bullet}}\ar@{-}[r]\ar@{-}[d]&{\substack{\alpha_5\\\bullet}}\ar@{-}[r]&{\substack{\alpha_6\\\bullet}}\\&&{\substack{\bullet\\\alpha_2}}\ar@{-}[d]&&\\&&{\substack{\circ\\-\alpha_0}}&&}
$
where $\alpha_0=\alpha_1+2\alpha_2+2\alpha_3+3\alpha_4+2\alpha_5+\alpha_6$.  
The center of $G$ is
$
Z_G=\langle\check\alpha_1(\xi_3)\check\alpha_3(\xi_3^2)\check\alpha_5(\xi_3)\check\alpha_6(\xi_3^2)\rangle\cong\mu_3.
$

Let  $w_h:=s_{\alpha_1}s_{\alpha_4}s_{\alpha_6}s_{\alpha_3}s_{\alpha_5}s_{\alpha_2}$.
Then
{\small$$(w_h\alpha_i)_{i=1,\ldots,6}= 001100,\,
-010100,\,
-101100,111211,\,
 -000111,\,
000110.
$$}

\subsubsection{Type $E_7$} 
The extended Dynkin diagram is 
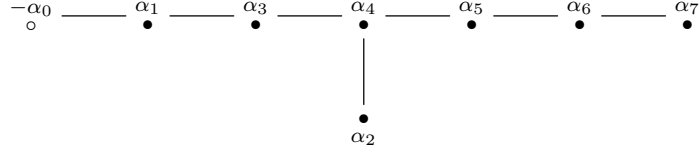
\beqn
\xymatrix{{\substack{-\alpha_0\\\circ}}\ar@{-}[r]&{\substack{\alpha_1\\\bullet}}\ar@{-}[r]&{\substack{\alpha_3\\\bullet}}\ar@{-}[r]&{\substack{\alpha_4\\\bullet}}\ar@{-}[r]\ar@{-}[d]&{\substack{\alpha_5\\\bullet}}\ar@{-}[r]&{\substack{\alpha_6\\\bullet}}\ar@{-}[r]&{\substack{\alpha_7\\\bullet}}\\&&&{\substack{\bullet\\\alpha_2}}&&}
\eeqn
where $\alpha_0=2234321$. 
The center of $E_7$ is
$
Z_G=\langle\check\alpha_2(-1)\check\alpha_5(-1)\check\alpha_7(-1)\rangle\cong\mu_2.
$

Let $w_h=s_1  s_4  s_3  s_5 s_7  s_6  s_2$. Then {\small$$(w_h\alpha_i)_{i=1,\ldots,7}=0011000,-0101000,-1011000,1112110,  -0001110,0001111,-0000011.$$}

\subsubsection{Type $E_8$} 
The extended Dynkin diagram is 
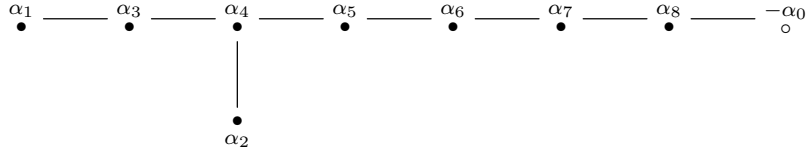
\beqn
\xymatrix{{\substack{\alpha_1\\\bullet}}\ar@{-}[r]&{\substack{\alpha_3\\\bullet}}\ar@{-}[r]&{\substack{\alpha_4\\\bullet}}\ar@{-}[r]\ar@{-}[d]&{\substack{\alpha_5\\\bullet}}\ar@{-}[r]&{\substack{\alpha_6\\\bullet}}\ar@{-}[r]&{\substack{\alpha_7\\\bullet}}\ar@{-}[r]&{\substack{\alpha_8\\\bullet}}\ar@{-}[r]&{\substack{-\alpha_0\\\circ}}\\&&{\substack{\bullet\\\alpha_2}}&&}
\eeqn
where $\alpha_0=23465432$. We let $\alpha_1=\frac{1}{2}(e_1+e_8-\sum_{i=2}^7e_i),\,\alpha_2=e_1+e_2,\,\alpha_i=e_{i-1}-e_{i-2},\,i=3,\ldots,8.$ 

Let $w_h=s_1s_4s_6s_8s_3s_2s_5s_7$. Then {\tiny$$(w_h\alpha_i)_{i=1,\ldots,8}=00110000,-01010000,-10110000,11121100,-00011100,00011111,-00000111,00000110.$$}

\subsection{Cyclotomic Hecke algebras associated to complex reflection groups}
In this subsection we fix notations for cyclotomic Hecke algebras associated to complex reflection groups, see~\cite[\S4]{BMR} (see also~\cite[\S2.2]{VX2}).

Let $C\subset GL(V)$ be a finite complex reflection group (see~\cite{ST,Co}). Recall that the reflection hyperplanes $H$ of $(C,V)$ are in bijection with distinguished reflections $s_H\in C$. In particular, the stabiliser group $C_H$ of $H$ in $C$ is generated by $s_H$ and $\det s_H=\zeta_{|C_H|}$.

The cyclotomic Hecke algebras associated to $C$ (see~\cite[Definition 4.21]{BMR}) are certain quotient algebras of the group algebra $\bC[B_C]$ of  the associated braid group $B_C$.   
In view of~\cite[Proposition 4.22]{BMR} and~\cite{E} (and the references therein), to specify a Hecke algebra associated to $C$, it suffices to specify the Hecke relations associated to braid generators $\sigma_s$ of distinguished reflections $s$.

Let us write 
$\cH_{C,p(z)}$ for the Hecke algebra with Hecke relations $p(\sigma_s)=0$. When the distinguished reflections form 2 conjugacy classes in $C$, we write $\cH_{C,p_1(z),p_2(z)}$ for the Hecke algebra with Hecke relations $p_1(\sigma_s)=0$, $p_2(\sigma_t)=0$ for  distinguished reflections $s$ in one  conjugacy class and $t$ in the other. For example, $\cH_{\mu_n,p(z)}\cong\bC[z]/(p(z))$.

For Weyl groups, we write $S_n$ for symmetric groups, $W_n$, resp. $W_n'$, for  Weyl groups of type $B_n/C_n$, resp. $D_n$, and $W_{G_2}, W_{F_4}, W_{E_n}$ for Weyl groups of type $G_2,F_4,E_n$. For a Weyl group $W$, we will also write
$
\cH_{W,-1}=\cH_{W,\Phi_1^2}.
$
For the infinite families $G_{m,1,r}:=S_r\ltimes(\mu_m)^r$ and $G_{2m_0,2,r}$ (an index 2 subgroup of $G_{2m_0,1,r}$) we will also write, as in~\cite[\S7.1]{VX2}, $$\cH^{\ell}G_{m,1,r}=\cH_{G_{m,1,r},\Phi_1^2,\Phi_1^\ell\Phi_2^{m-2\ell}},\qquad \cH^{m_0}G_{2m_0,2,r}=\cH_{G_{2m_0,1,r},\Phi_1^2,\Phi_1^{m_0}}.$$

\section{General strategy}\label{sec-strategy}

In this section we describe our general strategy. As explained in the introduction, to determine the cuspidal character sheaves, we are reduced to determine the local systems $\cM_\chi$ that arises from $\cP_\chi^\dag\cong\on{IC}(\Lg_1^{rs},\cM_\chi)$, and the $W$-orbits in $\hat I$, where $I=Z_{G_0}(\fa)$.

Let us first briefly recall the local systems $\cM_\chi$. We refer the readers to~\cite{GVX} and~\cite[\S2]{VX2} for details. Let $a_0\in\fa^{rs}\subset\Lg_1^{rs}=\Lg_1\cap\Lg^{rs}$. We have
\beqn
\widetilde{B}_W:=\pi_1^{G_0}(\Lg_1^{rs},a_0)\cong B_W\ltimes I
\eeqn
where  $B_W$ is the braid group associated to the complex reflection group $W=N_{G_0}(\fa)/Z_{G_0}(\fa)$. This arises from the following split short exact sequence
\beqn
1\to I\to\widetilde{B}_W\xrightarrow{\tilde q} B_W\to 1\,
\eeqn
where we have chosen a Kostant slice as in~\cite[\S2]{VX2} to split the sequence.

Let $\chi\in\hat I$. We associate to $\chi$ the following subgroups of $W$ (see~\cite[Definition 2.3 and (5.3)]{VX2})
\begin{align*}
&W_{\chi}=\on{Stab}_W(\chi)\\& W_\chi^0\subset W_\chi\text{ the complex reflection subgroup generated by the reflections in $W_\chi$}\\
&W_\chi^{\on{en}}\subset W_\chi^0\text{ the endoscopic subgroup}\,.
\end{align*} 
Here the $W_\chi^{\on{en}}$ is by construction the little Weyl group of a stable order $m$ automorphism on the group $\check{G}(\chi)^0$, where we regard $\chi$ as an element in the dual group $\check G$ (see~\cite[(5.1)]{VX2}) and $\check{G}(\chi)^0$ is the identity component of the fixed point group $(\check G)^{\chi}$.

Let $p:B_W\to W$ be the map given by mapping $\fa^{rs}$ to a point (as in~\cite{GVX1,GVX}). Let
\begin{align*}
&B_W^{\chi}=p^{-1}(W_\chi),\ B_W^{\chi,0}=p^{-1}(W_\chi^0),\quad\widetilde{B}_W^\chi=\tilde{q}^{-1}(B_W^\chi),\ \widetilde{B}_W^{\chi,0}=\tilde{q}^{-1}(B_W^{\chi,0})\,. 
\end{align*}
Applying~\cite[Theorem 2.23]{GVX} (cf.~\cite[Theorem 3.8]{LTVX}), we have that
\beqn
\cP_\chi^\dag\cong\on{IC}(\Lg_1^{rs},\cM_\chi)
\eeqn
where $\cM_\chi$ is the $G_0$-equivariant local system corresponding to the following  representation of $\widetilde{B}_W$:
\beq\label{eqn-nearby}
\cM_\chi=\left(\bC[\widetilde{B}_W]\otimes_{\bC[\widetilde B_{W}^{\chi,0}]}(\bC_\chi\otimes\bC_{\rho}\otimes\cH_{W_\chi^0})\right)\otimes\bC_{\tau}
\eeq
and $\cH_{W_\chi^0}$ is a Hecke algebra associated to the complex reflection group $W_\chi^0$. In particular $\dim\cH_{W_\chi^0}=|W_\chi^0|$. We refer the readers to {\em loc.cit.} for the precise definition of $\cM_\chi$. We will check (in a case-by-case manner) that in our stable grading setting, the characters $\rho$ and $\tau$ are trivial (see~Lemma~\ref{lem-tau} and Lemma~\ref{lem-rho}). Thus to determine $\cP_\chi^\dag$, we are reduced to determine $\cH_{W_\chi^0}$.

\subsection{Stable gradings of rank 1}Suppose that $\dim\fa=1$. In this case, we have $W=\langle s\rangle$ is a cyclic group and $B_W\cong\bZ$. Let $\chi\in\hat I$. Suppose that $W_\chi=W_\chi^0=\langle s^\ell\rangle$, where $\ell||W|$. We have
\beq\label{eqn-poly-rank1}
\cH_{W_\chi^0}=\cH_{W_\chi^0,R_\chi(z)}\cong\bC[z]/(R_\chi(z))
\eeq
where $R_\chi(z)\in \bC[z]$ is a monic polynomial of degree $|W|/\ell$.

The stable gradings of rank 1, where $m$ equals the (twisted) Coxeter number, have been studied in~\cite{VX2}. We will follow the strategy in~\cite[\S4]{VX3} (in particular, Theorem 4.3 of {\em loc.cit.}) to calculate the polynomials $R_\chi(z)$ for the remaining stable gradings of rank 1 (under the assumption that $W_\chi=W$ for all $\chi\in\hat I$). The last assumption holds for all cases except for the non-trivial characters of $I$ in type $(G_2,3_\rs)$.

\subsection{Stable gradings of rank at least $2$.}\label{ssec-str-red}

Suppose now $\dim\fa\geq 2$. Recall~\cite{GVX} (see also~\cite[\S2.5]{VX2}) that to determine $\cH_{W_\chi^0}$, we are reduced to the case of (semisimple) rank one as follows. 
For each distinguished reflection $s\in W$, let $\fa_s\subset\fa$ denote the hyperplane fixed by $s$. Let 
$$G_s=Z_G(\fa_s)\text{ and $(G_s)_{\on{der}}\subset G_s$  the derived subgroup.}$$
 Then $(G_s)_{\on{der}}$ is simply connected (see~\cite[\S8.4.6 (6)]{S2}). Moreover, $\theta|_{(G_s)_{\on{der}}}$ induces a stable grading of rank one on $(\Lg_s)_\der:=\on{Lie}(G_s)_{\on{der}}$ with little Weyl group $W_s=\langle s\rangle$. Let  
\begin{align*}
&G_{s,0}:=((G_s)_{\on{der}})^\theta,\ \Lg_{s,1}:=(\Lg_s)_\der\cap\Lg_1\\
&n_s=|W_s|,\ W_{s,\chi}:=W_s\cap W_\chi=\langle s^{e_s}\rangle,\ e_s=|W_s|/|W_{s,\chi}|\,.
 \end{align*}
 Let $\fa_s^\perp\subset \Lg_{s,1}$ denote the Cartan subspace of the induced grading on $(\Lg_s)_\der$ such that $\fa_s\oplus\fa_s^\perp=\fa$. Let
\beq\label{eqn-local}
 I_s:=Z_{G_{s,0}}(\fa_s^\perp)=Z_{G_{s,0}}(\fa)\subset I,\ \chi_s=\chi|_{I_s}\text{ and }W_{s,\chi_s}=\on{Stab}_{W_s}(\chi_s)\,.
 \eeq
In~\cite[\S2.7]{GVX} we associate a monic polynomial $R_{\chi,s}\in\bC[z]$ of degree $n_s$ to the rank one polar representation $(G_{s,0},\Lg_{s,1})$ and the character $\chi_s\in\hat I_s$. It is shown in {\em loc.cit.} that there exists a polynomial $\bar{R}_{\chi,s}\in\bC[z]$ such that 
\beqn
R_{\chi,s}(z)=\bar{R}_{\chi,s}(z^{e_s}).
\eeqn
The Hecke relations defining the Hecke algebra $\cH_{W_\chi^0}$ are  given by the polynomials $\bar{R}_{\chi,s}$ (since $s^{e_s}$ are the distinguished reflections in $W_\chi^0$). In fact, one can check that we have $R_{\chi,s}(z)=R_{\chi_s}(z^{\ell_s})$ where $\ell_s=|W_s|/|W_{s,\chi_s}|$ and $R_{\chi_s}(z)$ is the monic polynomial  defined as in~\eqref{eqn-poly-rank1}   for the rank one stably graded Lie algebra $(\Lg_s)_\der$ and $\chi_s=\chi|_{I_s}$.  

Thus we are reduced to determine, for each distinguished reflection $s\in W$, the characters $\chi_s\in\hat{I_s}$, and the monic polynomials $R_{\chi_s}$ associated to the rank one stable automorphisms  $\theta|_{(G_s)_{\on{der}}}$. 

We will show (see Proposition~\ref{prop-rankone}) that the rank one stable automorphisms  $\theta|_{(G_s)_{\on{der}}}$ can be identified as one of the following types (here we include the data for $I_s$ and $R_{\chi_s}$ from~\cite[\S6]{VX2} for later use)
\begin{longtable}[c]{p{5.5cm}p{2.5cm}p{6cm}}
\hline\hline
Type&$I_s$&$R_{\chi_s}$
\\\hline\hline
$(A_{n-1},n_\rs),\ n=2,3,4,5$&$\mu_{n}$&$(z^d-1)^{{n}/{d}}$ if $\chi_s$ is of order $d$\\
\hline
$(^2A_{2n},(4n+2)_\rs)$, $n=1,2$&$1$&$(z-1)^{n+1}(z+1)^n$\\
\hline
$(B_2,4_\rs)$&$\mu_2$&$(z-1)^3(z+1)$ if $\chi_s$ trivial\\
&& $(z^2-1)^2$ if $\chi_s$ non-trivial\\
\hline
$(^2D_4,8_\rs)$&$\mu_2$&$(z-1)^4$ if $\chi_s$ trivial\\
&& $(z^2-1)^2$ if $\chi_s$ non-trivial\\
\hline
$(^3D_4,12_\rs)$&$1$&$(z-1)^3(z+1)$\\
\hline\hline
\caption{Reduction to rank one}
\label{table-rd1}
\end{longtable}

\subsubsection{Stable $\bZ/2\bZ$-gradings}\label{ssec-stable2} Suppose $\theta$ is stable of order $2$. Then $\theta$ is an involution and $(G,G_0)$ is a split symmetric pair. 
In this case, we have $W=W_G$ and $\dim\fa=n=\on{rank}G$. So the reflections in $W$ are $s_\alpha$, $\alpha\in R$.  By~\cite{GVX1,VX}), we have
\begin{align*}
&I=\langle\check\alpha(-1)\mid\alpha\in R\rangle\cong\mu_2^n\qquad W_\chi^0=\langle s_\alpha\mid \chi(\check\alpha(-1))=1\rangle\qquad\text{and }\quad\cH_{W_\chi^0}=\cH_{W_\chi^0,-1}.
\end{align*}
By~\cite[\S3.3]{VX}, the groups $W_\chi$ (resp. $W_\chi^0$) can be  identified with the Weyl group of the fixed point group $\check{G}^\chi$ (resp. of the identity component $(\check{G}^\chi)^0)$, where $\check{G}$ is the dual group of $G$ and $\chi$ is regarded as an order $2$ element in $\check G$. This allows us to obtain the results for the gradings $2_\rs$ from results for symmetric pairs (see for example~\cite{S3}). In particular, we have $W_\chi/W_\chi^0\cong\check{G}^\chi/(\check{G}^\chi)^0$. This implies that $W_\chi=W_\chi^0$ except when $G=E_7$, and $W_\chi^0\cong S_8$ or $W_{E_6}$. In the latter cases we have
$W_\chi/W_\chi^0\cong\mu_2$. 
\begin{remark}
For $\bZ/2\bZ$-stable gradings, $W_\chi^0$ in~\cite{GVX1,VX} coincides with $W_\chi^{\en}$ in~\cite{VX2}.
\end{remark}

\section{Main results}\label{sec-main}
In this section we state our main results, that is, we determine $\hat I/W$, $\cP_\chi^\dag$, $\chi\in\hat I/W$ and thus the cuspidal character sheaves for stably graded exceptional Lie algebras (except that the answer in one case for $(G_2,3_\rs)$ is still conjectural). 
We write $\chi_0$ for the trivial character of $I$. We use notations from~\S\ref{sec-rec} and~\S\ref{sec-strategy}.

 \subsection{Nearby cycle sheaves} 
 Let $a_0\in\fa^{rs}$, $I=Z_{G_0}(a_0)=Z_{G_0}(\fa)$, and $\chi\in\hat I$. 
 Let $\cP_\chi$ be the nearby cycle sheaf associated to the nil-supercuspidal datum $(G_0.a_0,\chi)$, see~\cite[\S3]{LTVX}.
 \begin{theorem}\label{mainthm}Let $(G,\theta)$ be  stable, where $G$ is a simply connected almost simple exceptional group. 
 We have
 \beqn
 \{\cP_\chi^\dag\mid\chi\in\hat I/W\}=\{\on{IC}(\Lg_1^{rs},\cM_\chi)\mid\chi\in\hat I/W\}
 \eeqn
 where $$\cM_\chi=\bC[\widetilde{B}_W]\otimes_{\bC[\widetilde B_{W}^{\chi,0}]}(\bC_\chi\otimes\cH_{W_\chi^0}).$$
 The representatives $\chi\in \hat I/W$, the groups $W_\chi$ and the Hecke algebras $\cH_{W_\chi^0}$ are given in the following tables

\FloatBarrier
\begin{longtable}[c]{p{1.8cm}p{1cm}p{1cm}p{2cm}p{4cm}p{2cm}p{1.1cm}}
\hline
\hline
$(G,\theta)$&$I$&$\chi$&$W_{\chi}$&$\cH_{W_{\chi}^0}=\cH_{W_\chi}$&$\check{G}(\chi)^0$\\\hline
\hline
$(G_2,2_\rs)$&$\mu_2^2$&$\chi_0$&$W_{G_2}$&$\cH_{W_{G_2},-1}$&\\
&&$\chi_1$&$S_2\times S_2$&$\cH_{S_2,-1}\otimes \cH_{S_2,-1}$&$A_1^2$\\\hline
$(G_2,3_\rs)$&$\mu_3$&$\chi_0$&$\mu_6$&$\cH_{\mu_6,\Phi_1^3\Phi_2\Phi_6}$\\
&&$\chi_1$&$\mu_3$&$\cH_{\mu_3,\Phi_1^3}$&$A_2$\\\hline
$(G_2,6_\rs)$&$1$&$\chi_0$&$\mu_6$&$\cH_{\mu_6,\Phi_1^3\Phi_2\Phi_3}$\\\hline
\hline
$(^3D_4,3_\rs)$&$\mu_3^2$&$\chi_0$&$G_4$&$\cH_{G_4,\Phi_1^3}$&\\
&&$\chi_1$&$\mu_3$&$\cH_{\mu_3,\Phi_1^3}$&$A_2$\\\hline
$(^3D_4,6_\rs)$&$1$&$\chi_0$&$G_4$&$\cH_{G_4,\Phi_1^2\Phi_2}$\\\hline
$(^3D_4,12_\rs)$&$1$&$\chi_0$&$\mu_{4}$&$\cH_{\mu_{4},\Phi_1^3\Phi_2}$\\\hline
\hline
$(F_4,2_\rs)$&$\mu_2^4$&$\chi_0$&$W_{F_4}$&$\cH_{W_{F_4},-1}$&\\
&&$\chi_1$&$W_3\times S_2$&$\cH_{W_3,-1}\otimes\cH_{S_2,-1}$&$C_3\times A_1$\\
&&$\chi_2$&$W_4$&$\cH_{W_4,-1}$&$B_4$\\\hline
$(F_4,3_\rs)$&$\mu_3^2$&$\chi_0$&$G_5$&$\cH_{G_5,\Phi_1^3}$&\\
&&$\chi_1$&$\mu_3\times\mu_3$&$\cH_{\mu_3^2,\Phi_1^3}$&$A_2^2$\\\hline
$(F_4,4_\rs)$&$\mu_2^2$&$\chi_0$&$G_8$&$\cH_{G_8,\Phi_1^3\Phi_2}$&\\
&&$\chi_1$&$G_{4,1,2}$&$\cH^{3}G_{4,1,2}$&$B_4$\\\hline
$(F_4,6_\rs)$&$1$&$\chi_0$&$G_5$&$\cH_{G_5,\Phi_1^2\Phi_2}$\\\hline
$(F_4,8_\rs)$&$\mu_2$&$\chi_0$&$\mu_8$&$\cH_{\mu_8,\Phi_1^5\Phi_2\Phi_4}$\\
&&$\chi_1$&$\mu_8$&$\cH_{\mu_8,\Phi_1^5\Phi_2^3}$&$B_4$\\\hline
$(F_4,12_\rs)$&$1$&$\chi_0$&$\mu_{12}$&$\cH_{\mu_{12},\Phi_1^5\Phi_2^3\Phi_3\Phi_4}$\\\hline
\hline
$(^2E_6,2_\rs)$&$\mu_2^6$&$\chi_0$&$W_{E_6}$&$\cH_{W_{E_6},-1}$&\\
&&$\chi_1$&$S_6\times S_2$&$\cH_{S_6,-1}\otimes\cH_{S_2,-1}$&$A_5\times A_1$\\
&&$\chi_2$&$W_5'$&$\cH_{W_5',-1}$&$D_5$\\\hline
$(^2E_6,4_\rs)$&$\mu_4^2$&$\chi_0$&$G_8$&$\cH_{G_8,\Phi_1^4}$&\\
&&$\chi_1$&$G_{4,1,2}$&$\cH^{4}G_{4,1,2}$&$D_5$\\
&&$\chi_2$&$\mu_4\times\mu_2$&$\cH_{\mu_4,\Phi_1^4}\otimes\cH_{\mu_2,\Phi_1^2}$&$A_3\times A_1^2$\\\hline
$(^2E_6,6_\rs)$&$1$&$\chi_0$&$G_{25}$&$\cH_{G_{25},\Phi_1^2\Phi_2}$\\\hline
$(^2E_6,12_\rs)$&$1$&$\chi_0$&$\mu_{12}$&$\cH_{\mu_{12},\Phi_1^5\Phi_2^5\Phi_6}$\\\hline
$(^2E_6,18_\rs)$&$1$&$\chi_0$&$\mu_{9}$&$\cH_{\mu_{9},\Phi_1^5\Phi_2^2\Phi_3}$\\\hline
\hline
$(E_8,2_\rs)$&$\mu_2^8$&$\chi_0$&$W_{E_8}$&$\cH_{W_{E_8},-1}$\\
&&$\chi_1$&$W_8'$&$\cH_{W_8',-1}$&$D_8$\\
&&$\chi_2$&$W_{E_7}\times S_2$&$\cH_{W_{E_7},-1}\otimes\cH_{S_2,-1}$&$E_7\times A_1$\\\hline
$(E_8,3_\rs)$&$\mu_3^4$&$\chi_0$&$G_{32}$&$\cH_{G_{32},\Phi_1^3}$\\
&&$\chi_1$&$G_{25}\times\mu_3$&$\cH_{G_{25},\Phi_1^3}\otimes\cH_{\mu_3,\Phi_1^3}$&$E_6\times A_2$\\\hline
$(E_8,4_\rs)$&$\mu_2^4$&$\chi_0$&$G_{31}$&$\cH_{G_{31},\Phi_1^2}$\\
&&$\chi_1$&$G_{4,2,4}$&$\cH^2G_{4,2,4}$&$D_8$\\\hline
$(E_8,5_\rs)$&$\mu_5^2$&$\chi_0$&$G_{16}$&$\cH_{G_{16},\Phi_1^5}$\\
&&$\chi_1$&$\mu_5^2$&$\cH_{\mu_5^2,\Phi_1^5}$&$A_4^2$\\\hline
$(E_8,6_\rs)$&$1$&$\chi_0$&$G_{32}$&$\cH_{G_{32},\Phi_1^2\Phi_2}$\\[5pt]\hline
$(E_8,8_\rs)$&$\mu_2^2$&$\chi_0$&$G_{9}$&$\cH_{G_{9},\Phi_1^4,\Phi_1^2}$\\
&&$\chi_1$&$G_{8,2,2}$&$\cH^4G_{8,2,2}$&$D_8$\\\hline
$(E_8,10_\rs)$&$1$&$\chi_0$&$G_{16}$&$\cH_{G_{16},\Phi_1^3\Phi_2^2}$\\[5pt]\hline
$(E_8,12_\rs)$&$1$&$\chi_0$&$G_{10}$&$\cH_{G_{10},\Phi_1^2\Phi_2,\Phi_1^3\Phi_2}$\\[5pt]\hline
$(E_8,15_\rs)$&$1$&$\chi_0$&$\mu_{30}$&$\cH_{\mu_{30},\Phi_1^9\Phi_2^9\Phi_3\Phi_5\Phi_6^3}$\\[5pt]\hline
$(E_8,20_\rs)$&$1$&$\chi_0$&$\mu_{20}$&$\cH_{\mu_{20},\Phi_1^9\Phi_2^5\Phi_4\Phi_5}$\\[5pt]\hline
$(E_8,24_\rs)$&$1$&$\chi_0$&$\mu_{24}$&$\cH_{\mu_{24},\Phi_1^9\Phi_2^5\Phi_3^3\Phi_4\Phi_6}$\\[5pt]\hline
$(E_8,30_\rs)$&$1$&$\chi_0$&$\mu_{30}$&$\cH_{\mu_{30},\Phi_1^9\Phi_2^5\Phi_3^3\Phi_4^2\Phi_5\Phi_6}$\\[5pt]\hline\hline
\caption{Types $G_2,{}^3D_4, F_4,{}^2E_6,E_8$}
\label{tab-equal}
\end{longtable}

\FloatBarrier
\begin{longtable}[c]{p{1.5cm}p{1cm}p{1.7cm}p{2.2cm}p{3.5cm}p{1.5cm}p{1.7cm}}
\hline
\hline
$(G,\theta)$&$I$&$\chi$&$W_{\chi}$&$\cH_{W_{\chi}^0}$&$\check{G}(\chi)^0$&$\cH_{W_{\chi}^{\on{en}}}$\\\hline
\hline
$(E_6,3_\rs)$&$\mu_3^3$&$\chi_0$&$G_{25}$&$\cH_{G_{25},\Phi_1^3}$\\
&&$\chi_1,\chi_3$&$\mu_3\ltimes G_{4}$&$\cH_{G_4,\Phi_1^3}$&$^3D_4$&\\
&&$\chi_2$&$\mu_3\ltimes (\mu_3^3)$&$\cH_{\mu_3^3,\Phi_1^3}$&$A_2^3$&\\\hline
$(E_6,6_\rs)$&$Z_G$&$\chi_0$&$G_5$&$\cH_{G_5,\Phi_1^2\Phi_2,\Phi_1^3}$\\
&&$\chi_1,\chi_2$&$G_5$&$\cH_{G_5,\Phi_1^2\Phi_2,\Phi_1\Phi_3}$&$^3D_4$&$\cH_{G_4,\Phi_1^2\Phi_2}$\\\hline
$(E_6,9_\rs)$&$Z_G$&$\chi_0$&$\mu_{9}$&$\cH_{\mu_{9},\Phi_1^7\Phi_3}$\\
&&$\chi_1,\chi_2$&$\mu_{9}$&$\cH_{\mu_{9},\Phi_1^3\Phi_3^3}$&$A_2^3$&$\cH_{\mu_3,\Phi_1^3}$\\\hline
$(E_6,12_\rs)$&$Z_G$&$\chi_0$&$\mu_{12}$&$\cH_{\mu_{12},\Phi_1^7\Phi_2^3\Phi_3}$\\
&&$\chi_1,\chi_2$&$\mu_{12}$&$\cH_{\mu_{12},\Phi_1^3\Phi_2\Phi_3^3\Phi_6}$&$^3D_4$&$\cH_{\mu_4,\Phi_1^3\Phi_2}$\\\hline
\hline
$(E_7,2_\rs)$&$\mu_2^7$&$\chi_0$&$W_{E_7}$&$\cH_{W_{E_7},-1}$\\
&&$\chi_1$&$\langle\tau_1\rangle\ltimes S_8$&$\cH_{S_8,-1}$&$^2A_7$\\
&&$\chi_2$&$W_6'\times S_2$&$\cH_{W_6',-1}\otimes\cH_{S_2,-1}$&$^2D_6\times A_1$\\
&&$\chi_3$&$\langle\tau_3\rangle\ltimes W_{E_6}$&$\cH_{W_{E_6},-1}$&$^2E_6$\\\hline
$(E_7,6_\rs)$&$Z_G$&$\chi_0$&$G_{26}$&$\cH_{G_{26},\Phi_1^2,\Phi_1^2\Phi_2}$\\
&&$\chi_1$&$G_{26}$&$\cH_{G_{26},\Phi_1\Phi_2,\Phi_1^2\Phi_2}$&$^2E_6$&$\cH_{G_{25},\Phi_1^2\Phi_2}$\\\hline
$(E_7,14_\rs)$&$Z_G$&$\chi_0$&$\mu_{14}$&$\cH_{\mu_{14},\Phi_1^8\Phi_2^4\Phi_4}$\\
&&$\chi_1$&$\mu_{14}$&$\cH_{\mu_{14},\Phi_1^5\Phi_2^5\Phi_4^2}$&$^2A_7$&$\cH_{\mu_{7},\Phi_1^5\Phi_2}$\\\hline
$(E_7,18_\rs)$&$Z_G$&$\chi_0$&$\mu_{18}$&$\cH_{\mu_{18},\Phi_1^8\Phi_2^4\Phi_3^2\Phi_4}$\\
&&$\chi_1$&$\mu_{18}$&$\cH_{\mu_{18},\Phi_1^5\Phi_2^5\Phi_4^2\Phi_3\Phi_6}$&$^2E_6$&$\cH_{\mu_{9},\Phi_1^5\Phi_2^2\Phi_3}$\\\hline\hline
\caption{Types $E_6, E_7$}
\label{tab-e6/e7}
\end{longtable}
 \end{theorem}
\begin{remark}
We can check (using tables in~\cite[\S8.6-\S8.8]{TX} and Ennola duality) that in each case the Hecke algebra $\cH_{W}$ coincides with $\cH_{\mathbb{T},1}^{\mathbb{G}}(1)$, specialisation at $x=1$ of the cyclotomic Hecke algebra $\cH_{\mathbb{T},1}^{\mathbb{G}}(x)$ attached by Brou\'e-Malle~\cite{BM}  to the principal $\Phi_m$ Harish-Chandra series~\cite{BMM} of Brou\'e-Malle-Michel, see~\cite[\S1]{TX}.
\end{remark}

\begin{remark} We note that in~\cite{RLYG} the group  is assumed to be adjoint. So our $I$ differs from those in {\em loc.cit.} by $Z_G$, the center of $G$, in types $E_6,E_7$.
The data of $W_{\chi_0}=W$ is taken from~\cite{RLYG}, where we use notations of the Shephard-Todd classification~\cite{ST}.
\end{remark}

Here we included the statement from~\cite[\S6]{VX2} for the cases when $m$ is the Coxeter number or twisted Coxeter number. In Table~\ref{tab-equal} and Table~\ref{tab-e6/e7} we also include data on endoscopic groups $\check G(\chi)^0$ and, when $W_\chi^0\neq W_\chi^{en}$, the Hecke algebras $\cH_{W_\chi^{\en}}$ (see~\cite[\S5]{VX2}) associated to $W_\chi^{\en}$. Recall that $W_\chi^{\en}$ is the little Weyl group associated to the stable order $m$ automorphism of $\check G(\chi)^0$. This allows us to verify the endoscopic interpretation of the Hecke algebras $\cH_{W_\chi^0}$, that is, the expectations~\cite[(5.5a) and (5.5b)]{VX2} and Theorem 5.2 of {\em loc.cit.} hold in each case.  Let us summarize this as the following 
\begin{corollary} \label{cor-end}
Let $G$ be a simply connected almost simple algebraic group over $\bC$ and $\theta:G\to G$ an order $m$ stable automorphism (where $(G,\theta)\neq(G_2,3_\rs)$). Let $\chi\in \hat I$ be a nontrivial character and $\check G(\chi)^0$ the corresponding endoscopic group.  
Let $\cH_{W_\chi^{\en}}$ be the Hecke algebra associated to the trivial character for the stable order $m$ automorphism of $\check G(\chi)^0$. Then we have
\beqn
\cM_\chi\cong\bC[\widetilde{B}_W]\otimes_{\bC[\widetilde{B}_W^{\chi,\en}]}(\bC_\chi\otimes\cH_{W_\chi^{\en}})
\eeqn
where $B_W^{\chi,\on{en}}=p^{-1}(W_\chi^{\on{en}})$ and $\widetilde{B}_W^{\chi,\on{en}}=\tilde{q}^{-1}(B_W^{\chi,\on{en}})\,.$
\end{corollary}
\begin{remark}
It would be interesting to have a uniform proof of this endoscopic interpretation.
\end{remark}
We prove Theorem~\ref{mainthm} by a case-by-case analysis following the strategy described in~\S\ref{sec-strategy}. In particular, the rank $\geq 2$ cases occupy~\S\ref{sec-reduction} and~\S\ref{sec-rankgt2}, and the rank one cases are dealt with in~\S\ref{sec-rkone}.  The triviality of the characters $\tau$ and $\rho$ in~\eqref{eqn-nearby} are summarized in the two lemmas below (making use of the case-by-case analysis).

\begin{lemma}\label{lem-tau}
The character $\tau:I\to\{\pm1\}, x\mapsto \on{det}(x|_{\Lg_1})$ is trivial.
\end{lemma}
\begin{proof}
It is clear that the lemma holds when $I=Z_G$.  For the other cases, as in~\cite[Proof of Lemma 7.11]{VX2}, for $x\in I=T^\theta$, we have $\on{det}(x|_{\Lg_1})=\prod_{\alpha\in R/\sim}\alpha(x)$, where $\alpha\sim\beta$ if they are in the same $\theta$-orbit. This implies that in particular, if $\alpha$ and $-\alpha$ are not in the same orbit, then $\on{det}(x|_{\Lg_1})=1$. This is the case when $m=3$ or $5$. We can check the remaining cases using the generators of $I$ described in~\S\ref{sec-rankgt2} and~\S\ref{sec-rkone}.
\end{proof}
\begin{lemma}\label{lem-rho}
The character $\rho$ in~\eqref{eqn-nearby} is trivial.
\end{lemma}
\begin{proof}
This follows from the definition of $\rho$, see~\cite[(2.66)]{GVX}, by observing that we have $\bar R_{\chi,s}=z-1$ when $e_s=n_s$.
\end{proof}

 \subsection{Cuspidal character sheaves}
Recall the subgroups $W_\chi^{\en}\subset W_\chi^0\subset W_\chi\subset W$ associated to $\chi\in\hat I$. The following  propositions follow from our case-by-case analysis.

\begin{proposition}
\hfill\begin{enumerate}
\item The group $I$ is generated by the subgroups $I_s$ (see~\eqref{eqn-local}), $s\in W$ distinguished reflections.
\item A distinguished reflection $s\in W_\chi^{\en}$ if and only if $\chi|_{I_s}=1$.
\end{enumerate}
\end{proposition}
\begin{proposition}
 We have that $W_{\chi}=W_{\chi}^0=W_{\chi}^{\on{en}}$ except the following cases
\hfill\begin{enumerate}[itemsep=.5em]
\item $W_\chi^0=W_{\chi}^{\on{en}}$ and $W_\chi/W_\chi^0\cong \mu_3$  when $(G,\theta)=(E_6,3_\rs)$ and $\chi\neq \chi_0$ 
\item $W_{\chi}= W_\chi^0$ and $W_{\chi}^0/W_{\chi}^{\on{en}}\cong\mu_3$ when $G=E_6,$ $\theta\neq3_\rs$ and $\chi\neq \chi_0$

\item $W_\chi^0=W_{\chi}^{\on{en}}$ and $W_\chi/W_\chi^0\cong\mu_2$  when $(G,\theta)=(E_7,2_\rs)$ and $\chi=\chi_1$ or $\chi_3$ 
\item $W_{\chi}= W_\chi^0$ and $W_{\chi}^0/W_{\chi}^{\on{en}}\cong\mu_2$ when $G=E_7$, $\theta\neq2_\rs$ and $\chi\neq \chi_0$\,. 
\end{enumerate}

\end{proposition}
As in~\cite[\S7]{VX2}, let $\on{Irr}(\cH_{W_\chi^0})$ denote the set of irreducible representations (over~$\bC$) of the Hecke algebra $\cH_{W_\chi^0}$. For $\rho\in\on{Irr}\cH_{W_\chi^0}$, let
$$V_{\rho,\chi}:=\bC[\widetilde{B}_W]\otimes_{\bC[\widetilde B_{W}^{\chi,0}]}(\bC_\chi\otimes\rho)$$ where $B_{W}^{\chi,0}$ acts on $\rho$ via the quotient $\bC[\widetilde B_{W}^{\chi,0}]\to\cH_{W_\chi^0}$. Then $V_{\rho,\chi}$ is an irreducible $\bC[\widetilde{B}_W]$-module if $W_\chi=W_\chi^0$. When $W_\chi\neq W_\chi^0$, let us write $V_{\rho,\chi}^\delta$ for the irreducible direct summands of $V_{\rho,\chi}$. Let $\cL_{\rho,\chi}$, resp. $\cL_{\rho,\chi}^\delta$ denote the irreducible $G_0$-equivariant local systems on $\Lg_1^{rs}$ corresponding to $V_{\rho,\chi}$, resp. $V_{\rho,\chi}^\delta$.

\begin{corollary}The set of cuspidal character sheaves is
\bern
&&\{\on{IC}(\Lg_1^{rs},\cL_{\rho,\chi})\mid \rho\in\on{Irr}\cH_{W_\chi^0},\ \chi\in\hat I/W,\ W_\chi=W_\chi^0\}\\&\cup&\{\on{IC}(\Lg_1^{rs},\cL_{\rho,\chi}^\delta)\mid \rho\in\on{Irr}\cH_{W_\chi^0}, \chi\in \hat I/W,\ W_\chi\neq W_\chi^0\}\,.
\eern 
\end{corollary}

\subsection{Calculating the endoscopic groups}In this subsection we give an example in the case of $(G,\theta)=(E_6,3_\rs)$ to demonstrate how we calculate the endoscopic groups $\check G(\chi)^0$. Let $\check G$ be the dual group and $\check T$ the dual torus of $T=Z_G(\fa)$.

We have $I=T^\theta$ and we can regard $\chi\in\hat I$ as an element of $\check T^{\check\theta}$ (see~\cite[(5.1)]{VX2}). Let $\omega_i\in X_*(\check T)=X^*(T)$, $1\leq i\leq 6$, be such that
$
\langle\omega_i,\check\alpha_j\rangle=\delta_{i,j},\,i,j=1,\ldots,6.
$
Then $\check T=\{\prod_{i=1}^6\omega_i(x_i)\mid x_i\in\bC^*\}$. 
Let us write $\omega_i=\sum_{j=1}^6 a_{ij}\alpha_j$ as $\omega_i=(a_{i1},\ldots,a_{i6})$. Then
\bern
&&\omega_1=\tfrac{1}{3}(4,3,5,6,4,2),\ \omega_2=(1,2,2,3,2,1),\ \omega_3=\tfrac{1}{3}(5,6,10,12,8,4)\\
&&\omega_4=(2,3,4,6,4,2),\ \omega_5=\tfrac{1}{3}(4,6,8,12,10,5),\ \omega_6=\tfrac{1}{3}(2,3,4,6,5,4).
\eern
Let  $w_h$ be the Coxeter element in $E_6$ chosen in~\S\ref{ssec-dynkin}. We can assume that $\theta|_T=w=w_h^4$.
One checks that 
$w(\omega_i)= -\omega_1+\omega_2+\omega_3-\omega_4, -\omega_1+\omega_3-\omega_4+\omega_5-\omega_6,\,$$
 -\omega_1+\omega_2+\omega_3-2\omega_4+\omega_5,\,
 -\omega_1+\omega_2+2\omega_3-3\omega_4+2\omega_5-\omega_6,\,\
 \omega_2+\omega_3-2\omega_4+\omega_5-\omega_6,\,
 \omega_2-\omega_4+\omega_5-\omega_6.
$
So we have
\bern
&&\check T^{w}=\langle\check\gamma_1:=\omega_1(\xi_3^2)\,\omega_3(\xi_3^2)\,\omega_5(\xi_3)\,\omega_6(\xi_3),\ \check\gamma_2:=\omega_2(\xi_3)\,\omega_3(\xi_3^2)\,\omega_5(\xi_3^2),\\
&&\qquad\qquad \check\gamma_3:=\omega_3(\xi_3^2)\,\omega_4(\xi_3)\,\omega_5(\xi_3)\omega_6(\xi_3^2)\rangle\cong\mu_3^3\,.
\eern
Let us write $\check\gamma^{abc}=\check\gamma_1^{a}\check\gamma_2^{b}\check\gamma_3^{c}$. One checks that 
\begin{align*}
&(\check G^{\check\gamma^{abc}})^0\cong A_2\times A_2\times A_2&& abc=010,102,112,122,020,201,221,211\\
&(\check G^{\check\gamma^{abc}})^0\cong  D_4&&\text{otherwise}.
\end{align*}
Moreover $\theta$ leaves each factor of $A_2$ stable in the first case. For example, $(\check G^{\check\gamma_2})^0$ has root basis $\{\beta_1=100000,\theta\beta_1=011100\}\cup\{\beta_2=000001,\theta\beta_2=010110\}\cup\{\beta_3=000100,\theta\beta_3=112221\}$, where $\theta^2\beta_i=-\beta_i-\theta\beta_i$, $i=1,2,3$.

\section{Distinguished reflections and reduction to rank one}\label{sec-reduction}

In this section we determine, for each distinguished  reflection $s\in W$, the structure of $(G_{s})_\der$, the groups $I_s$, and identify $\theta|_{(G_{s})_\der}$ in a case-by-case manner. We do so by describing the distinguished reflections in $W$ explicitly in~\S\ref{ssec-cyclotomic} and~\S\ref{ssec-noncyclo} (when $\dim\fa\geq 2$) following Reeder's approach in~\cite{R}.

Recall the notations from~\S\ref{ssec-thetaorbits}.  The following proposition (together with the analysis in~\S\ref{ssec-cyclotomic}-\S\ref{ssec-noncyclo}) allows us to determine $\theta|_{(G_s)_\der}$.
\begin{proposition}\label{prop-distinguished}
There exists a partition $ R=\sqcup_{i} S_i$ such that 
\begin{enumerate}
\item each $S_i$ is a union of $\theta$-orbits in $R$ and it forms a sub root system of $ R$ 
\item each $S_i$ determines a unique distinguished reflection $t_i:=t_{S_i}\in W$
\item $\theta|_{(\Lg_{S_i})_{\on{der}}}$ is (GIT) stable of rank 1 with little Weyl group $\langle t_i\rangle$.
\end{enumerate}
The distinguished reflections in   $W$ are precisely the $t_{i}$'s, so the reflections in $W$ are $t_{i}^j$, $j=1,\ldots,|\langle t_{i}\rangle |-1$. Moreover,
$
Z_\Lg(\fa_{t_{i}})=\Lg_{S_i}
:=\Lt\oplus\bigoplus_{\alpha\in S_i}\Lg_\alpha\,$, so $(\Lg_{t_i})_\der=(\Lg_{S_i})_{\der}$.
\end{proposition} 
\begin{remark}
The proposition also holds for stably graded Lie algebras in classical types, see~\cite[pp44-45]{VX2}\,.
\end{remark}

\begin{proposition}\label{prop-rankone}
For each $(G,\theta)$ of rank $\geq 2$, the $\theta|_{(G_{s})_\der}$ are as follows
\FloatBarrier
\begin{longtable}[c]{p{1.7cm}p{2.2cm}p{1.5cm}p{1cm}p{3.2cm}}
\hline\hline
$(G,\theta)$&$w\vartheta$&$\Phi(w\vartheta)$&$r$&$\theta|_{(G_s)_\der}$\\
\hline\hline
$(^3D_4,3_\rs)$&$A_2\times\tilde A_2$&$\Phi_3^2$&$2$&$(A_2,3_\rs)$\\
\hline
$(F_4,3_\rs)$&$A_2+\tilde{A_2}$&$\Phi_3^2$&$2$&$(A_2,3_\rs)$\\ 
\hline
$(F_4,4_\rs)$&$D_4(a_1)$&$\Phi_4^2$&$2$&$(B_2,4_\rs)$\\ 
\hline
$(F_4,6_\rs)$&$F_4(a_1)$&$\Phi_6^2$&$2$&$(^2A_2,6_\rs)$\\
\hline
$(E_6,3_\rs)$&$3A_2$&$\Phi_3^3$&$3$&$(A_2,3_\rs)$\\
\hline
$(E_6,6_\rs)$&$E_6(a_2)$&$\Phi_6^2\Phi_3$&$2$&$(^2A_2,6_\rs)$ $(A_2,3_\rs)$\\
\hline
$(^2E_6,4_\rs)$&$-D_4(a_1)$&$\Phi_4^2\Phi_2^2$&$2$&$(A_3,4)$\\
\hline

$(E_7,6_\rs)$&$E_7(a_4)$&$\Phi_6^3\Phi_2$&$3$&$(A_1,2_\rs)\ (^2A_2,6_\rs)$\\ 
\hline
$(E_8,3_\rs)$&$4A_2$&$\Phi_3^4$&$4$&$(A_2,3)$\\ 
\hline
$(E_8,4_\rs)$&$2D_4(a_1)$&$\Phi_4^4$&$4$&$(A_1,2_\rs)$\\
\hline
$(E_8,5_\rs)$&$2A_4$&$\Phi_5^2$&$2$&$(A_4,5_\rs)$\\ 
\hline
$(E_8,6_\rs)$&$E_8(a_8)$&$\Phi_6^4$&$4$&$(^2A_2,6_\rs)$\\ 
\hline
$(E_8,8_\rs)$&$D_8(a_3)$&$\Phi_8^2$&$2$&$(^2D_4,8_\rs)\ (A_2,2_\rs)$\\ 
\hline
$(E_8,10_\rs)$&$E_8(a_6)$&$\Phi_{10}^2$&$2$&$(^2A_4,10_\rs)$\\ \hline
$(E_8,12_\rs)$&$E_8(a_3)$&$\Phi_{12}^2$&$2$&$(^3D_4,12_\rs)\ (^2A_2,6_\rs)$\\ 
\hline\hline
\caption{Stable gradings of rank $\geq 2$}
\label{tab-rankgt2}
\end{longtable}
\end{proposition}
\begin{remark}
We omit the entries for stable $\bZ/2\bZ$-gradings (see~\S\ref{ssec-stable2}) where $\theta|_{(G_s)_\der}\cong (A_1,2_\rs)$. We have included in Table~\ref{tab-rankgt2} the following data from~\cite{RLYG}:
\begin{itemize}
\item $w\vartheta\in W_G\vartheta$ that gives rise to $\theta$. The notation for Weyl group elements is from~\cite{Ca}.
\item $\Phi(w\vartheta)$ denotes the characteristic polynomial of $w\vartheta$.

\item $r=\dim\fa$, the rank of the graded Lie algebra.
\end{itemize}

\end{remark}

To prove Proposition~\ref{prop-distinguished} and hence Proposition~\ref{prop-rankone}, we will describe the subsets $S_i$ in the partition $R=\sqcup S_i$ explicitly and associate a unique distinguished reflection $t_{S_i}\in W$ to each $S_i$, case-by-case, in the remainder of this section. We also determine  the groups $I_{S_i}$ associated to $\theta|_{(\Lg_{S_i})_{\der}}$ which we use in~\S\ref{sec-rankgt2}.  We give a summary of the subroot systems arising from $S_i$, referred to as root types,  in the table below (grouped according to the little Weyl group, denoted by $C$ here), where $R_0\,(a\times b)$ indicates that there are $b$ subsets $S_i\subset R$ of type $R_0$ and each $S_i$ consists of $a$ orbits. We also include the following data 
\begin{itemize}
\item $|C|$ denotes the size of the complex reflection group $C$, $N_s$ denotes the numbers of distinguished reflections in $C$, where $a^{n_a}$ means that the number of order $a$ distinguished reflections is $n_a$ \cite{Co}
\item diagram of $C$ as in~\cite{BMR}
\item the possible $(G,\theta)$ such that $C$ arises as the corresponding little Weyl group (\cite{RLYG})
\end{itemize}
\FloatBarrier
\begin{longtable}[c]{p{1cm}p{1.8cm}p{1.2cm}p{5.5cm}p{1.8cm}p{2cm}}\\
\hline\hline
$C$&$|C|$&$N_s$&Diagram&$(G,\theta)$&Root types\\\hline\hline
$ G_4$&$2^3\cdot 3$&$3^4$&$\xymatrix{{\text{\textcircled{\tiny{3}}}}\ar@{-}[r]&{\text{\textcircled{\tiny{3}}}}\,}$&$(^3D_4,3_\rs)$&$A_2\,(2\times 4)$\\&&&&$(^3D_4,6_\rs)$&$A_2\,(1\times 4)$\\\hline
$G_5$&$2^3\cdot 3^2$&$3^{8}$&$\xymatrix{{\text{\textcircled{\tiny{3}}}}\ar@{=}[r]&{\text{\textcircled{\tiny{3}}}}\,}$&$(F_4,3_\rs)$&$A_2\,(2\times 8)$\\
&&&&$(F_4,6_\rs)$&$A_2\,(1\times 8)$\\
&&&&$\begin{array}{cc}(E_6,6_\rs)\\\ \end{array}$&$\begin{array}{cc}A_2\,(1\times 4)\\A_2^2(2\times 4)\end{array}$\\
\hline
$G_8$& $2^5\cdot 3$&$4^{6}$&$\xymatrix{{\text{\textcircled{\tiny{4}}}}\ar@{-}[r]&{\text{\textcircled{\tiny{4}}}}\,}$&$(F_4,4_\rs)$&$B_2\,(2\times 6)$\\
&&&&$(^2E_6,4_\rs)$&$A_3\,(3\times 6)$\\\hline
$G_{9}$&$2^6\cdot 3$&$\begin{array}{cc}4^{6}\\ 2^{6}\end{array}$&$\xymatrix{{\text{\textcircled{\tiny{4}}}}\ar@3{-}[r]&{\text{\textcircled{\tiny{2}}}}}$&$\begin{array}{cc}(E_8,8_\rs)\\\ \end{array}$&$\begin{array}{cc}D_4\,(3\times 6)\\A_1^4(1\times 12) \end{array}$\\\hline
$G_{10}$&$2^5\cdot3^2$&$\begin{array}{cc}4^{6}\\3^{8}\end{array}$&$\xymatrix{{\text{\textcircled{\tiny{4}}}}\ar@2{-}[r]&{\text{\textcircled{\tiny{3}}}}}$&$\begin{array}{cc}(E_8,12_\rs)\\\ \end{array}$&$\begin{array}{cc}D_4\,(2\times 6)\\A_2^2\,(1\times 8)\end{array}$\\\hline
$G_{16}$&$2^3\cdot 3\cdot 5^2$&$5^{12}$&\xymatrix{{\text{\textcircled{\tiny{5}}}}\ar@{-}[r]&{\text{\textcircled{\tiny{5}}}}}&$(E_8,5_\rs)$&$A_4\,(4\times12)$\\&&&&$(E_8,10_\rs)$&$A_4\,(2\times12)$\\\hline
$G_{25}$&$2^3\cdot 3^4$&$3^{12}$&$\xymatrix{{\text{\textcircled{\tiny{3}}}}\ar@{-}[r]&{\text{\textcircled{\tiny{3}}}}\ar@{-}[r]&{\text{\textcircled{\tiny{3}}}}}$&$(E_6,3_\rs)$&$A_2\,(2\times 12)$\\&&&&$(^2E_6,6_\rs)$&$A_2\,(1\times 12)$\\\hline
$G_{26}$&$2^4\cdot 3^4$&$\begin{array}{c}2^9\\3^{12}\end{array}$&$\xymatrix{{\text{\textcircled{\tiny{2}}}}\ar@{=}[r]&{\text{\textcircled{\tiny{3}}}}\ar@{-}[r]&{\text{\textcircled{\tiny{3}}}}}$&$\begin{array}{c}(E_7,6_\rs)\\\ \end{array}$&$\begin{array}{c}A_1^4\,(1\times 9)\\A_2\,(1\times 12) \end{array}$\\\hline\\
$G_{31}$&$2^{10}\cdot 3^2\cdot 5$&$2^{60}$&$\xymatrix@C-.8pc{&{\text{\textcircled{\tiny{2}}}}\ar@{-}[dl]\ar@{}@/^{2pc}/[rr]\ar@{}@/_{1pc}/[dr]&&{\text{\textcircled{\tiny{2}}}}\ar@{-}[dr]\ar@{}@/^{1pc}/[dl]&\\{\text{\textcircled{\tiny{2}}}}\ar@{-}[rr]&&{\text{\textcircled{\tiny{2}}}}\ar@{-}[rr]&&{\text{\textcircled{\tiny{2}}}}}$&$(E_8,4_\rs)$&$A_1^2\,(1\times 60)$\\\hline
$G_{32}$& $2^7\cdot3^5\cdot 5$ & $3^{40}$&
$
\xymatrix{{\text{\textcircled{\tiny{3}}}}\ar@{-}[r]&{\text{\textcircled{\tiny{3}}}}\ar@{-}[r]&{\text{\textcircled{\tiny{3}}}}\ar@{-}[r]&{\text{\textcircled{\tiny{3}}}}}
$&$(E_8,3_\rs)$&$A_2\,(2\times 40)$\\
&&&&$(E_8,6_\rs)$&$A_2\,(1\times 40)$\\\hline\hline
\caption{Root types and distinguished reflections}
\label{tab-rt}
\end{longtable}
When $\vartheta=1$ and $w$ is cyclotomic (i.e., its minimal polynomial is irreducible over $\bQ$),  the root types have been determined in~\cite{R}. In what follows we will use a similar strategy to study the remaining cases.
 
 \begin{proof}[Proof of Proposition~\ref{prop-distinguished}] 
 Let us write $t_i=t_{S_i}$. 
 By our case-by-case analysis in~\S\ref{ssec-cyclotomic}-\S\ref{ssec-noncyclo}, we have that 
 \beq\label{eqn-claim}
 \text{$a\in\fa_{t_{i}}$ if and only if $\beta(a)=0$ for any $\beta\in S_i$.}
 \eeq
  Since the $S_i$'s are disjoint, we conclude that $t_{i}\neq t_{j}$, $i\neq j$. Comparing the number of $S_i$'s and the number of distinguished reflections in the little Weyl group $W$ in Table~\ref{tab-rt}, we see that the $t_{i}$'s are precisely the distinguished reflections in $W$. 
  
Since $Z_\Lg(\fa_{t_i})=\Lt\oplus\bigoplus_{\alpha\in R,\alpha(\fa_{t_i})=0}\Lg_\alpha$,~\eqref{eqn-claim} implies that $\Lg_{S_i}\subset Z_\Lg(\fa_{t_i})$. By~\cite[Theorem 2.12]{DK}, we have $Z_{\Lg_0}(\fa_{t_i})\cap Z_{\Lg_0}(\fa_{t_j})=Z_{\Lg_0}(\fa)=0$ for $i\neq j$. Note that if $\alpha\in Z_{\Lg}(\fa_{t_i})$, then $\theta\alpha\in Z_{\Lg}(\fa_{t_i})$. So $\{\alpha\in R,\alpha(\fa_{t_i})=0\}$ is a union of $\theta$-orbits in $R$. Since each $\theta$-orbit of roots contributes one dimensional to $\Lg_0$ (see~\cite[Lemma 8.2]{LVX}), we conclude that $\Lg_{S_i}= Z_\Lg(\fa_{t_i})$. Together with our explicit description of the $S_i$'s in~\S\ref{ssec-cyclotomic}-\S\ref{ssec-noncyclo}, this completes the proof of the proposition. 
  \end{proof}

\subsection{The cyclotomic cases.} \label{ssec-cyclotomic}

As in~\cite{R}, we say that $w\vartheta$ is cyclotomic, if its minimal polynomial $M(w\vartheta)$ is irreducible (over $\bQ$). Suppose that $w\vartheta$ is cyclotomic.  Let $K=\bQ(w\vartheta)$ and $K\alpha=\{f(w\vartheta)\alpha\mid f(t)\in\bQ[t]\}$ for $\alpha\in R$. Two roots $\alpha$ and $\beta$ are said to be $K$-equivalent if $K\alpha=K\beta$. Then each $K$-equivalence class  $S\subset R$ forms a subroot system of rank equal to the degree of $K$ over $\bQ$. The same argument as in~\cite{R} shows that the primitive reflections (i.e., powers of distinguished reflections with the same order) of $W=W_G^\theta$ (see~\eqref{eqn-weyl}) are precisely the generators of the cyclic groups $W(S)\cap W_G^\theta $, where $S$ runs through the $K$-equivalence classes in $ R$ and $W(S)\subset W_G$ is the Weyl group of $S$ (generated by the reflections $s_\alpha$, $\alpha\in S$). In what follows, we describe each case more explicitly.

\subsubsection {$m=3$, $M(w\vartheta)=\Phi_3$.} In this case $\theta^2\alpha+\theta\alpha+\alpha=0$ so $\langle\theta^2\alpha,\check\alpha\rangle=\langle\theta\alpha,\check\alpha\rangle=-1$ (note that $\theta$ takes long, resp. short roots to long, resp. short roots). The $K$-equivalence classes are of the form $$S=\{\alpha,\theta\alpha,\theta^2\alpha\}\cup\{-\alpha,-\theta\alpha,-\theta^2\alpha\}\cong A_2$$ with a root basis $\{\alpha,\theta\alpha\}$. It follows that $\theta|_{(\Lg_S)_{\on{der}}}\cong(A_2,3_\rs)$. The distinguished reflection of order $3$ is 
\beqn
t_S=s_{\alpha}s_{\theta\alpha}\,\text{ and }I_{S}=\langle\check\alpha(\zeta_3^2)\theta\check\alpha(\zeta_3)\rangle\cong\mu_3.
\eeqn
This can be seen explicitly as follows. We have $t_S(\delta)=\delta+\theta^2\alpha(\delta) h_\alpha-\theta\alpha(\delta) h_{\theta\alpha},\,\delta\in\Lt$ and so $s_\alpha s_{\theta\alpha}=s_{\theta\alpha}s_{\theta^2\alpha}$. It follows that $t_S$ is well-defined (independent of choice of $\alpha\in S$) and $t_S\in W_G^\theta=W$. Moreover, $t_S(a)=a+\alpha(a)(\zeta_3h_\alpha-\zeta_3^2h_{\theta\alpha}),\ a\in\fa,\ 
t_S(a_0)=\zeta_3a_0\text{ for }a_0=h_\alpha-\zeta_3h_{\theta\alpha}\in\fa\,$. 
  Since $t_S(a)=a\iff \alpha(a)=0$ for $a\in\fa$ and $t_S(a_0)=\zeta_3a_0$, we conclude that $t_S$ is a distinguished reflection of order $3$. The claim on $I_S$ follows since 
   $I_S=\{x:=\check\alpha(x_1)\check\theta\alpha(x_2)\mid t_S(x)=x\}$.   
If we write $S_i=\pm\{\theta^a\beta_i,a=0,1,2\}$, $t_i=s_{\beta_i}s_{\theta\beta_i}$ and $\nu_i:=\check\beta_i(\zeta_3^2)\theta\check\beta_i(\zeta_3)$, then we have
\beq\label{action-m=3}
I_{t_i}=\langle\nu_i\rangle\text{ and } t_i(\nu_j)=\nu_j\nu_i^{\langle\check\beta_j-\theta\check\beta_j,\beta_i\rangle}\,.
\eeq

\subsubsection{$m=6$, $M(w\vartheta)=\Phi_6$.}  In this case, $\theta^2\alpha-\theta\alpha+\alpha=0$ so $\langle\theta^2\alpha,\check\alpha\rangle=-\langle\theta\alpha,\check\alpha\rangle=-1$. 
The $K$-equivalence classes are of the form $$S=\{\alpha,\theta\alpha,\theta^2\alpha,-\alpha,-\theta\alpha,-\theta^2\alpha\}\cong A_2$$ with a root basis $\{\alpha,-\theta\alpha\}$. It follows that $\theta|_{(\Lg_S)_{\on{der}}}\cong(^2A_2,6_\rs)$. The distinguished reflection of order $3$ is 
\beqn
t_S=s_{\alpha}s_{\theta\alpha}s_{\alpha}s_{\theta\alpha}\,\text{ and }I_{S}=1.
\eeqn

\subsubsection{$m=5$, $M(w)=\Phi_5$} In this case, $\sum_{i=0}^4\theta^i\alpha=0$. So either $\langle\theta\alpha,\check\alpha\rangle=\langle\theta^4\alpha,\check\alpha\rangle=-1$ or $\langle\theta^2\alpha,\check\alpha\rangle=\langle\theta^3\alpha,\check\alpha\rangle=-1$. We can choose $\alpha$ with $\langle\theta\alpha,\check\alpha\rangle=-1$ so that 
the $K$-equivalence classes are of the form
$$S=\{\theta^i\alpha,i=0,\ldots,4\}\cup\{-\theta^i\alpha,i=0,\ldots,4\}\cup\{\theta^i\beta,i=0,\ldots,4\}\cup\{-\theta^i\beta,i=0,\ldots,4\}\cong A_4$$ where $\beta=\alpha+\theta\alpha$. A root basis of $S$ is $\{\alpha,\theta\alpha,\theta^2\alpha,\theta^3\alpha\}$.  It follows that $\theta|_{(\Lg_S)_{\on{der}}}\cong(A_4,5_\rs)$. The distinguished reflection of order $5$ is 
\beqn
t_S=s_{\alpha}s_{\theta\alpha}s_{\theta^2\alpha}s_{\theta^3\alpha}\,\text{ and }I_{S}=\langle\prod_{i=0}^3\theta^i\check\alpha(\zeta_5^{4-i})\rangle\cong\mu_5.
\eeqn
If we write $S_i=\pm\{\theta^a\beta_i\}\cup\pm\{\theta^a(\beta_i+\theta\beta_i)\}$, $t_i=s_{\beta_i}s_{\theta\beta_i}s_{\theta^2\beta_i}s_{\theta^3\beta_i}$ and $\nu_i:=\prod_{a=0}^3\theta^a\check\beta_i(\zeta_5^{4-a})$, then we have
\beq\label{action-m=5}
I_{t_i}=\langle\nu_i\rangle\text{ and } t_i(\nu_j)=\nu_j\nu_i^{\langle\check\beta_j+2\theta\check\beta_j-2\theta^2\check\beta_j-\theta^3\check\beta_j,\beta_i\rangle}\,.
\eeq

\subsubsection{$m=10$, $M(w)=\Phi_{10}$}  In this case, $\sum_{i=0}^4(-1)^i\theta^i\alpha=0$. We can choose $\alpha$ such that $\langle\theta\alpha,\check\alpha\rangle=1$ and the $K$-equivalence classes are
$$S=\{\pm\theta^i\alpha,i=0,\ldots,4\}\cup \{\pm\theta^i\beta,i=0,\ldots,4\}\cong A_4$$ where $\beta=\alpha-\theta\alpha$. A root basis of $S$ is $\{\alpha,-\theta\alpha,\theta^2\alpha,-\theta^3\alpha\}$. It follows that $\theta|_{(\Lg_S)_{\on{der}}}\cong(^2A_4,5_\rs)$. The distinguished reflection (of order 5) is
\beqn
t_S=(s_{\alpha}s_{\theta\alpha}s_{\theta^2\alpha}s_{\theta^3\alpha})^2,\ \text{ and }I_S=1.
\eeqn

\subsubsection{$m=4$ (and $G=E_8$), $M(w)=\Phi_{4}$}  In this case, we $\langle\theta\alpha,\check\alpha\rangle=0$ for all $\alpha\in R$. The $K$-equivalence classes are
 $$S=\{\alpha,\theta\alpha,-\alpha,-\theta\alpha\}\cong A_1^2.$$  
 It follows that $\theta|_{(\Lg_S)_{\on{der}}}\cong(A_1,2_\rs)$. The distinguished reflection (of order 2) is
\beqn
t_S=s_\alpha s_{\theta\alpha},\ \text{ and }I_S=\langle\check\alpha(-1)\theta\check\alpha(-1)\cong\mu_2.
\eeqn
If we write $S_i=\pm\{\beta_i,\theta\beta_i\}$, $t_i=s_{\beta_i}s_{\theta\beta_i}$ and $\nu_i:=\prod_{a=0}^1\theta^a\check\alpha(-1)$, then we have
\beq\label{action-m=4-E8}
I_{t_i}=\langle\nu_i\rangle\text{ and } t_i(\nu_j)=\nu_j\nu_i^{\langle\check\beta_j+\theta\check\beta_j,\beta_i\rangle}\,.
\eeq

\subsubsection{$m=4$  (and $G=F_4$), $M(w)=\Phi_{4}$}  In this case, we can order the roots so that $\alpha+\beta=-\theta\beta$, $\beta-\theta\beta=\theta\alpha$ and  the $K$-equivalence classes are
$$S=\{\pm\alpha,\pm\theta\alpha\}\cup\{\pm\beta,\pm\theta\beta\}\cong B_2$$ with a root basis $\{\alpha,\beta\}$ (where $\alpha$ is long and $\beta$ is short).  It follows that $\theta|_{(\Lg_S)_{\on{der}}}\cong(B_2,4_\rs)$. The distinguished reflection (of order 4) is
\beqn
t_S=s_\alpha s_{\beta},\ \text{ and }I_S=\langle\check\beta(-1)\rangle\cong\mu_2.
\eeqn
If we write $S_i=\pm\{\delta_i,\theta\delta_i\}\cup\pm\{\beta_i,\theta\beta_i\}$ (where $\delta_i$ is long and $\beta_i$ is short), $t_i=s_{\delta_i}s_{\beta_i}$  and $\nu_i:=\check\beta_i(-1)$, then we have
\beq\label{action-m=4-F4}
I_{t_i}=\langle\nu_i\rangle\text{ and } t_i(\nu_j)=\nu_j\nu_i^{\langle\check\beta_j,\beta_i\rangle}\,.
\eeq

\subsubsection{$m=8$  (and $G=E_8$), $M(w)=\Phi_{8}$}  By~\cite[\S3.4.4]{R},
\begin{enumerate}
\item There are $6$ $K$-equivalence classes of the form $$S=\{\pm\theta^i\beta,i=0,1,2,3\}\cup\{\pm\theta^i\alpha,i=0,1,2,3\}\cup\{\pm\theta^i\gamma,i=0,1,2,3\}\cong D_4$$ where one can choose $\alpha,\beta,\gamma$ such that $\gamma=\beta+\theta\beta-\theta^{3}\beta$, $\alpha=\beta+\theta\beta$, $\langle \check\alpha,\theta\alpha\rangle=0$ and  $\langle \check\beta,\theta\beta\rangle=-1$. A root basis is $\{\beta,\theta\beta,\theta^2\beta,-\gamma\}$ (with $\theta\beta$ the branching node).

We have $\theta|_{(\Lg_S)_\der}\cong(^2D_4,8_\rs)$. The distinguished reflection (of order $4$) is
\beqn
t_S=(s_\gamma s_{\theta^2\beta}s_{\theta\beta}s_\beta s_{\theta^2\beta}s_{\theta\beta})^3\text{ and }I_S=\langle\theta^2\check\beta(-1)\check\gamma(-1)\rangle=\langle\prod_{i=0}^3\theta^i\check\beta(-1)\rangle\cong\mu_2\,.
\eeqn

\item There are $12$ $K$-equivalence classes of the form $$S=\{\pm\theta^i\alpha,i=0,1,2,3\}\cong A_1^4.$$
Note that $\theta$ permutes the factors of $A_1$ cyclicly. So $\theta|_{(\Lg_S)_\der}\cong(A_1,2_\rs)$. The distinguished reflection (of order $2$) is
\beqn
t_S=\prod_{i=0}^3s_{\theta^i\alpha}\text{ and }I_S=\langle\prod_{i=0}^3\theta^i\check\alpha(-1)\rangle\cong\mu_2\,.
\eeqn

\end{enumerate}
If we write $S_j=\pm\{\theta^a\beta_j\}\cup\pm\{\theta^a(\beta_j+\theta\beta_j)\}\cup\pm\{\theta^a\gamma_j\}$, $t_j=(s_{\gamma_j} s_{\theta^2\beta_j}s_{\theta\beta_j}s_\beta s_{\theta^2\beta_j}s_{\theta\beta_j})^3$, $j=1,\ldots,6$, $S_i=\pm\{\theta^a\beta_i\}$, $t_i=\prod_{a=0}^3s_{\theta^a\beta_i}$, $i=7,\ldots,12$, and $\nu_i=\prod_{a=0}^3\theta^a\check\beta_i(-1)$, then we have
\beq\label{action-m=8-E8}
I_{t_i}=\langle\nu_i\rangle\text{ and }t_i(\nu_j)=\nu_j\nu_i^{\langle\sum_{a=0}^3\theta^a\check\beta_j,\beta_i\rangle}\,.
\eeq

\subsubsection{$m=12$  (and $G=E_8$), $M(w)=\Phi_{12}=x^4-x^2+1$.}  By~\cite[\S3.4.5]{R}, 
\begin{enumerate}
\item There are $8$ $K$-equivalence classes of the form $$S=\{\pm\theta^i\alpha,i=0,1,2,3,4,5\}\cong A_2^2$$ where $\langle\check\alpha,\theta\alpha\rangle=0$, $\langle\check\alpha,\theta^4\alpha\rangle=-\langle\check\alpha,\theta^2\alpha\rangle=-1$ and a root basis of $S$ is $\{\alpha,-\theta^2\alpha\}\cup\{\theta\alpha,-\theta^3\alpha\}$. 
Note that $\theta$ permutes the two factors of $A_2$. We have $\theta|_{(\Lg_S)_\der}\cong(^2A_2,6_\rs)$. The distinguished reflection (of order $3$) is
\beqn
t_S=(s_\alpha s_{\theta^2\alpha} s_{\theta\alpha} s_{\theta^3\alpha})^2\text{ and }I_S=1\,.
\eeqn

\item 
There are $6$ $K$-equivalence classes of the form $$S=\{\pm\theta^i\alpha,i=0,\ldots,5\}\cup\{\pm\theta^i\beta,i=0,\ldots,5\}\cong D_4$$
 where one can choose $\alpha,\beta$ such that  $\langle \check\alpha,\theta\alpha\rangle=1$ and  $\beta=\theta\alpha-\alpha$. A root basis of $S$ is $\{\beta,\theta\beta,-\theta^5\alpha,-\theta^3\beta\}$ (where $\theta\beta$ corresponds to the branch node).  
 
 We have $\theta|_{(\Lg_S)_\der}\cong(^3D_4,12_\rs)$. The distinguished reflection (of order $4$) is
\beqn
t_S=(s_{\theta\beta}s_{\beta}s_{\theta^3\beta}s_{\theta\beta}s_{\theta^5\alpha}s_{\theta^3\beta})^3\text{ and }I_S=1.
\eeqn

\end{enumerate}

\subsection{The non-cyclotomic cases}\label{ssec-noncyclo}

From Table~\ref{tab-rankgt2}, the only non-cyclotomic $w\vartheta$ arises in types $(E_6,6_\rs)$, $(^2E_6,4_\rs)$ and $(E_7,6_\rs)$. We study these cases using similar strategy as follows.

\subsubsection{Type $(E_6,6_\rs)$.} \label{sssec-e6-6s}
In this case we have $M(w)=\Phi_6\Phi_3=x^4+x^2+1$. Let $w=w_h^2$. One checks that 
\begin{enumerate}
\item There are 4 $\theta$-orbits of the form $$S=\{\alpha,\theta\alpha,\theta^2\alpha,-\alpha,-\theta\alpha,-\theta^2\alpha\}\cong A_2$$ with a root basis $\{\alpha,-\theta\alpha\}$, 
where $\theta^2\alpha-\theta\alpha+\alpha=0$. We have $\theta|_{(\Lg_S)_\der}\cong(^2A_2,6_\rs)$. The distinguished reflection (of order $3$) is
$
t_S=(s_\alpha  s_{\theta\alpha} )^2\text{ and }I_S=1\,.
$

\item There are 8 $\theta$-orbits of the form $\{\theta^i\alpha,i=0,\ldots,5\}$, where $\theta^3\alpha\neq-\alpha$ and $\theta^4\alpha+\theta^2\alpha+\alpha=0$. For such orbits, let $$S=\{\theta^i\alpha,i=0,\ldots,5\}\cup\{-\theta^i\alpha,i=0,\ldots,5\}\cong A_2^2$$ with a root basis $\{\alpha,\theta^2\alpha\}\cup\{\theta\alpha,\theta^3\alpha\}$. We have $\theta|_{(\Lg_S)_\der}\cong(A_2,3_\rs)$. The distinguished reflection (of order $3$) is
$$
t_S=s_\alpha s_{\theta^2\alpha} s_{\theta\alpha}s_{\theta^3\alpha}\text{ and }I_S=\langle \check\alpha(\zeta_3^2)\theta^2\check\alpha(\zeta_3)\theta\check\alpha(\zeta_3^2)\theta^3\check\alpha(\zeta_3)\rangle\cong\mu_3\,.
$$

\end{enumerate}

\subsubsection{Type $({}^2E_6,4_\rs)$.} In this case $M(w\vartheta)=\Phi_4\Phi_2=x^3+x^2+x+1$. Let $w\vartheta=-w_h^3$. One checks that 
\begin{itemize}
\item There are 6 $\theta$-orbits of the form $\{\alpha,\theta\alpha,-\alpha,-\theta\alpha\}$, where $\theta^2\alpha+\alpha=0$.  
\item There are 12 $\theta$-orbits of the form $\{\theta^i\alpha,i=0,\ldots,3\}$, where $\theta^2\alpha\neq-\alpha$ and $\theta^3\alpha+\theta^2\alpha+\theta\alpha+\alpha=0$.

\end{itemize}

Let $\alpha\in\Phi$ be such that $\theta^2\alpha\neq-\alpha$. Note that $\theta^3\alpha+\theta^2\alpha+\theta\alpha+\alpha=0$ and $\theta^4\alpha=\alpha$ implies that $\langle\theta\alpha,\check\alpha\rangle=\langle\theta^2\alpha,\theta\check\alpha\rangle=-1$ and $\langle\theta^2\alpha,\check\alpha\rangle=0$. Let $$S=\{\theta^i\alpha,i=0,\ldots,3\}\cup\{-\theta^i\alpha,i=0,\ldots,3\}\cup\{\beta,\theta\beta,-\beta,-\theta\beta\}\cong A_3$$ with a root basis $\{\alpha,\theta\alpha,\theta^2\alpha\}$, where $\beta=\alpha+\theta\alpha$.  We have $\theta|_{(\Lg_S)_\der}\cong(A_3,4_\rs)$. The distinguished reflection of order $4$ is
$$t_S=s_\alpha s_{\theta\alpha}s_{\theta^2\alpha}\text{ and }I_S=\langle\check\alpha(\zeta_4)\theta\check\alpha(-1)\theta^2\check\alpha(\zeta_4^{-1})\rangle\cong\mu_4\,.$$

Let us write $S_i=\pm\{\theta^a\beta_i,a=0,\ldots,3\}\cup\pm\{\theta^a(\beta_i+\theta\beta_i),a=0,1\}$, $t_i=s_{\beta_i} s_{\theta\beta_i}s_{\theta^2\beta_i}$ and $\nu_i=\check\beta_i(\zeta_4^3)\theta\check\beta_i(\zeta_4^2)\theta^2\check\beta_i(\zeta_4)$. We have 
\beq\label{action=m=4-2e6}
I_{t_i}=\langle\nu_i\rangle\text{ and }
t_i(\nu_j)=\nu_j\nu_i^{\langle\check\beta_j+2\theta\check\beta_j+3\theta^2\check\beta_j,\beta_i\rangle}.
\eeq

\subsubsection{Type $(E_7,6_\rs)$.}\label{sssec-e76s} In this case we have $M(w)=\Phi_6\Phi_2=x^3+1$. Let $w=w_h^3$. One checks that
\begin{enumerate}
\item There are $12$ $\theta$-orbits of the form $$S=\{\alpha,\theta\alpha,\theta^2\alpha,-\alpha,-\theta\alpha,-\theta^2\alpha\}\cong A_2\text{ with a root basis $\{\alpha,\theta\alpha\}$}$$ where $\theta^2\alpha-\theta\alpha+\alpha=0$ (so $\langle \theta\alpha,\alpha\rangle=-\langle \theta^2\alpha,\alpha\rangle=\langle\theta^2\alpha,\theta\alpha\rangle=1$). We have $\theta|_{(\Lg_S)_\der}\cong(^2A_2,6_\rs)$.  The distinguished reflection (of order $3$) is
$$t_S=s_{\alpha}s_{\theta\alpha}s_{\alpha}s_{\theta\alpha}\text{ and }I_S=1.$$

\item There are $9$ $\theta$-orbits of the form $$S=\{\beta,\theta\beta,\theta^2\beta,-\beta,-\theta\beta,-\theta^2\beta\}\cong A_1^3$$ with with a root basis $\{\beta\}\cup\{\theta\beta\}\cup\{\theta^2\beta\}$.
We have $\theta|_{(\Lg_S)_\der}\cong(A_1,2_\rs)$. The distinguished reflection (of order $2$) is
$$t_S=s_{\beta}s_{\theta\beta}s_{\theta^2\beta}\text{ and }I_S=\langle\check\beta(-1)\theta\check\beta(-1)\theta^2\check\beta(-1)\rangle\cong\mu_2.$$

\end{enumerate}

\section{Stable gradings of rank at least 2}\label{sec-rankgt2}
 
In this section we prove Theorem~\ref{mainthm} for stable gradings of rank $\geq 2$.  
Recall $T=Z_G(\fa)$ and $\theta=\on{Int}(n_{w})\circ\vartheta$. We have $I=T^\theta$. In each case we choose a specific $w$ and display the partition $R=\sqcup S_i$ in Proposition~\ref{prop-distinguished}, the distinguished reflections $t_i$, and $I_i:=I_{S_i}=I_{t_i}$, in the form described in~\S\ref{ssec-cyclotomic} and~\S\ref{ssec-noncyclo}, explicitly. This allows us to determine $\hat I/W$, $W_\chi^0$, $W_\chi$, $\chi|_{I_{i}}$ and hence the polynomials $R_{\chi,t_i}$, using the strategy described in~\S\ref{ssec-str-red}, together with Proposition~\ref{prop-rankone} and Table~\ref{table-rd1}.

\subsection{The gradings with trivial $I$}Suppose that $I=1$.  By Proposition~\ref{prop-rankone} and Table~\ref{table-rd1}, we have the following:

\subsubsection{$(^3D_4,6_\rs)$, $(F_4,6_\rs)$, $(^2E_6,6_\rs)$, and $(E_8,6_\rs)$} We have $R_{s}=\Phi_1^2\Phi_2$ since $\theta|_{(G_s)_{\on{der}}}\cong({^2}A_2,6_\rs)$.
 
\subsubsection{$(E_8,12_\rs)$} We have two conjugacy classes of distinguished reflections in $W$. For the distinguished reflections of order 4, we have $\theta|_{(G_s)_{\on{der}}}\cong({^3}D_4,12_\rs)$ so $R_s=\Phi_1^3\Phi_2$.  For the distinguished reflections of order 3, we have $\theta|_{(G_t)_{\on{der}}}\cong({^2}A_2,6_\rs)$ so $R_t=\Phi_1^2\Phi_2$.

\subsubsection{$(E_8,10_\rs)$} We have  $R_{s}=\Phi_1^3\Phi_2^2$ since $\theta|_{(G_s)_\der}\cong(^2A_5,10_\rs)$.

\subsection{The gradings with non-trivial $I$}We state the results on $\hat I/W, W_\chi^0, W_\chi$  as Lemma~\ref{lem-case1}-Lemma~\ref{lem-3se8}, whose proofs are entirely similar. We will only prove a couple of them in more detail. For the rest, we display the necessary information and leave most of the details to the reader. In what follows we use the Coxeter elements $w_h\in W_G$ chosen in~\S\ref{ssec-dynkin}.

\subsubsection{$(^3D_4,3_\rs)$}Let $\alpha_i=e_i-e_{i+1}$, $i=1,2,3$ and $\alpha_4=e_3+e_4$ be a set of simple roots of $D_4$. We can assume that $\theta|_R=(w\circ\vartheta)^4$, where $w=s_{\alpha_1}s_{\alpha_2}$ and $\vartheta:\alpha_1\mapsto\alpha_4\mapsto\alpha_3,\alpha_2\mapsto\alpha_2$. We have $S_i=\pm\{\beta_i,\theta\beta_i,\theta^2\beta_i\}$, $i=1,\ldots,4$ where 
{\small\bern
\{\beta_i,\theta\beta_i\}=\{1000,0111\},\{0100,-0101\},\{0010,-1110\},\{0110,\,-1211\}.\eern}

\begin{lemma}\label{lem-case1}
We have
$$I=\langle\gamma_1:=\check\alpha_2(\xi_3)\check\alpha_4(\xi_3^2),\,\gamma_2=\check\alpha_1(\xi_3^2)\check\alpha_3(\xi_3)\check\alpha_4(\xi_3^2)\rangle\cong\mu_3^2\,.$$
The action of $W$ on $\hat I$ has two orbits with representatives $\chi_0,\chi_1:(\gamma_1,\gamma_2)\mapsto(\zeta_3,1)$ and
\beqn
W_{\chi_1}= W_{\chi_1}^0=\langle t_4=s_{0110}s_{1211}\rangle\cong\mu_3.
\eeqn
\end{lemma}
\begin{proof}
Since $\theta:\check\alpha_1\mapsto\sum_{i=2}^4\check\alpha_i,\check\alpha_2\mapsto-\check\alpha_2-\check\alpha_4,\check\alpha_3\mapsto-\sum_{i=1}^3\check\alpha_i,\check\alpha_4\mapsto\check\alpha_2$,
we have $I=T^\theta$ as claimed. 
Using~\eqref{action-m=3} one checks that
\begin{align}\label{eqn-I1}
I_{i}=\langle\nu_i\rangle,\ \nu_1= \gamma_1\gamma_2,\ \nu_2=\gamma_1,\ \nu_3=\gamma_1^2\gamma_2,\ \nu_{4}=\gamma_2
\end{align} 
\begin{align*}
\text{ and }\quad&(t_i(\gamma_1),t_i(\gamma_2))=(\gamma_1^2\gamma_2,\gamma_1^2),\ (\gamma_1,\gamma_1^2\gamma_2), (\gamma_2,\gamma_1^2\gamma_2^2),\ (\gamma_1\gamma_2,\gamma_2)\,.
\end{align*}
It follows that $W_{\chi_1}^0=\langle t_4\rangle$ and $W.\chi_1=\hat I\setminus\{\chi_0\}$. Hence $|W_{\chi_1}|=|W|/8=3$ which implies that $W_{\chi_1}=W_{\chi_1}^0$.
\end{proof}

It follows from~\eqref{eqn-I1} that $\chi_1|_{I_{4}}=1$ and $\chi_1|_{I_{i}}$ is of order $3$, $i=1,2,3$. 
Since $\theta|_{(G_{t_i})_{\on{der}}}\cong(A_2,3_s)$, we conclude that $R_{\chi,t_i}=(z-1)^3$ if $t_i\in W_\chi$ and  $R_{\chi,t_i}=z^3-1$ otherwise.

\subsubsection{$(F_4,4_\rs)$}

Let $w=w_h^3$.  
We have $S_i=\pm\{\delta_i,\theta\delta_i\}\cup\pm\{\beta_i,\theta\beta_i\}$, $i=1,\ldots,6$, where
\begin{align*}
(\delta_i)=1000,0100,1100,1120,1220,0122,\ (\beta_i)=0111,0010,0121,0001,0011,1110.
\end{align*}

\begin{lemma}We have
$$I=\langle\gamma_1:=\check\alpha_3(-1),\,\gamma_2:=\check\alpha_4(-1)\rangle\cong\mu_2^2.$$
The action of $W$ on $\hat I$ has $2$ orbits with representatives $\chi_0,\chi_1:(\gamma_1,\gamma_2)\mapsto(-1,1)$ and 
\beqn
W_{\chi_1}=W_{\chi_1}^0=\langle t_3,t_4,t_1^2,t_2^2,t_5^2,t_6^2\rangle\cong G_{4,1,2}.
\eeqn
\end{lemma}
We have
$
w:\check\alpha_1\mapsto\check\alpha_1+2\check\alpha_2+\check\alpha_3+\check\alpha_4,\ \check\alpha_2\mapsto\check\alpha_2+\check\alpha_3,\,\check\alpha_3\mapsto-2\check\alpha_2-\check\alpha_3,\,\check\alpha_4\mapsto-2\check\alpha_1-2\check\alpha_2-2\check\alpha_3-\check\alpha_4.
$  
Let $\Gamma_i=(t_i(\gamma_1),t_i(\gamma_2))$. Using~\eqref{action-m=4-F4} one checks that
\begin{align*}
&I_{1}=I_{5}=\langle\gamma_1\gamma_2\rangle,&& I_{2}=I_{6}=\langle\gamma_1\rangle,&&I_{3}=I_{4}=\langle\gamma_2\rangle\,\\
&\Gamma_1=\Gamma_5=(\gamma_2,\gamma_1)&& \Gamma_2=\Gamma_6=(\gamma_1
,\gamma_1\gamma_2)&& \Gamma_3=\Gamma_4=(\gamma_1\gamma_2,\gamma_2).
\end{align*}
It follows that $\chi_1|_{I_{t_i}}=1$, $i=3,4$, and $\chi_1|_{I_{t_i}}$ is of order $2$, $i=1,2,5,6$.
Since $\theta|_{(G_{t_i})_\der}\cong(B_2,4_\rs)$, $i=1,\ldots,6$, we conclude that $R_{\chi_0,t_i}=R_{\chi_1,t_3}=R_{\chi_1,t_4}=(z-1)^3(z+1)$, and $R_{\chi_1,t_i}=(z^2-1)^2$, $i=1,2,5,6$.

\subsubsection{$(F_4,3_\rs)$} 
Let $w=w_h^4$. 
We have $S_i=\pm\{\beta_i,\theta\beta_i,\theta^2\beta_i\}$, $i=1,\ldots,8$, where
{\small\beqn
(\beta_i)= 1000 ,\,  0100 ,\,  1100 ,\, 1220 ,\,  0010 ,\,  0001 ,\, 0011 ,\,  0121 .
\eeqn}
 
\begin{lemma}We have
$$I=\langle\gamma_1:=\check\alpha_1(\xi_3^2)\check\alpha_2(\xi_3)\check\alpha_3(\xi_3),\ \gamma_2:=\check\alpha_1(\xi_3^2)\check\alpha_2(\xi_3^2)\check\alpha_4(\xi_3)\rangle\cong\mu_3^2.$$
The action of $W$ on $\hat I$ has $2$ orbits with representatives $\chi_0,\chi_1:(\gamma_1,\gamma_2)\mapsto(\zeta_3, 1)$ and 
\beqn
W_{\chi_1}=W_{\chi_1}^0=\langle t_4\rangle\times\langle t_5\rangle\cong\mu_3\times \mu_3.
\eeqn
\end{lemma}
We have
$
w:\check\alpha_1\mapsto\check\alpha_2+\check\alpha_3,\ \check\alpha_2\mapsto\check\alpha_1+\check\alpha_2+\check\alpha_3+\check\alpha_4,\,\check\alpha_3\mapsto-2\check\alpha_1-2\check\alpha_2-2\check\alpha_3-\check\alpha_4,\,\check\alpha_4\mapsto-2\check\alpha_2-\check\alpha_3-\check\alpha_4.
$
 Using~\eqref{action-m=3}, one checks that
 \bern
I_{i}=\langle\nu_i\rangle&& \nu_i=\gamma_1,\gamma_1\gamma_2,\gamma_1^2\gamma_2,\gamma_2,\gamma_2^2,\gamma_1^2\gamma_2,\gamma_1^2,\gamma_1\gamma_2\\
&&t_i(\gamma_1)=\gamma_1,\gamma_2^2,\gamma_1^2\gamma_2^2,\gamma_1\gamma_2,\gamma_1\gamma_2,\gamma_2,\gamma_1,\gamma_1^2\gamma_2\\&&
t_i(\gamma_2)=\gamma_1\gamma_2,\gamma_1\gamma_2^2,\gamma_1,\gamma_2,\gamma_2,\gamma_1^2\gamma_2^2,\gamma_1^2\gamma_2,\gamma_1^2.
\eern
It follows that $\chi_1|_{I_{t_i}}$ is of order $3$, if $t_i\notin W_{\chi_1}^0$ and $\chi_1|_{I_{t_i}}=1$ otherwise. 
Since $\theta|_{(G_{t_i})_\der}\cong(A_2,3_\rs)$ we conclude that $R_{\chi,t_i}=z^3-1$, if $t_i\notin W_{\chi}^0$,  and $R_{\chi,t_i}=(z-1)^3$, if $t_i\in W_{\chi}^0$.

\subsubsection{$(E_6,6_\rs)$}
Let $w=w_h^2$. Then $$(w\alpha_i)_{i=1,\ldots,6}=010111,-101111,-011211,112321,-111210,111100.$$  
\begin{lemma}We have
\beqn
I= Z_G\cong\mu_3.
\eeqn
The little Weyl group $W$ acts on $\hat I$ trivially. 
\end{lemma}
By the discussion in~\S\ref{sssec-e6-6s}, we have $S_i=\pm\{\beta_i,\theta\beta_i,\theta^2\beta_i\}$, $i=1,\ldots,4$, and $S_i=\pm\{\theta^j\beta_i,j=0,\ldots,5\}$, $i=5,\ldots,8$, where
{\small\begin{align*}
(\beta_i)=& 010000 , 000100 , 010100 , 011110 , 100000 , 001000 , 001100 , 101100 \,.
\end{align*}}
One checks that 
$
I_{i}=1,\,i=1,\ldots,4,\ I_{i}=Z_G,\,i=5,\ldots,8.
$
It follows that $\chi_j|_{I_{t_i}}=1$, $i=1,2,3,4$, and $\chi_j|_{I_{t_i}}$ is of order $3$, $i=5,\ldots,8$, for $j=1,2$.
We have
$\theta|_{(G_{t_i})_{\on{der}}}\cong({}^2A_2,6_\rs), i=1,2,3,4,$ and $\theta|_{(G_{t_i})_{\on{der}}}\cong(A_2,3_\rs),\, i=5,6,7,8$.  So
$R_{\chi_j,t_i}=(z-1)^2(z+1),$ $i=1,\ldots,4,\,j=0,1,2$, and $R_{\chi_0,t_i}=(z-1)^3, R_{\chi_j,t_i}=z^3-1,\,i=5,\ldots,8,\,j=1,2.$

Note that $\langle t_i\mid\chi_1|_{I_{t_i}}=1\rangle=\langle t_i,i=1,2,3,4\rangle\cong G_4\cong W_{\chi}^{\en}$.

\subsubsection{$(E_6,3_\rs)$}Let $w=w_h^4$. Then
\beqn
(w\alpha_i)_{i=1,\ldots,6}=011100,-111111,-112210,112221,-011221,010110.
\eeqn
We have $S_i=\pm\{\beta_i,\theta\beta_i,\theta^2\beta_i\}$, $i=1,\ldots,12$, where
{\small\begin{align*}
(\beta_i)= &100000 , 010000 , 001000 , 000100 , 000010 , 000001 ,\\& 101000 , 010100 , 001100 , 101100 , 000111 , 011110\, .
\end{align*}}
\begin{lemma}
We have
\beqn
I\cong\langle\gamma_1,\gamma_2,\gamma_3\rangle\cong\mu_3^3
\eeqn
where
\beqn
\gamma_1=\check\alpha_1(\xi_3)\check\alpha_2(\xi_3)\check\alpha_3(\xi_3)\check\alpha_6(\xi_3^2),\ \gamma_2=\check\alpha_1(\xi_3)\check\alpha_4(\xi_3)\check\alpha_6(\xi_3),\ \gamma_3=\check\alpha_1(\xi_3^2)\check\alpha_2(\xi_3)\check\alpha_5(\xi_3)\check\alpha_6(\xi_3).
\eeqn
Define $\chi_i\in\hat I$ by
\beqn
\chi_1:(\gamma_1,\gamma_2,\gamma_3)\mapsto(\zeta_3,1,1),\,\chi_2:(\gamma_1,\gamma_2,\gamma_3)\mapsto(1,\zeta_3,1)\,\chi_3:(\gamma_1,\gamma_2,\gamma_3)\mapsto(1,1,\zeta_3).
\eeqn
The $W$ action on $\hat I$ has four orbits with representatives $\chi_0,\chi_1,\chi_2,\chi_3$. Moreover
\begin{align*}
&W_{\chi_1}^0=\langle t_{10},t_{12}\rangle\cong G_4,\ W_{\chi_2}^0=\langle t_8,t_{10},t_{11}\rangle\cong (\mu_3)^3,\ \ W_{\chi_3}^0=\langle t_{11},t_{12}\rangle\cong G_4,\\
&W_{\chi_i}=\langle\sigma_i\rangle\ltimes W_{\chi_i}^0\cong\mu_3\ltimes W_{\chi_i}^0,\,i=1,2,3.
\end{align*}

\end{lemma}
Using~\eqref{action-m=3}, one checks that $I_{i}=\langle\nu_i\rangle$
\begin{align*}
&\nu_i=\gamma^{110},\gamma^{222},\gamma^{012},\gamma^{212},\gamma^{210},\gamma^{011},\gamma^{122},\gamma^{101},\gamma^{221},\gamma^{001},\gamma^{100},\gamma^{020}
\end{align*}
and
\begin{align*}
t_i(\gamma_1)&=\gamma^{020},\gamma^{022},\gamma^{121},\gamma^{221},\gamma^{220},\gamma^{122},\gamma^{222},\gamma^{100},\gamma^{021},\gamma^{100},\gamma^{100},\gamma^{120}\\
t_i(\gamma_2)&=\gamma^{120},\gamma^{202},\gamma^{001},\gamma^{222},\gamma^{100},\gamma^{021},\gamma^{010},\gamma^{212},\gamma^{010},\gamma^{011},\gamma^{110},\gamma^{010}\\
t_i(\gamma_3)&=\gamma^{221},\gamma^{220},\gamma^{022},\gamma^{122},\gamma^{121},\gamma^{020},\gamma^{120},\gamma^{001},\gamma^{222},\gamma^{001},\gamma^{001},\gamma^{021}
\end{align*}
where we have written $\gamma^{abc}=\gamma_1^a\gamma_2^b\gamma_3^c$. 
We give more detail in this case. The little Weyl group $W=G_{25}$. We can choose $t_{10}, t_{11},t_{12}$ as generators of $W$, where $t_{10}t_{12}t_{10}=t_{12}t_{10}t_{12}$, $t_{11}t_{12}t_{11}=t_{12}t_{11}t_{12}$ and $t_{11}t_{10}=t_{10}t_{11}$. This can be checked as follows. The Cartan subspace $\fa$ has a basis
$a_1=\zeta_3 h_{\beta_{10}}-\zeta_3^2h_{\theta\beta_{10}}$, $a_2=\zeta_3 h_{\beta_{11}}-\zeta_3^2h_{\theta\beta_{11}}$, $a_3=\zeta_3 h_{\beta_{12}}-\zeta_3^2h_{\theta\beta_{12}}$. With respect to this basis we have
\beqn
t_{10}=\begin{pmatrix}\zeta_3&0&-\zeta_3\\0&1&0\\0&0&1\end{pmatrix}\quad t_{11}=\begin{pmatrix}1&0&0\\0&\zeta_3&-\zeta_3\\0&0&1\end{pmatrix}\quad t_{12}=\begin{pmatrix}1&0&0\\0&1&0\\1&1&\zeta_3\end{pmatrix}\,.
\eeqn
Moreover
\begin{align*}
&t_1=t_{11}t_{12}t_{11}^{-1}&&t_2=t_{10}^{-1}t_{11}t_{12}t_{10}t_{12}^{-1}t_{11}^{-1}t_{10}&&t_3=t_{12}t_{10}t_{12}^{-1}\\
&t_4=t_{12}t_{10}^{-1}t_{12}t_{11}t_{12}^{-1}t_{10}t_{12}^{-1}&
&t_5=t_{12}t_{11}t_{12}^{-1}&&t_6=t_{10}t_{12}t_{10}^{-1}\\
&t_7=t_{10}t_{12}t_{11}t_{12}^{-1}t_{10}^{-1}&&t_8=t_{12}t_{10}^{-1}t_{12}t_{11}t_{12}^{-1}t_{10}t_{12}^{-1}&&t_9=t_{11}t_{12}t_{10}t_{12}^{-1}t_{11}^{-1}\,.
\end{align*}
One now checks that
\begin{align*}
&W_{\chi_1}^0=\langle t_3,t_6,t_{10},t_{12}\rangle=\langle t_{10},t_{12}\rangle\cong G_4,\ \ W_{\chi_3}^0=\langle t_1,t_5,t_{11},t_{12}\rangle=\langle t_{11},t_{12}\rangle\cong G_4,\\& W_{\chi_2}^0=\langle t_8,t_{10},t_{11}\rangle\cong (\mu_3)^3,\ 
W_{\chi_i}=\langle\sigma_i\rangle\ltimes W_{\chi_i}^0,\,i=1,2,3
\end{align*}
where we can take $\sigma_1=t_{11}t_{12}t_{10}^2t_{12}t_{11}t_{12}t_{10}$, $\sigma_3=t_{10}t_{12}t_{11}^2t_{12}t_{10}t_{12}t_{11}$ and $\sigma_2=t_{12}^2t_{10}^2t_{11}t_{12}$. One checks that each $\sigma_i$ has order 3, $\sigma_1t_{12}\sigma_1^{-1}=t_{10}$, $\sigma_3t_{12}\sigma_3^{-1}=t_{11}$, $\sigma_2t_{8}\sigma_2^{-1}=t_{11}$, and  $\sigma_2t_{11}\sigma_2^{-1}=t_{10}$.

We have $\theta|_{(G_{t_i})_\der}\cong(A_2,3_\rs)$, $i=1,\ldots,12.$ Moreover, $\chi|_{I_{t_i}}=1$ iff $t_i\in W_\chi^0$. So $R_{\chi,t_i}=(z-1)^3$ if $t_i\in W_\chi^0$ and $R_{\chi,t_i}=z^3-1$ otherwise.

\subsubsection{$(^2E_6,4_\rs)$}  Let $w\vartheta=-w_h^3$. Then
\beqn
(w\vartheta\alpha_i)_{i=1,\ldots,6}=-101110,011210,111221,-122321,112211,-001111.
\eeqn
We have $S_i=\pm\{\theta^j\beta_i,j=0,\ldots,3\}\cup\pm\{\delta_i,\theta\delta_i\}$, $i=1,\ldots,6$, where $\delta_i=\beta_i+\theta\beta_i$ and 
{\small\beqn
(\beta_i)= 011111 , 101100 , 001000 , 001100 , 101000 ,
 100000 \,.
\eeqn}
\begin{lemma}
We have
\beqn
I=\langle\gamma_1:=\check\alpha_1(\xi_4)\check\alpha_4(\xi_4^2)\check\alpha_6(\xi_4^3),\,\gamma_2:=\check\alpha_2(\xi_4^2)\check\alpha_3(\xi_4)\check\alpha_5(\xi_4^3)\rangle\cong\mu_4^2.
\eeqn 
Define $\chi_i\in\hat I$, $i=1,2$ by
\beqn
\chi_1:(\gamma_1,\gamma_2)\mapsto(-1,-1),\ \chi_2:(\gamma_1,\gamma_2)\mapsto(\zeta_4,1).
 \eeqn
The $W$-action on $\hat I$ has $3$ orbits with representatives $\chi_i$, $i=0,1,2$. Moreover
\bern
&&W_{\chi_1}=W_{\chi_1}^0=\langle t_2,t_5,t_1^2,t_3^2,t_4^2,t_6^2\rangle\cong G_{4,1,2} 
\qquad W_{\chi_2}=W_{\chi_2}^0=\langle t_4\rangle\times\langle t_3^2\rangle\cong\mu_4\times\mu_2.
\eern
\end{lemma}

Using~\eqref{action=m=4-2e6} one checks that
{\small\begin{align*}
&I_{i}=\langle\nu_i\rangle\quad\nu_i=\gamma_1,\,\gamma_1\gamma_2^3,\, \gamma_1^2\gamma_2,\,\gamma_2,\,\gamma_1^3\gamma_2^3,\,\gamma_1\gamma_2^2\,\\
&t_i(\gamma_1)=\gamma_1,\,\gamma_1^2\gamma_2^3,\,\gamma_1^3\gamma_2^3,\,\gamma_1\gamma_2^3,\,\gamma_2^3,\,\gamma_1^3\qquad
t_i(\gamma_2)=\gamma_1\gamma_2,\, \gamma_1,\,\gamma_2^3,\, \gamma_2,\,\gamma_1\gamma_2^2,\,\gamma_1\gamma_2^3\,.
\end{align*}}
We have $\theta_{(G_{t_i})_\der}\cong(A_3,4_\rs)$. So $R_{\chi_0,t_i}=(z-1)^4$ for all $i$. Moreover $\chi_1|_{I_{t_i}}=1$ for $i=2,5$ and $\chi_1|_{I_{t_i}}$ is of order 2 for $i=1,3,4,6$. So $R_{\chi_1,t_i}=(z-1)^4$, $i=2,5$ and $R_{\chi_1,t_i}=(z^2-1)^2$, $i=1,3,4,6$. Similarly, $\chi_2|_{I_{t_4}}=1$, $\chi_2|_{I_{t_3}}$ is of order 2, and  $\chi_2|_{I_{t_i}}$ is of order 4, $i=1,2,5,6$.  So $R_{\chi_2,t_4}=(z-1)^4$, $R_{\chi_2,t_3}=(z^2-1)^2$ and $R_{\chi_2,t_i}=z^4-1$, $i=1,2,5,6$.

\subsubsection{$(E_7,6_\rs)$}Let $w=w_h^3$. Then {\small$$(w\alpha_i)_{i=1,\ldots,7}=1011111,-0112111,-1112211, 1224321, \-0001110,-1122100,0011100.$$}
By~\S\ref{sssec-e76s}, we have $S_i=\pm\{\beta_i,\theta\beta_i,\theta^2\beta_i\}\cong A_2$, $i=1,\ldots,12$, and $S_i=\pm\{\beta_i,\theta\beta_i,\theta^2\beta_i\}\cong A_1^3$, $i=13,\ldots,21$, where
{\small\begin{align*}
(\beta_i)=& 0112100 , 1112100 , 0000111  , 0000010, 0011000 , 0111110 , 0001000 ,\\& 0011000 , 1010000 , 1000000 , 1011000 , 0010000 , 1122210 , 1112210 ,\\& 0112210 , 0000110 , 0111100 , 1111000 , 1011100 , 0111000 , 0001110 .
\end{align*}}
One checks that $I_{t_i}=1$, $i=1,\ldots, 12$, and $I_{t_i}=Z_{G}\cong\mu_2$, $i=13,\ldots,21$.

\begin{lemma}
We have $$I=\langle\gamma:=\check\alpha_2(-1)\check\alpha_5(-1)\check\alpha_7(-1)\rangle=Z_G\cong\mu_2.$$
Define $\chi_1\in\hat I$ by $\chi_1(\gamma)=-1$. The action of $W$ on $\hat I$ has two orbits with representatives $\chi_0,\chi_1$. Moreover,
$W_{\chi_1}=W_{\chi_1}^0=W\cong G_{26}$.
\end{lemma}

We have $\theta|_{(G_{t_i})_\der}\cong(^2A_2,6_\rs)$, $i=1,\ldots,12$ and $\theta|_{(G_{t_i})_\der}\cong(A_1,2_\rs)$, $i=13,\ldots,21$. Moreover, $\chi_1|_{I_{t_i}}=1$, $i=1,\ldots, 12$, and $\chi_1|_{I_{t_i}}$ is of order $2$, $i=13,\ldots, 31$. So
$
R_{\chi_0,t_i}=R_{\chi_1,t_i}=(z-1)^2(z+1),\,i=1,\ldots,12,$ and $R_{\chi_0,t_i}=(z-1)^2, R_{\chi_1,t_i}=z^2-1, i=13,\ldots,21.
$

Note that $\langle t_i\mid\chi_1|_{I_{t_i}}=1\rangle=\langle t_i,\,i=1,\ldots,12\rangle\cong G_{25}\cong W_{\chi}^{\en}$.

\subsubsection{$(E_8,8_\rs)$}\label{ssec-e8-8}
 
Let $w=s_{e_5-e_6}s_{e_1-e_4}s_{e_2-e_3}s_{e_7-e_8}s_{e_4-e_5}s_{e_1-e_2}s_{e_1+e_2}s_{e_3-e_7}.$
Then
{\tiny\bern
(w\alpha_i)_{i=1,\ldots,8}=-12233210,-01121000,00001000,23454321,-22343211, -00111110, 00111100,-00011100.
\eern}
We have $S_i=\pm\{\theta^j\beta_i,j=0,\ldots,3\}\cong A_1^4$, $i=1,\ldots,12$, and $S_i=\pm\{\theta^j\beta_i,j=0,\ldots,3\}\cup\pm\{\theta^j\delta_i,j=0,\ldots,3\}\cup\{\theta^j\eta_i,j=0,\ldots,3\}\cong D_4$ with $\delta_i=\beta_i+\theta\beta_i-\theta^3\beta_i$, $\eta_i=\beta_i+\theta\beta_i$, $i=13,\ldots,18$, where
{\tiny\begin{align*}
(\beta_i)=& {10000000} , {00000001} , {10100000} , {00110000}, {00011000} , {00001110} , {11110000},{10111000},  {11121000} , \\
&  {11121110}  , {11232100} , {11222211} , {00000010} , {01110000} ,
 {11111000} , {10111100}, {10111110} ,
 {11221100} .
\end{align*}}

\begin{lemma}
We have
$$I=\langle\gamma_1:=\check\alpha_2(-1)\check\alpha_3(-1),\,\gamma_2:=\check\alpha_5(-1)\check\alpha_7(-1) \check\alpha_8(-1)\rangle\cong\mu_2^2\,.$$
The action of $W_\fa\cong G_9$ on $\hat I$ has $2$ orbits with representatives $\chi_0,\chi_1:(\gamma_1,\gamma_2)\mapsto(-1,1)$ and
\bern
 W_{\chi_1}=W_{\chi_1}^0=\langle t_1,t_3,t_9,t_{12},t_{13}^2,t_{14}^2,t_{15}^2,t_{16}^2,t_{17},t_{18}\rangle\cong G_{8,2,2}.
\eern
\end{lemma}
Using~\eqref{action-m=8-E8}, one checks that
{\small\bern
&&\nu_i=\gamma_1, i=2,4,5,6,13,14,\, \nu_i=\gamma_2, i=1,3,9,12,17,18,\, \nu_i=\gamma_1\gamma_2, i=7,8,10,11,15,16\\
&&t_{i}(\gamma_j)=\gamma_j \text{ if $I_{i}=\langle\gamma_j\rangle$},\ t_{i}(\gamma_j)=\gamma_j \gamma\text{ if $I_{i}=\langle\gamma\rangle$ and $\gamma\neq \gamma_j$.}
\eern}
We have $\theta|_{(G_{t_i})_{\on{der}}}\cong(A_1,2_\rs)$, for $i=1,\ldots,12$ and $\theta|_{(G_{t_i})_{\on{der}}}\cong(^2D_4,8_\rs)$ for $i=13,\ldots,18$. 
Moreover, $\chi_1|_{I_{t_i}}=1$, $i=1,3,9,12,17,18$ and $\chi_1|_{I_{t_i}}$ nontrivial otherwise. So
$R_{\chi_0,t_i}=(z-1)^2,\,i=1,\ldots,12,\,R_{\chi_0,t_i}=(z-1)^4,\,i=13,\ldots,18$, $
R_{\chi_1,t_i}=(z-1)^2,\,i=1,3,9,12,\, R_{\chi_1,t_i}=z^2-1,i=2,4,5,6,7,8,10,11$, 
$R_{\chi_1,t_i}=(z^2-1)^2,\,i=13,14,15,16,\ R_{\chi_1,t_i}=(z-1)^4,\,i=17,18\,.
$

\subsubsection{$(E_8,5_\rs)$}Let $w=w_h^{6}$.
Then 
{\tiny$$
(w\alpha_i)_{i=1,\ldots,8}=11221111,\,-12232210,\,-22343211,\,23465421,\,-12354321,\, 12243221,\,-11122221,\,10111110\,.
$$}
 We have $S_i=\pm\{\theta^j\beta_i,j=0,\ldots,4\}\cup\pm\{\theta^j\delta_i,j=0,\ldots,4\}, \delta_i=\beta_i+\theta\beta_i$, $i=1,\ldots,12$ and
 {\small \begin{align*}
(\beta_i)=& {01110000} ,
 {01000000} ,  {11222211} ,
 {00000100} ,  {11121110} ,
 {00111000} , \\&{00011111} ,
 {10111000} , {12232111} ,
 {01122221} , {12232100} ,
 {00001100} \,.
 \end{align*}}
  
 \begin{lemma}
 We have $$I=\langle\gamma_1:=\check\alpha_1(\zeta_5^3)\check\alpha_4(\zeta_5)\check\alpha_6(\zeta_5^2)\check\alpha_8(\zeta_5),\gamma_2:=\check\alpha_2(\zeta_5^2)\check\alpha_3(\zeta_5^3)\check\alpha_5(\zeta_5)\check\alpha_7(\zeta_5)\rangle\cong\mu_5^2.$$
 The action of $W$ on $\hat I$ has $2$ orbits with representatives $\chi_0,\chi_1:(\gamma_1,\gamma_2)\mapsto (\zeta_5,1)$ and
 \beqn
W_{\chi_1}=W_{\chi_1}^0=\langle t_7\rangle\times \langle t_{12}\rangle\cong \mu_5\times\mu_5.
 \eeqn
 \end{lemma}

 Using~\eqref{action-m=5}, one checks that
 {\small\begin{align*}
 &I_{1}=I_{4}=\langle\gamma_1^{3}\gamma_2\rangle,\, I_{2}=I_{3}=\langle\gamma_1^{4}\gamma_2\rangle,\,I_{5}=I_{8}=\langle\gamma_1\gamma_2\rangle,\\&\,I_{6}=I_{{11}}=\langle\gamma_1\rangle,\,
 I_{7}=I_{{12}}=\langle\gamma_2\rangle,\,I_{9}=I_{{10}}=\langle\gamma_1^{2}\gamma_2\rangle\,;\\
& t_i(\gamma_1)=\gamma^{34},\gamma^{01},\gamma^{24},\gamma^{41},\gamma^{04},\gamma^{10},\gamma^{14},\gamma^{21},\gamma^{31},\gamma^{44},\gamma^{10},\gamma^{12}\\
 &t_i(\gamma_2)=\gamma^{44},\gamma^{42},\gamma^{10},\gamma^{13},\gamma^{12},\gamma^{11},\gamma^{01},\gamma^{40},\gamma^{14},\gamma^{43},\gamma^{41},\gamma^{01}
\end{align*}}
where $\gamma^{ab}:=\gamma_1^a\gamma_2^b$. 
  We have $\theta|_{(G_{t_i})_{\on{der}}}\cong(A_4,5_\rs)$ and $\chi_1|_{I_{t_i}}=1$, $i=7,12$. So $R_{\chi,t_i}=(z-1)^5$ if $t_i\in W_\chi$, and $R_{\chi,t_i}=z^5-1$ otherwise.
 
\subsection{$(E_8,4_\rs)$}
Let $w=(s_{e_5-e_6}s_{e_1-e_4}s_{e_2-e_3}s_{e_7-e_8}s_{e_4-e_5}s_{e_1-e_2}s_{e_1+e_2}s_{e_3-e_7})^2.$
Then
{\tiny\bern
(w\alpha_i)_{i=1,\ldots, 8}=11221000,-23465431,-223465432,23465432,-00000011, -01122100, 01011000,-01010000\,.
\eern}
We have $S_i=\pm\{\beta_i,\theta\beta_i\}$, $i=1,\ldots,60$, where
{\tiny \begin{align*}
(\beta_i)=& {10000000} ,\, {01000000} , {00100000} , {00010000}, {00001000} , {00000100} , {00000010} , {00000001} , {10100000} , {{00110000}} ,\\& {{00011000}} , {{00001100}}, {{00000110}} , {{10110000}} , {{01110000}} , {{00111000}},
 {{00011100}} , {{00001110}} , {{00000111}} , {{11110000}} \\& {{10111000}} , {{01111000}} , {{01011100}} , {{00011110}}, {{00001111}} , {{11111000}} , {10111100} , {01111100} , {01011110} , {00011111} ,\\& {01121000} , {11111100}, {10111110} , {01011111} , {11121000} , {01111110} , {11111110} , {10111111} , {01111111} , {11121100} ,\\& {11111111} , {11221100}, {11121110} , {11122110}, {11222100} , {{11121111}} , {{11232100}} , {{11232110}} , {{11222210}} , {{11222211}} , \\&{{11232210}} , {{11232111}}, {{11222221}} , {{12232210}} , {{11232221}} , {{12232211}} , {{12233210}} , {{12233211}} , {{12233321}} , {{22354321}} \,.
\end{align*}}
 Let
{\tiny\bern
&&J_1:=\{6,25,29,60\},\, J_2:=\{33,42,46,59\},\  J_3:=\{20,21,53,58\},\,J_4:=\{2,3,5,31\},\  J_5:=\{1,50,54,55\},\\&&  J_6:=\{27,41,44,52\},\,J_7:=\{12,19,24,36\},\,  J_8:=\{13,23,30,39\},\,J_9:=\{37,40,45,48\},\,J_{10}:=\{14,26,49,56\},\\
&&J_{11}:=\{4,7,15,16\},\, J_{12}:=\{9,35,51,57\},\, J_{13}:=\{32,38,43,47\},\,\, J_{14}:=\{17,18,28,34\}, J_{15}:=\{8,10,11,12\}.
\eern}
\begin{lemma}
We have
\begin{align*}
I=&\langle\gamma_1:=\check\alpha_2(-1)\check\alpha_3(-1),\,\gamma_2:=\check\alpha_3(-1)\check\alpha_5(-1),\\&\gamma_3:=\check\alpha_4(-1)\check\alpha_7(-1),\,\gamma_4:=\check\alpha_5(-1)\check\alpha_7(-1) \check\alpha_8(-1)\rangle\cong\mu_2^4\,.
\end{align*}
The action of  $W$ on $\hat I$ has $2$ orbits with representatives $\chi_0,\chi_1:(\gamma_1,\gamma_2,\gamma_3,\gamma_4)\mapsto(-1,1,1,1)$ and 
\bern
&&W_{\chi_1}=W_{\chi_1}^0=\langle t_i,i\in \cup_{a\in{2,3,4,8,9,10,14}}J_{a}\rangle\cong G_{4,2,4}.
\eern
\end{lemma}
Let $\Gamma_i=(t_i(\gamma_1),t_i(\gamma_2),t_i(\gamma_3),t_i(\gamma_4))$. Using~\eqref{action-m=4-E8} one checks that
{\tiny\bern
&&I_i=\langle\gamma_1\rangle,\,i\in J_1,\,\,I_i=\langle\gamma_2\rangle,\,i\in J_2,\ I_i=\langle\gamma_3\rangle,\,i\in J_3,\,I_i=\langle\gamma_4\rangle,\,i\in J_4,\ I_i=\langle\gamma_1\gamma_2\rangle,\,i\in J_5,\,\,I_i=\langle\gamma_1\gamma_3\rangle,\,i\in J_6\\&&I_i=\langle\gamma_1\gamma_4\rangle,\,i\in J_7,\,\,I_i=\langle\gamma_2\gamma_3\rangle,\,i\in J_8,\,I_i=\langle\gamma_2\gamma_4\rangle,\,i\in J_9,\,I_i=\langle\gamma_3\gamma_4\rangle,\,i\in J_{10},\,I_i=\langle\gamma_1\gamma_2\gamma_3\rangle,\,i\in J_{11},\\&&I_i=\langle\gamma_1\gamma_2\gamma_4\rangle,\,i\in J_{12},\,I_i=\langle\gamma_1\gamma_3\gamma_4\rangle,\,i\in J_{13},\,\,I_i=\langle\gamma_2\gamma_3\gamma_4\rangle,\,i\in J_{14},\,I_i=\langle\gamma_1\gamma_2\gamma_3\gamma_4\rangle,\,i\in J_{15},
\eern}
and
{\tiny\begin{align*}
&\Gamma_i=(\gamma_1,\gamma_1\gamma_2,\gamma_1\gamma_3,\gamma_4),\,i\in J_1&&\Gamma_i=(\gamma_1\gamma_2,\gamma_2,\gamma_2\gamma_3,\gamma_4),\,i\in J_2&&
\Gamma_i=(\gamma_1\gamma_3,\gamma_2\gamma_3,\gamma_3,\gamma_3\gamma_4),\,i\in J_3,\\&\Gamma_i=(\gamma_1,\gamma_2,\gamma_3\gamma_4,\gamma_3),\,i\in J_4&&
\Gamma_i=(\gamma_2,\gamma_1,\gamma_3,\gamma_4),\,i\in J_5&&
\Gamma_i=(\gamma_3,\gamma_2,\gamma_1,\gamma_1\gamma_3\gamma_4),\,i\in J_6\\
&\Gamma_i=(\gamma_1,\gamma_1\gamma_2\gamma_4,\gamma_3,\gamma_4),\,i\in J_7&&\Gamma_i=(\gamma_1,\gamma_3,\gamma_2,\gamma_2\gamma_3\gamma_4),\,i\in J_8
&&\Gamma_i=(\gamma_1\gamma_2\gamma_4,\gamma_2,\gamma_3,\gamma_4),\,i\in J_9,\\ &\Gamma_i=(\gamma_1\gamma_3\gamma_4,\gamma_2\gamma_3\gamma_4,\gamma_4,\gamma_3),\,i\in J_{10}
&&\Gamma_i=(\gamma_1,\gamma_2,\gamma_3,\gamma_1\gamma_2\gamma_3\gamma_4),\,i\in J_{11}&&\Gamma_i=(\gamma_2\gamma_4,\gamma_1\gamma_4,\gamma_1\gamma_2\gamma_3\gamma_4,\gamma_4),\,i\in J_{12}\\
&\Gamma_i=(\gamma_3\gamma_4,\gamma_2,\gamma_3,\gamma_1\gamma_3),\,i\in J_{13}&& 
\Gamma_i=(\gamma_1,\gamma_3\gamma_4,\gamma_3,\gamma_2\gamma_3),\,i\in J_{14}
&&\Gamma_i=(\gamma_1,\gamma_2,\gamma_1\gamma_2\gamma_4,\gamma_1\gamma_2\gamma_3),\,i\in J_{15}\,.
\end{align*}}
We have $\theta|_{(G_{t_i})_{\on{der}}}\cong(A_1,2_\rs)$ and $\chi_1|_{I_{t_i}}=1$ if and only if $t_i\in W_{\chi_1}^0$. So
$
R_{\chi,t_i}=(z-1)^2\text{ if }t_i\in W_{\chi_1}^0$, and  $R_{\chi,t_i}=z^2-1\text{ otherwise}.
$

\subsubsection{$(E_8,3_\rs)$}Let $w=w_h^{10}$. Then 
{\tiny\beqn
(w\alpha_i)_{i=1,\ldots,8}=01122100,-1223211,-12343210,23454321,-22344321,11233221,-01122221,01011110\,.
\eeqn}
We have $S_i=\{\theta^j\beta_i,j=0,1,2\}$, $i=1,\ldots,40$, where
{\tiny\begin{align*}
(\beta_i)=& {10000000} ,  {01000000} ,  {00100000} , {00010000} , {00001000} , {00000100} , {00000010} ,  {00000001},
{10100000} , {01010000} ,\\&  {00110000} ,  {00011000} , {00001100} ,  {00000110} ,  {00000011} , {10110000} , {01011000} ,  {00111000} ,  {00011100} , {00001110} \\
& {00000111} , {01111000} ,  {01011100} , {00011110} ,  {00001111} ,  {11111000} , {10111100} ,  {01111100} 
, {00011111} ,  {01111110} ,\\&  {01121100} , {11111110} , {01111111} ,  {11121100} ,  {01122110} ,  {11121110} , {11121111} ,  {11222110} ,  {11222111} , {11232211} \,.\end{align*}
}

\begin{lemma}\label{lem-3se8}
We have 
\bern
I
&=&\langle\gamma_1=\check\alpha_1(\zeta_3)\check\alpha_6(\zeta_3)\check\alpha_8(\zeta_3),\gamma_2=\check\alpha_2(\zeta_3)\check\alpha_5(\zeta_3)\check\alpha_7(\zeta_3),\\&&\qquad\gamma_3=\check\alpha_3(\zeta_3)\check\alpha_5(\zeta_3)\check\alpha_7(\zeta_3^2),\,\gamma_4=\check\alpha_4(\zeta_3)\check\alpha_6(\zeta_3)\check\alpha_8(\zeta_3^2)\rangle\cong\mu_3^4.
\eern
The action of $W$  on $\hat I$ has $2$ orbits with representatives $\chi_0,\chi_1:(\gamma_1,\ldots,\gamma_4)\mapsto(\zeta_3,1,1,1)$ and
\beqn
W_{\chi_1}=W_{\chi_1}^0=\langle t_i,i=7,8,15,16,23,26,34,36,37,38,39,40\rangle\times\langle t_{12}\rangle\cong G_{25}\times\mu_3.
\eeqn
\end{lemma}
Let $\gamma^{abcd}:=\gamma_1^a\gamma_2^b\gamma_3^c\gamma_4^d$.  
Using~\eqref{action-m=3}, one checks that $I_i=\langle\nu_i\rangle$
{\tiny\begin{align*}
\nu_i=&\gamma^{2112},\gamma^{2010},\gamma^{2122},\gamma^{2011},\gamma^{1102},\gamma^{1120},\gamma^{0221},\gamma^{0101},\gamma^{1201},\gamma^{1021},\gamma^{1100},\gamma^{0110},\gamma^{2222},\gamma^{1011},\gamma^{0022},\gamma^{0212},\gamma^{2120},\\
&\gamma^{2202},\gamma^{1200},\gamma^{2110},\gamma^{1112},\gamma^{1212},\gamma^{0210},\gamma^{1121},\gamma^{2211},\gamma^{0021},\gamma^{2101},\gamma^{2002},\gamma^{1222},\gamma^{2220},\gamma^{1010},\gamma^{1002},\gamma^{2021},\gamma^{0122},\\
&\gamma^{2000},\gamma^{0010},\gamma^{0111},\gamma^{0201},\gamma^{0002},\gamma^{0100}
\\
\text{ and }t_i(\gamma_1)=&\gamma^{2112},\gamma^{2012},\gamma^{2121},\gamma^{2010},\gamma^{1100},\gamma^{1122},\gamma^{0221},\gamma^{0102},\gamma^{1200},\gamma^{1022},\gamma^{1101},\gamma^{0110},\gamma^{2222},\gamma^{1010},\gamma^{0020},\gamma^{0210},\gamma^{2122}\\&
\gamma^{2201},\gamma^{1202},\gamma^{2110},\gamma^{1112},\gamma^{1210},\gamma^{0211},\gamma^{1120},\gamma^{2212},\gamma^{0022},\gamma^{2102},\gamma^{2002},\gamma^{1222},\gamma^{2220},\gamma^{1012},\gamma^{1002},\gamma^{2022},\gamma^{0121},\\&\gamma^{2000},\gamma^{0012},\gamma^{0111},\gamma^{0200},\gamma^{0002},\gamma^{0101}\\
t_i(\gamma_2)=&\gamma^{2102},\gamma^{2010},\gamma^{2112},\gamma^{2021},\gamma^{1122},\gamma^{1120},\gamma^{0201},\gamma^{0111},\gamma^{1211},\gamma^{1001},\gamma^{1100},\gamma^{0110},\gamma^{2212},\gamma^{1021},\gamma^{0012},\gamma^{0202},\gamma^{2120},\\&\gamma^{2222},\gamma^{1200},\gamma^{2110},\gamma^{1102},\gamma^{1202},\gamma^{0210},\gamma^{1101},\gamma^{2221},\gamma^{0001},\gamma^{2111},\gamma^{2022},\gamma^{1212},\gamma^{2220},\gamma^{1010},\gamma^{1022},\gamma^{2001},\gamma^{0112},\\&\gamma^{2000},\gamma^{0010},\gamma^{0121},\gamma^{0211},\gamma^{0022},\gamma^{0100}\\
t_i(\gamma_3)=&\gamma^{2112},\gamma^{2210},\gamma^{2122},\gamma^{2111},\gamma^{1002},\gamma^{1220},\gamma^{0121},\gamma^{0001},\gamma^{1201},\gamma^{1021},\gamma^{1200},\gamma^{0110},\gamma^{2222},\gamma^{1011},\gamma^{0122},\gamma^{0012},\gamma^{2020},\\&\gamma^{2202},\gamma^{1000},\gamma^{2010},\gamma^{1012},\gamma^{1112},\gamma^{0210},\gamma^{1121},\gamma^{2011},\gamma^{0221},\gamma^{2201},\gamma^{2002},\gamma^{1122},\gamma^{2120},\gamma^{1110},\gamma^{1202},\gamma^{2121},\gamma^{0222},\\&\gamma^{2200},\gamma^{0010},\gamma^{0011},\gamma^{0101},\gamma^{0102},\gamma^{0100}\\
t_i(\gamma_4)=&\gamma^{1112},\gamma^{2010},\gamma^{1122},\gamma^{2011},\gamma^{0102},\gamma^{0120},\gamma^{1221},\gamma^{2101},\gamma^{2201},\gamma^{1021},\gamma^{0100},\gamma^{2110},\gamma^{0222},\gamma^{1011},\gamma^{0022},\gamma^{1212},\gamma^{1120},\\&\gamma^{0202},\gamma^{2200},\gamma^{1110},\gamma^{0112},\gamma^{2212},\gamma^{1210},\gamma^{0121},\gamma^{0211},\gamma^{0021},\gamma^{1101},\gamma^{2002},\gamma^{2222},\gamma^{0220},\gamma^{1010},\gamma^{1002},\gamma^{2021},\gamma^{2122},\\&\gamma^{2000},\gamma^{0010},\gamma^{2111},\gamma^{1201},\gamma^{0002},\gamma^{2100}\,.
\end{align*}}
Since $\theta|_{(G_{t_i}})_\der\cong (A_2,2_\rs)$ and $\chi_1|_{I_{t_i}}=1$ if and only if $t_i\in W_{\chi_1}^0$, we have
$
R_{\chi,t_i}=(z-1)^3$ (resp. $z^3-1$) if $t_i\in W_{\chi}^0$ (resp. otherwise).

\section{Stable gradings of rank 1}\label{sec-rkone}
 
In this section we prove Theorem~\ref{mainthm} for the stable $\bZ/m\bZ$-gradings of rank 1 (when $m$ is not the (twisted) Coxeter number). We  follow the approach in~\cite[\S4]{VX3}. We regard the pair $(G_0,\Lg_1)$ as a polar representation and consider the regular finite prehomogeneous vector space $(G_0\times\bC^*,\Lg_1)$.

Let us write $\bC[\Lg_1]^{G_0}=\bC[f]$. In our GIT stable setting, we have $f=f_0f_1^{n_1}\ldots f_k^{n_k}$, where $f_0,\ldots,f_k$ are fundamental semi-invariants with corresponding characters $\psi_i:G_0\to\bC$, and we have chosen a semi-invariant $f_0$ so that $f_0$ occurs in $f$ with multiplicity 1.

Consider tuples $s_\bullet=(s_1,\ldots,s_k)\in\bQ^k$ such that $$\psi=\psi_{s_\bullet}=\prod_{i=1}^k\psi_i^{s_i}\in X^*(G_0).$$
We show that in each case, except for $(G_2,3_\rs)$, each character $\chi\in\hat I$ arises from such a $\psi$ as above, that is, $\chi=\psi|_{I}$. Let $D$ be the generating invariant of $\on{Sym}(\Lg_1)^{G_0}$ viewed as a differential operator on $\Lg_1$. Consider the b-function
\beq\label{eqn-bfn}
Df^su=b(s)f^{s-1}u,\ u=\prod_{i=1}^kf_i^{-s_i}.
\eeq
By~\cite[Theorem 4.3]{VX3}, the calculation of the polynomials $R_\chi(z)$ reduces to the calculation of the b-functions~\eqref{eqn-bfn}, that is, we have $R_\chi(z)=b_{\exp}(z)$ when $\chi=\psi_{s_\bullet}|_I$. Theorem~\ref{mainthm} follows from our explicit calculations in the rest of the section. We will calculate the fundamental semi-invariants and the b-functions $b(s)=\prod_i(s-a_i)$ (up to a constant factor) and $b_{\exp}(z):=\prod_i(z-e^{2\pi\mathbf{i}a_i})$ in each case.

\subsection{Kac diagram and the representation $(G_0,\Lg_1)$} Let $\theta$ be a grading corresponding to the Kac diagram with Kac coordinates $s_i\in\{0,1\}$, where $i$ corresponds to nodes of the extended Dynkin diagram as specified in~\S\ref{ssec-dynkin}. We use the Kac diagrams determined in~\cite{RLYG}.

According to~\cite[Proposition 8.6]{Ka}, we have
\begin{enumerate}
\item Let $i_1,\ldots,i_a$ be the indices of the Kac diagram such that $s_{i_j}=0$. Then the Lie algebra $\Lg_0\cong\cZ_{\Lg_0}\oplus(\Lg_0)_{\on{der}}$, where $\cZ_{\Lg_0}$ is the center  ($n-a$ dimensional), and $(\Lg_0)_{\on{der}}$ is a semisimple Lie algebra whose Dynkin diagram is the subdiagram of the affine Dynkin diagram consisting of the vertices $i_1,\ldots,i_a$. 
\item Let $j_1,\ldots,j_b$ be the indices of the Kac diagram such that $s_{j_k}=1$. Then the $\Lg_0$-module $\Lg_1$ is isomorphic to a direct sum of $b$ irreducible modules with highest weights $-\alpha_{j_1},\ldots,-\alpha_{j_b}$. 
\end{enumerate}
This allows us to determine the structure of $G_0$ and $\Lg_1$ as a representation of $G_0$.

For our explicit calculations, we will use the fact that $G_0=Z_{G_0}.(G_0)_{\der}$ and display $\Lg_1$ as a representation of $Z_{G_0}$ (which acts by central characters on irreducible factors of $\Lg_1$), and a representation of $(G_0)_{\der}$, which is a simply connected group.

We will use the following notations 
\begin{itemize}
\item $\mathbf{1}$ denotes the trivial representation of $(G_0)_{\der}$
\item $V(\lambda)$ denotes the irreducible representation of highest weight $\lambda$
\item $\bC^2$, (resp. $(\bC^2)^*$) denotes the (resp. dual of) standard representation $V(\omega)$ of $SL_2$ with realisation: $g:v\mapsto gv$ (resp. $g:v\mapsto(g^t)^{-1}v$), where $\omega$ denotes the fundamental weight of $SL_2$
\item $\on{Sym}^2(\bC^2)$ denotes the irreducible representation $V(2\omega)$ of $SL_2$ with realisation: $g:X\mapsto gXg^t$, where $X$ is a symmetric $2\times 2$ matrix
\item The $SL_2\times SL_2$ representation $\bC^2\boxtimes(\bC^2)^*$ is realised as: $(g_1,g_2):X\mapsto g_1Xg_2^{-1}$ for $X\in M_{2\times 2}(\bC)$.  
\item The $SL_2^{\times 4}$ representation $\bC^2\boxtimes(\bC^2)^*\boxtimes\bC^2\boxtimes\mathbf{1}$ is realised as $$(g_1,g_2,g_3,g_4):(X,Y)\mapsto (g_1 Xg_2^{-1},g_1Yg_2^{-1})\,g_3^t$$ for
 $X,Y\in M_{2\times 2}(\bC)$ where $(X, Y)\,g_3^t:=(aX+bY, cX+dY)$ for $g_3=\left(\begin{matrix}a&b\\c&d\end{matrix}\right)$.
\end{itemize} 

When describing $Z_{G_0}$ we will fix a fundamental $\theta$-stable maximal torus $T_{\mathbf{f}}$ such that $(T_{\mathbf{f}}^{\vartheta_\mathbf{f}})^0$ is a maximal torus of $G_0$, where $\vartheta_\mathbf{f}\in\on{Out}G$ fixes a pinning with respect to $T_{\mathbf{f}}$ (see~\cite[(2.4)]{VX2}). Each $z\in Z_{G_0}$ can be written in the form $z=\prod_{i=1}^n\check\alpha_i(t_i)$, $t_i\in\bC^*$, where $n=\on{rank}G$.

\subsection{Type $(F_4,8_\rs)$}Recall that~\cite{Le} the Kac diagram is $ \xymatrix{{\substack{1\\\circ\\-\alpha_0}}\ar@{-}[r]&{\substack{1\\\bullet\\\alpha_1}}\ar@{-}[r]&{\substack{1\\\bullet\\\alpha_2}}\ar@2{=>}[r]&{\substack{0\\\bullet\\\alpha_3}}\ar@{-}[r]&{\substack{1\\\bullet\\\alpha_4}}}$.  As a $(G_0)_\der\cong SL_2$-representation
\beqn
\Lg_1\cong\mathbf{1}\oplus\mathbf{1}\oplus \on{Sym}^2\bC^2\oplus (\bC^2)^*
\eeqn
where 
\begin{subequations}\label{act1}
\beq\bft\cdot(x_1,x_2,X,v)=(t_1x_1,t_1^{-2}t_2x_2,t_1t_2^{-2}t_3^2X,t_3t_4^{-2}v),\ \bft\in Z_{G_0}
\eeq
\beq
g\cdot(x_1,x_2,X,v)=(x_1,x_2,gXg^t,(g^t)^{-1}v),\ g\in SL_2\,.
\eeq
\end{subequations}
Note that $\mathbf{t}=\prod_{i=1}^4\check\alpha_i(t_i)\in Z_{G_0}$ if and only if $t_3^2=t_2t_4$.

The fundamental semi-invariants and the corresponding characters are
\begin{align*}
&f_0=x_1&&\psi_0(\mathbf{t})=t_1&&f_1=v^tXv&&\psi_1(\mathbf{t})=t_1t_4^{-2}\\
&f_2=\det X&&\psi_2(\mathbf{t})=t_1^{2}t_2^{-2}t_4^{2}&&
f_3=x_2&&\psi_3(\mathbf{t})=t_1^{-2}t_2
\end{align*}
where $\bft\in Z_{G_0}$. 
The invariant is $f=f_0f_1f_2f_3^2$. Applying~\cite[Theorem II.1]{S} (with $\chi_1=\chi_2=1$, $e_1=e_2=1$ and $n_1=1,n_2=2$) we obtain that
\beq\label{eqn-bfn-1}
D_{f_1f_2}(f_1^{z_1}f_2^{z_2})=c_0z_1^2z_2(z_1+z_2+\tfrac{1}{2})(z_1+z_2-\tfrac{1}{2})f_1^{z_1-1}f_2^{z_2-1}\,\text{ where $c_0$ is a constant}.
\eeq
Hence we have
\bern
&&D_f(f^sf_1^{-s_1}f_2^{-s_2}f_3^{-s_3})=cb(s)f^{s-1}f_1^{-s_1}f_2^{-s_2}f_3^{-s_3}
\eern
where
$$b(s)=s\left(s-s_1\right)^2\left(s-s_2\right)\left(s-\tfrac{s_3}{2}\right)\left(s-\tfrac{s_3+1}{2}\right)\left(s-\tfrac{2s_1+2s_2-1}{4}\right)\left(s-\tfrac{2s_1+2s_2+1}{4}\right).$$
Let $\psi_{s_\bullet}=\prod_{i=1}^3\psi_i^{s_i}$. Then $\psi_{s_\bullet}\in X^*(G_0)$ if and only if $s_1,s_3\in\mathbb{Z}$ and $s_2\in\frac{1}{2}\bZ$.
We have
\bern
b(s)=s^5(s-\tfrac{1}{2})(s-\tfrac{1}{4})(s+\tfrac{1}{4})\text{ and }b_{\on{exp}}(z)=\Phi_1^5\Phi_2\Phi_4&&\text{ when $s_1=s_2=s_3=0$}\\
b(s)=s^5(s-\tfrac{1}{2})^3\text{ and }b_{\on{exp}}(z)=\Phi_1^5\Phi_2^3&&\text{ when $s_1=s_3=0$ and $s_2=\tfrac{1}{2}$}.
\eern

Let $a=(1,1,I,(1,0)^t)\in\Lg_1$ be a generic element (as $f(a)\neq 0$). Let $\bft\in Z_{G_0}$ and $g\in SL_2$. We have
\begin{equation}\label{central-1}
\bega
\bft g\cdot a=(t_1,t_1^{-2}t_2,t_1t_2^{-2}t_3^2gg^t,t_3t_4^{-2}(g^t)^{-1}(1,0)^t)=a\\
\Leftrightarrow \bft =\check\alpha_3(t_3)\check\alpha_4(t_3^2),\ g=\on{diag}(t_3^{-3},t_3^3)=\check\alpha_3(t_3^{-3}),\ \text{ where } t_3^4=1.
\eega
\end{equation}
So 
\begin{align*}
I=Z_{G_0}(a)=\{\check\alpha_3(t_3^{-2})\check\alpha_4(t_3^2)\mid t_3^4=1\}=\langle \gamma:=\check\alpha_3(-1)\check\alpha_4(-1)\rangle\cong\mu_2\,.
\end{align*}
Moreover, $\psi_{s_\bullet}(\bft g)=t_1^{s_1+2s_2-2s_3}t_2^{-2s_2+s_3}t_4^{-2s_1+2s_2}$ implies that $\psi_{s_\bullet}(\gamma)=(-1)^{-2s_1+2s_2}$. So $\psi_{0,0,0}|_I=\chi_0$, and $\psi_{0,1/2,0}|_I=\chi_1$, the nontrivial character of $I$.

Using~\eqref{central-1} and~\eqref{act1} one checks that $\on{det}(\gamma|_{\Lg_1})=1$.  So the character $\tau:I\to\{\pm 1\}$ is trivial.

\subsection{Type $(E_6,9_\rs)$} The Kac diagram is $\xymatrix@R-1.5pc{{\substack{1\\\bullet\\\alpha_1}}\ar@{-}[r]&{\substack{1\\\bullet\\\alpha_3}}\ar@{-}[r]&{\substack{0\\\bullet\\\alpha_4}}\ar@{-}[r]&{\substack{1\\\bullet\\\alpha_5}}\ar@{-}[r]&{\substack{1\\\bullet\\\alpha_6}}\\&&{\substack{1\\\bullet\\\alpha_2}}\ar@{-}[u]\ar@{-}[d]&&\\&&{\substack{1\\\circ\\-\alpha_0}}&&}$.  As a $(G_0)_\der\cong SL_2$-representation, 
$
 \Lg_1\cong\mathbf{1}\oplus\mathbf{1}\oplus\mathbf{1}\oplus \bC^2\oplus (\bC^2)^*\oplus (\bC^2)^*$ 
and 
 \bern
 &&\bft\cdot(x_0,x_1,x_2,v_1,v_2,v_3)=\left(t_1^{-2}t_3x_0,t_5t_6^{-2}x_1,t_2x_2,t_4t_5^{-2}t_6v_1,t_2^{-2}t_4v_2,t_1t_3^{-2}t_4v_3\right),\ \bft\in Z_{G_0}.
 \eern
  We have $\bft=\prod_{i=1}^6\check\alpha_i(t_i)\in Z_{G_0}$ if and only if $t_4^2=t_2t_3t_5$. The fundamental semi-invariants and the corresponding characters are
\begin{align*}
&f_0=x_0&&\psi_0(\bft)=t_1^{-2}t_3&&f_1=x_1&&\psi_1(\bft)=t_5t_6^{-2}\\
&f_2=x_2&&\psi_2(\bft)=t_2&&
f_3=\on{det}\left(\begin{matrix}v_2&v_3\end{matrix}\right)&&\psi_3(\bft)=t_1t_2^{-1}t_3^{-1}t_5\\
&f_4=v_1^t v_2&&\psi_4(\bft)=t_2^{-1}t_3t_5^{-1}t_6&&
f_5=v_1^tv_3&&\psi_5(\bft)=t_1t_2t_3^{-1}t_5^{-1}t_6
\end{align*}
where $\bft\in Z_{G_0}$. The invariant is $f=f_0f_1f_2f_3f_4f_5$. 
One checks that
\begin{align}\label{eqn-diff-2}
D_{f_i}(f_3^{z_3}f_4^{z_4}f_5^{z_5})=z_i(1+z_3+z_4+z_5)f_3^{z_3}f_4^{z_4}f_5^{z_5}f_i^{-1},\,i=3,4,5.
\end{align}
It follows that
$
D_f(f^s\prod_{i=1}^5f_i^{-s_i})=cb(s)f^{s-1}\prod_{i=1}^5f_i^{-s_i}$
where
\bern
&&b(s)=s(s-s_1)(s-s_2)(s-s_3)(s-s_4)(s-s_5)\\&&\qquad\qquad\left(s-\tfrac{s_3+s_4+s_5}{3}\right)\left(s-\tfrac{s_3+s_4+s_5+1}{3}\right)
\left(s-\tfrac{s_3+s_4+s_5-1}{3}\right).
\eern
Let $\psi_{s_\bullet}=\prod_{i=1}^5\psi_i^{s_i}$. We have $\psi_{s_\bullet}\in X^*(G_0)$ if and only if 
\beqn 
\text{$s_1\in\tfrac{1}{3}\bZ$, $s_1+s_2\in\bZ$, $s_1-s_3\in\bZ$, $s_1+s_5\in\bZ$, and $s_4\in\bZ$.}
\eeqn
We have
\begin{align*}
&b(s)=s^7(s-\tfrac{1}{3})(s+\tfrac{1}{3})\text{ and }b_{\exp}(z)=\Phi_1^7\Phi_3&&\text{ when }s_i=0\\
&b(s)=s^3(s-\tfrac{1}{3})^3(s+\tfrac{1}{3})^3\text{ and }b_{\exp}(z)=\Phi_1^3\Phi_3^3&&\text{ when }s_4=0, s_1=s_3=\tfrac{1}{3},s_2=s_5=-\tfrac{1}{3}\\
&\ &&\text{ or }s_4=0, s_1=s_3=-\tfrac{1}{3},s_2=s_5=\tfrac{1}{3}\,.
\end{align*}
 
Let $a=(1,1,1,(1,1)^t,(1,0)^t,(0,1)^t)\in\Lg_1$. Let $\bft\in Z_{G_0}$ and $g\in SL_2$. We have
\begin{align*}
&\bft g\cdot a=a\implies \bft =\check\alpha_1(t_6^2)\check\alpha_3(t_6)\check\alpha_4(t_4)\check\alpha_5(t_6^2)\check\alpha_6(t_6),\ g=\on{diag}(t_4,t_4),\text{ where } t_4^2=1,t_6^3=1.
\end{align*}
Hence
$I=Z_{G_0}(a)=\langle\gamma:= \check\alpha_1(\zeta_3^2)\check\alpha_3(\zeta_3)\check\alpha_5(\zeta_3^2)\check\alpha_6(\zeta_3)\rangle=Z_G\cong\mu_3$. 
Moreover,  $\psi_{s_\bullet}(\gamma)=(\zeta_3)^{-3s_3}$. So $\psi_{0,0,0,0,0}(\gamma)=1$, $\psi_{0,\frac{1}{3},-\frac{1}{3},0,\frac{1}{3},-\frac{1}{3}}(\gamma)=\zeta_3^2$ and $\psi_{0,-\frac{1}{3},\frac{1}{3},0,-\frac{1}{3},\frac{1}{3}}(\gamma)=\zeta_3$.

\subsection{Type $({}^2E_6,12_\rs)$} 
The Kac diagram is $ \xymatrix{{\substack{1\\\circ\\-\beta_0}}\ar@{-}[r]&{\substack{1\\\bullet\\{\beta_1}}}\ar@{-}[r]&{\substack{0\\\bullet\\\beta_2}}\ar@2{<=}[r]&{\substack{1\\\bullet\\\beta_3}}\ar@{-}[r]&{\substack{1\\\bullet\\\beta_4}}}$ where $\beta_1=\frac{\alpha_1+\alpha_6}{2}$, $\beta_2=\frac{\alpha_3+\alpha_5}{2}$, $\beta_3=\alpha_4$, $\beta_4=\alpha_2$ and $\beta_0=2\beta_1+3\beta_2+2\beta_3+\beta_4$. Let $$T_{\mathbf{f}}^\vartheta=\{\bft:=\check\beta_1(t_1)\check\beta_2(t_2)\check\beta_3(t_3)\check\beta_4(t_4)\mid t_i\in\bC^*\}\subset G_0.$$
Note that $\bft\in Z_{G_0}$ if and only if $t_2^2=t_1t_3$.  
As a $(G_0)_\der\cong SL_2$-representation $\Lg_1\cong\mathbf{1}\oplus\mathbf{1}\oplus (\bC^2)^*\oplus\on{Sym}^2\bC^2
$
and 
\bern
&&\bft\cdot(x_0,x_1,v,X)=(t_1x_0,t_3t_4^{-2}x_1,t_1^{-2}t_2v,t_2^2t_3^{-2}t_4X), \ \bft\in Z_{G_0}\,.\eern
The fundamental semi-invariants and the corresponding characters are
\begin{align*}
&f_0=\on{Det}X&&\psi_0(\bft)=t_1^{2}t_3^{-2}t_4^2&&f_1=v^tXv&&\psi_1(\bft)=t_1^{-2}t_4\\
&f_2=x_0&&\psi_2(\bft)=t_1&&
f_3=x_1&&\psi_3(\bft)=t_3t_4^{-2}
\end{align*}
where $\bft\in Z_{G_0}$. The invariant is $f=f_0f_1^2f_2^2f_3^2$. Applying~\eqref{eqn-bfn-1} we obtain that $D_f(f^s)=cb(s)f^{s-1}$ where $c$ is a constant and  
\beqn
b(s)=s^5(s-\tfrac{1}{2})^5(s+\tfrac{1}{6})(s-\tfrac{1}{6})\,.
\eeqn
So $b_{\exp}(z)=\Phi_1^5\Phi_2^5\Phi_6$.

\subsection{Type $(E_7,14_\rs)$}

The Kac diagram is $$\xymatrix@R-1pc{{\substack{1\\\circ\\-\alpha_0}}\ar@{-}[r]&{\substack{1\\\bullet\\\alpha_1}}\ar@{-}[r]&{\substack{1\\\bullet\\\alpha_3}}\ar@{-}[r]&{\substack{0\\\bullet\\\alpha_4}}\ar@{-}[r]&{\substack{1\\\bullet\\\alpha_5}}\ar@{-}[r]&{\substack{1\\\bullet\\\alpha_6}}\ar@{-}[r]&{\substack{1\\\bullet\\\alpha_7}}\\&&&{\substack{1\\\bullet\\\alpha_2}}\ar@{-}[u]&&&}.$$
We have $\Lg_1\cong\mathbf{1}\oplus\mathbf{1}\oplus\mathbf{1}\oplus\mathbf{1}\oplus \bC^2\oplus (\bC^2)^*\oplus (\bC^2)^*$ as a $(G_0)_{\der}\cong SL_2$-representation and 
\bern
&&\bft\cdot (x_0,x_1,x_2,x_3,v_1,v_2,v_3)=(t_1x_0,t_1^{-2}t_3x_1,t_5t_6^{-2}t_7x_2,t_6t_7^{-2}x_3,t_4t_5^{-2}t_6v_1,t_2^{-2}t_4v_2,t_1t_3^{-2}t_4v_3)
\eern
for  $\bft=\prod_{i=1}^7\check\alpha_i(t_i)\in Z_{G_0}$.  
 We have $\bft\in Z_{G_0}$ if and only if $t_4^2=t_2t_3t_5$.   

The fundamental semi-invariants and the corresponding characters are
\begin{align*}
&f_0=x_0&&\psi_0(\bft)=t_1&&
f_1=x_1&&\psi_1(\bft)=t_1^{-2}t_3\\
&f_2=x_2&&\psi_2(\bft)=t_5t_6^{-2}t_7&&
f_3=\on{det}\left(\begin{matrix}v_2&v_3\end{matrix}\right)&&\psi_3(\bft)=t_1t_2^{-1}t_3^{-1}t_5\\
&f_4=v_1^tv_2&&\psi_4(\bft)=t_2^{-1}t_3t_5^{-1}t_6&&
f_5=v_1^tv_3&&\psi_5(\bft)=t_1t_2t_3^{-1}t_5^{-1}t_6\\
&f_6=x_3&&\psi_6(\bft)=t_6t_7^{-2}
\end{align*}
where $\bft\in Z_{G_0}$.
The invariant is $f=f_0f_1^2f_2^2f_3f_4f_5^2f_6$. 

Using~\eqref{eqn-diff-2}, we obtain that
$D_f(f^s\prod_{i=1}^6f_i^{-s_i})=cb(s)f^{s-1}\prod_{i=1}^6f_i^{s_i}$ 
where  
\bern
b(s)&=&s(s-s_6)\left(s-\tfrac{s_1}{2}\right)\left(s-\tfrac{s_1+1}{2}\right)\left(s-\tfrac{s_2}{2}\right)\left(s-\tfrac{s_2+1}{2}\right)\left(s-s_3\right)\left(s-s_4\right)\left(s-\tfrac{s_5}{2}\right)\\&&\left(s-\tfrac{s_5+1}{2}\right)\left(s-\tfrac{s_3+s_4+s_5}{4}\right)\left(s-\tfrac{s_3+s_4+s_5+1}{4}\right)
\left(s-\tfrac{s_3+s_4+s_5+2}{4}\right)\left(s-\tfrac{s_3+s_4+s_5-1}{4}\right).
\eern
Let $\psi_{s_\bullet}=\prod_{i=1}^5\psi_i^{s_i}$. Then $\psi_{s_\bullet}\in X^*(G_0)$ if and only if 
$
s_1,s_2,s_4\in\bZ,\,s_3\in\tfrac{1}{2}\bZ,\,s_3+s_5\in\bZ,s_3+s_6\in\bZ.
$
We have
\begin{align*}
&b(s)=s^8(s-\tfrac{1}{2})^4(s-\tfrac{1}{4})(s+\tfrac{1}{4})\text{ and }b_{\exp}(z)=\Phi_1^8\Phi_2^4\Phi_4&&\text{when $s_i=0$}\\
&b(s)=s^5(s-\tfrac{1}{2})^5(s-\tfrac{1}{4})^2(s+\tfrac{1}{4})^2\text{and }b_{\exp}(z)=\Phi_1^5\Phi_2^5\Phi_4^2&&\text{when }s_1=s_2=s_4=0\\&\ &&s_3=s_6=\tfrac{1}{2}\text{ and }s_5=-\tfrac{1}{2}\,.
\end{align*}

Let $a=(1,1,1,1,(1,1)^t,(1,0)^t,(0,1)^t)\in\Lg_1$. One checks that 
\begin{align*}
I&=Z_{G_0}(a)=\{\bft.g,\bft=\check\alpha_2(t)\check\alpha_5(t)\check\alpha_7(t)\check\alpha_4(t_4),g=\on{diag}(t_4,t_4)\mid t_4^2=1,t^2=1\}
\\
&=\langle\check\alpha_2(-1)\check\alpha_5(-1)\check\alpha_7(-1)\rangle=Z_G\cong\mu_2\,.
\end{align*}
Moreover, $\psi_{s_\bullet}(\gamma)=(-1)^{-2s_2+2s_4+2s_6}$.

\subsection{Type $(E_8,24_\rs)$} The Kac diagram is $$\xymatrix@R-1pc{{\substack{1\\\bullet\\\alpha_1}}\ar@{-}[r]&{\substack{1\\\bullet\\\alpha_3}}\ar@{-}[r]&{\substack{0\\\bullet\\\alpha_4}}\ar@{-}[r]&{\substack{1\\\bullet\\\alpha_5}}\ar@{-}[r]&{\substack{1\\\bullet\\\alpha_6}}\ar@{-}[r]&{\substack{1\\\bullet\\\alpha_7}}\ar@{-}[r]&{\substack{1\\\bullet\\\alpha_8}}\ar@{-}[r]&{\substack{1\\\circ\\-\alpha_0}}\\&&{\substack{1\\\bullet\\\alpha_2}}\ar@{-}[u]&&&&}\,.$$
As a $(G_0)_\der\cong SL_2$-representation
$\Lg_1\cong\triv^{\oplus5}\oplus \bC^2\oplus (\bC^2)^*\oplus (\bC^2)^*$ and  
\begin{align*}
\bft\cdot(x_0,x_1,x_2,x_3,x_4,v_1,v_2,v_3)&=(t_8x_0,t_1^{-2}t_3x_1,t_5t_6^{-2}t_7x_2,t_6t_7^{-2}t_8x_3,t_7t_8^{-2}x_4,\\&\qquad t_1t_3^{-2}t_4v_1,t_2^{-2}t_4v_2,t_4t_5^{-2}t_6v_3) 
\end{align*}
for $\bft=\prod_{i=1}^8\check\alpha_i(t_i)\in Z_{G_0}$. 
 We have $\bft\in Z_{G_0}$ if and only if  $t_4^2=t_2t_3t_5$.  

The fundamental semi-invariants and the corresponding characters are
\begin{align*}
&f_0=x_0&&\chi_0=t_8&&
f_1=x_1&&\chi_1=t_1^{-2}t_3\\
&f_2=x_2&&\chi_2=t_5t_6^{-2}t_7&&
f_3=v_1^t\,v_2&&\chi_3=t_1t_2^{-1}t_3^{-1}t_5\\
&f_4=\on{det}\left(\begin{matrix}v_2&v_3\end{matrix}\right)&&\chi_4=t_2^{-1}t_3t_5^{-1}t_6&&
f_5=v_1^t\,v_3&&\chi_5=t_1t_2t_3^{-1}t_5^{-1}t_6\\
&f_6=x_3&&\chi_6=t_6t_7^{-2}t_8&&
f_7=x_4&&\chi_7=t_7t_8^{-2}.
\end{align*}
The invariant is $f=f_0f_1^2f_2^4f_3f_4^2f_5^3f_6^3f_7^2$.  
We have $\prod_{i=1}^6\psi_i^{s_i}\in X^*(G_0)$ if and only if 
$
s_i\in\bZ.
$
Applying~\eqref{eqn-diff-2}, we obtain that $D_f(f^s)=cb(s)f^{s-1}$, where $c$ is a constant and $$b(s)=s^9(s-\tfrac{1}{2})^5(s-\tfrac{1}{4})(s-\tfrac{3}{4})(s-\tfrac{1}{3})^3(s-\tfrac{2}{3})^3(s-\tfrac{1}{6})(s+\tfrac{1}{6}).$$
So $b_{\exp}(z)=\Phi_1^9\Phi_2^5\Phi_3^3\Phi_4\Phi_6.$

\subsection{Type $(E_8,20_\rs)$}The Kac diagram is $$\xymatrix@R-1pc{{\substack{1\\\bullet\\\alpha_1}}\ar@{-}[r]&{\substack{1\\\bullet\\\alpha_3}}\ar@{-}[r]&{\substack{0\\\bullet\\\alpha_4}}\ar@{-}[r]&{\substack{1\\\bullet\\\alpha_5}}\ar@{-}[r]&{\substack{0\\\bullet\\\alpha_6}}\ar@{-}[r]&{\substack{1\\\bullet\\\alpha_7}}\ar@{-}[r]&{\substack{1\\\bullet\\\alpha_8}}\ar@{-}[r]&{\substack{1\\\circ\\-\alpha_0}}\\&&{\substack{1\\\bullet\\\alpha_2}}\ar@{-}[u]&&&&}.$$
As a $(G_0)_\der\cong SL_2\times SL_2$-representation
$$\Lg_1\cong\triv^{\oplus3}\oplus (\bC^2)^*\boxtimes\triv\oplus (\bC^2)^*\boxtimes\triv\oplus \triv\boxtimes\bC^2\oplus \bC^2\boxtimes(\bC^2)^*$$
where  
\begin{align*}
&\bft\cdot(x_0,x_1,x_2,v_1,v_2,v_3,X)=(t_8x_0,t_1^{-2}t_3x_1,t_7t_8^{-2}x_2,t_1t_3^{-2}t_4v_1,t_2^{-2}t_4v_2,t_6t_7^{-2}t_8v_3,t_4t_5^{-2}t_6X)\\
&(g_1,g_2)\cdot(x_0,x_1,x_2,v_1,v_2,v_3,X)=(x_0,x_1,x_2,(g_1^t)^{-1}v_1,(g_1^t)^{-1}v_2,g_2v_3, g_1Xg_2^{-1})
\end{align*}
 for $\bft=\prod_{i=1}^8\check\alpha_i(t_i)\in Z_{G_0}$ and $(g_1,g_2)\in SL_2\times SL_2$ (the corresponding roots are $\alpha_4,\alpha_6$). We have $\bft\in Z_{G_0}$ if and only if  $t_4^2=t_2t_3t_5$ and $t_6^2=t_5 t_7$.
 
The fundamental semi-invariants and the corresponding characters are
\begin{align*}
&f_0=x_0&&\psi_0(\bft)=t_8&&
f_1=x_1&&\psi_1(\bft)=t_1^{-2}t_3\\
&f_2=x_2&&\psi_2(\bft)=t_7t_8^{-2}&&
f_3=\on{det}\left(\begin{matrix}v_1&v_2\end{matrix}\right)&&\psi_3(\bft)=t_1t_2^{-1}t_3^{-1}t_5\\
&f_4=\det X&&\psi_4(\bft)=t_2t_3t_5^{-2}t_7&&
f_5=v_2^t\,X\,v_3&&\psi_5(\bft)=t_2^{-1}t_3t_7^{-1}t_8\\
&f_6=v_1^t\,X\,v_3&&\psi_6(\bft)=t_1t_2t_3^{-1}t_7^{-1}t_8
\end{align*}
where $\bft\in Z_{G_0}$. 
The invariant is $f=f_0f_1^2f_2^2f_3^2f_4f_5f_6^2$.  We have $\prod_{i=1}^6\psi_i^{s_i}\in X^*(G_0)$ if and only if $s_i\in \bZ$. One checks that
\begin{subequations}
\beq
D_{f_i}\left(\prod_{j=3}^6f_j^{z_j}\right)=z_i(1+z_i+z_5+z_6)f_i^{-1}\prod_{j=3}^6f_j^{z_j},\ \,i=3,4
\eeq
\beq
D_{f_i}\left(\prod_{j=3}^6f_j^{z_j}\right)=z_i(1+z_4+z_5+z_6)(1+z_3+z_5+z_6)f_i^{-1}\prod_{j=3}^6f_j^{z_j},\ \,i=5,6\,.
\eeq
\end{subequations}
Hence $D_f(f^s)=cb(s)f^{s-1}$ where $c$ is a constant and 
$$b(s)=s^6(s-\tfrac{1}{2})^5(s+\tfrac{1}{4})(s-\tfrac{1}{4})(s+\tfrac{1}{5})(s-\tfrac{1}{5})(s-\tfrac{2}{5})(s-\tfrac{3}{5}).$$ 
So $b_{\exp}(z)=\Phi_1^9\Phi_2^5\Phi_4\Phi_5.$

\subsection{Type $(E_8,15_\rs)$}The Kac diagram is $$\xymatrix@R-1pc{{\substack{1\\\bullet\\\alpha_1}}\ar@{-}[r]&{\substack{0\\\bullet\\\alpha_3}}\ar@{-}[r]&{\substack{1\\\bullet\\\alpha_4}}\ar@{-}[r]&{\substack{0\\\bullet\\\alpha_5}}\ar@{-}[r]&{\substack{1\\\bullet\\\alpha_6}}\ar@{-}[r]&{\substack{0\\\bullet\\\alpha_7}}\ar@{-}[r]&{\substack{1\\\bullet\\\alpha_8}}\ar@{-}[r]&{\substack{1\\\circ\\-\alpha_0}}\\&&{\substack{0\\\bullet\\\alpha_2}}\ar@{-}[u]&&&&}.$$ 
As a $(G_0)_\der\cong SL_2\times SL_2\times SL_2\times SL_2$-representation
 $$\Lg_1\cong (\bC^2)^*\boxtimes\mathbf{1}\boxtimes\mathbf{1}\boxtimes\mathbf{1}\oplus \bC^2\boxtimes(\bC^2)^*\boxtimes\bC^2\boxtimes\mathbf{1}\oplus \mathbf{1}\boxtimes\bC^2\boxtimes\mathbf{1}\boxtimes(\bC^2)^*\oplus \mathbf{1}\boxtimes\mathbf{1}\boxtimes\mathbf{1}\boxtimes\bC^2\oplus\triv$$
 where 
\begin{align*}
&\bft\cdot(v_1,X,Y,Z,v_2,x_0)=(t_1^{-2}t_3v_1,t_2t_3t_4^{-2}t_5X,t_2t_3t_4^{-2}t_5Y,t_5t_6^{-2}t_7Z,t_7t_8^{-2}v_2,t_8x_0)\\
&(g_1,g_2,g_3,g_4)\cdot(v_1,X,Y,Z,v_2,x_0)=((g_1^t)^{-1}v_1,(g_1 Xg_2^{-1},g_1Yg_2^{-1})\,g_3^t,g_2Z g_4^{-1},g_4v_2,x_0)
\end{align*} 
for $\bft=\prod_{i=1}^8\check\alpha_i(t_i)\in Z_{G_0}$ and $(g_1,g_2,g_3,g_4)\in SL_2\times SL_2\times SL_2\times SL_2$ (with corresponding simple roots $\alpha_3,\alpha_5,\alpha_2,\alpha_7$). 

 We have $\bft\in Z_{G_0}$ if and only if  $t_2^2=t_4$, $t_3^2=t_1t_4$, $t_5^2=t_4t_6$ and $t_7^2=t_6t_8$.  
A related prehomogeneous vector space has been studied in~\cite[\S3.5, (5-17)]{U}. 
The fundamental semi-invariants and the corresponding characters are (here we have used {\em loc. cit.})
\begin{align*}
&f_0=\substack{\text{discriminant of the binary quadratic form}\\\det(uX+vY)\text{ in $u,v$}}&&\psi_0=t_1^{2}t_4^{-2}t_6^{2}&&f_1=x_0&&\psi_1=t_8\\
&f_2=\on{det}\left(\begin{matrix}XZv_2&YZv_2\end{matrix}\right)&&\psi_2=t_1t_8^{-2}&&
f_3=\on{det}\left(\begin{matrix}X^t\,v_1&Y^t\,v_1\end{matrix}\right)&&\psi_3=t_1^{-2}t_6\\
&f_4=\det Z&&\psi_4=t_4 t_6^{-2}t_8.
\end{align*}
 The invariant is $f=f_0f_1^2f_2^2f_3^2f_4^2$. By {\em loc. cit.}, we have $D_f(f^s)=cb(s)f^{s-1}$ where $c$ is a constant and 
\beqn
b(s)=cs^9(s-\tfrac{1}{2})^9(s+\tfrac{1}{6})^3(s-\tfrac{1}{6})^3(s-\tfrac{1}{3})(s-\tfrac{2}{3})(s+\tfrac{1}{5})(s-\tfrac{1}{5})(s-\tfrac{2}{5})(s-\tfrac{3}{5})\,.
\eeqn
 Hence
$b_{\on{exp}}(z)=\Phi_1^9\Phi_2^9\Phi_3\Phi_5\Phi_6^3.$

\subsection{Type $(G_2,3_\rs)$} 

The Kac diagram is $\xymatrix@C=1em{{\substack{0\\\bullet\\\alpha}}&{{\substack{1\\\bullet\\\beta}}}\ar@{-}[r]\ar@3{->}[l]&{{\substack{1\\\circ\\-\alpha_0}}}}$. Recall that $\langle\beta,\check\alpha\rangle=-3$. As a $(G_0)_{\on{der}}\cong SL_2$-representation $$\Lg_1\cong \mathbf{1}\oplus V(3\omega).$$ We identify $V(3\omega)$ as the space of degree 3 homogeneous polynomials $x_1 u^3+x_2u^2v+x_3uv^2+x_4v^3$ in two variables $u,v$, and write $\Lg_1=\oplus_{i=0}^4\bC x_i$. We have $\bft=\check\alpha(t_1)\check\beta(t_2)\in Z_{G_0}$ if and only if $t_1^2=t_2$, and 
\beqn
\bft:(x_0,x_1,x_2,x_3,x_4)\mapsto(t_2x_0,t_1^3t_2^{-2}x_1,t_1^3t_2^{-2}x_2,t_1^3t_2^{-2}x_3,t_1^3t_2^{-2}x_4),\ \bft\in Z_{G_0}.
\eeqn
The fundamental semi-invariants and the corresponding characters are (see~\cite[Proposition 6, p79]{SK})
\begin{align*}
&f_0=x_2^2x_3^2 -27x_1^2x_4^2- 4 x_1 x_3^3-4x_2^3x_4 + 18 x_1 x_2 x_3 x_4&&\psi_0(\bft)=t_2^{-2}\\
&f_1=x_0&&\psi_1(\bft)=t_2 \end{align*}
where $\bft\in Z_{G_0}$.
The invariant is $f=f_0f_1^2$.

 By~\cite{Sh}, we have $D_{f_0}f_0^s=s^2(s+\tfrac{1}{6})(s-\tfrac{1}{6})f_0^{s-1}$. Thus  we have $D_f(f^sf_1^{-s_1})=c b(s)f^{s-1}f_1^{-s_1}$, where $c$ is a constant and 
   \beqn
  b(s)= s^2\left(s-\tfrac{s_1}{2}\right)\left(s-\tfrac{s_1+1}{2}\right)\left(s+\tfrac{1}{6}\right)\left(s-\tfrac{1}{6}\right).
   \eeqn
 We have $\psi_1^{s_1}\in X^*(G_0)$ if and only if $s_1\in\bZ$. 
   When $s_1=0$, we have $$b(s)=s^3(s-\tfrac{1}{2})(s+\tfrac{1}{6})(s-\tfrac{1}{6})\text{ and  }b_{\exp}(z)=\Phi_1^3\Phi_2\Phi_6.$$
   We conclude that $R_{\chi_0,s}=\Phi_1^3\Phi_2\Phi_6$.

   We can assume that $\theta=\on{Int}n_w$ with $w=s_\beta s_\alpha s_\beta s_\alpha$. Then $W=\langle s\rangle\cong\mu_6$, where $s=s_\alpha s_\beta$, and $w:\check\alpha\mapsto-2\check\alpha-3\check\beta, \check\beta\mapsto\check\alpha+\check\beta$. The following lemma can be checked easily.
   
   \begin{lemma}
   We have $I:=\langle\gamma:=\check\alpha(\xi_3)\rangle\cong\mu_3$ and the action of $W$ on $\hat I$ has representatives $\chi_0,\chi_1$, where $\chi_1:\gamma\mapsto\zeta_3$. Moreover,
   \beqn
 W_{\chi_1}=W_{\chi_1}^0=\langle s^2\rangle\cong\mu_3.
\eeqn
   \end{lemma}

In what follows we show that $\cH_{W_{\chi_1}}\cong\cH_{\mu_3,\Phi_1^3}$. We do so by showing that there are 5 cuspidal character sheaves $\on{IC}(\Lg_1^{rs},\cE_i)$, where $\on{rank}\cE_i=1$, $i=1,2,3,4$, and $\on{rank}\cE_5=2$.

Let us first determine the set of simple $G_0$-equivariant perverse sheaves on  $\cN_{-1}=\Lg_{-1}\cap\cN$. For this we will use another realisation of $\theta:G\to G$ as follows
$$\theta=\on{Int}(\check\alpha(\zeta_3)\check\beta(\zeta_3^2))$$
where $\check\alpha,\check\beta\in X_*(T_{\rf})$ are the simple co-roots with respect to a  fundamental $\theta$-stable maximal torus $T_{\rf}$ (see~\cite[\S2.1]{VX2}). Let $B\supset T_\rf$ be the standard Borel subgroup of $G$.
We have
\begin{align*}
&\Lg_0=\Lt_{\rf}\oplus \Lg_\alpha\oplus\Lg_{-\alpha},\qquad
 \Lg_1=\Lg_\beta\oplus\Lg_{\alpha+\beta}\oplus \Lg_{2\alpha+\beta}\oplus \Lg_{3\alpha+\beta}\oplus \Lg_{-3\alpha-2\beta}
\\
&\Lg_{-1}=\Lg_{-\beta}\oplus\Lg_{-\alpha-\beta}\oplus \Lg_{-2\alpha-\beta}\oplus \Lg_{-3\alpha-\beta}\oplus \Lg_{3\alpha+2\beta}.
\end{align*}
We choose $x_{\pm\alpha}:\bG_a\to U_{\pm\alpha}\subset G$, $h_\alpha, h_\beta:\bG_m\to T_{\rf}$, and $n_\alpha\in N_G(T_{\rf})$ etc as in~\cite{C2} (see also \cite[\S2.1]{X1}). We have
\begin{align*}
G_0=&\{h_\alpha(t_1)h_\beta(t_2)x_\alpha(a)\mid t_1,t_2\in\bC^*,a\in\bC\}\\
&\sqcup\{x_\alpha(a_2)n_\alpha h_\alpha(t_1)h_\beta(t_2)x_\alpha(a_1)\mid t_1,t_2\in\bC^*,a_1,a_2\in\bC\}.
\end{align*}
Let $X_r=dx_r(1)\in\Lg_r$ for $r\in R$. 
We have
\begin{subequations}\label{eqn-adjoint}
{\tiny\beq
\begin{aligned}
 x_\alpha(a):&\ X_\beta\mapsto X_\beta+aX_{\alpha+\beta}+a^2X_{2\alpha+\beta}+a^3X_{3\alpha+\beta}&&X_{2\alpha+\beta}\mapsto X_{2\alpha+\beta}+3aX_{3\alpha+\beta}&&X_{3\alpha+\beta}\mapsto X_{3\alpha+\beta}\\
&X_{\alpha+\beta}\mapsto X_{\alpha+\beta}+2aX_{2\alpha+\beta}+3a^2X_{3\alpha+\beta}&
&X_{3\alpha+2\beta}\mapsto X_{3\alpha+2\beta}&&X_{-\beta}\mapsto X_{-\beta}\\
&X_{-2\alpha-\beta}\mapsto X_{-2\alpha-\beta}-2aX_{-\alpha-\beta}+3a^2X_{-\beta}&&X_{-\alpha-\beta}\mapsto X_{-\alpha-\beta}-3aX_{-\beta}\\
&X_{-3\alpha-\beta}\mapsto X_{-3\alpha-\beta}-aX_{-2\alpha-\beta}+a^2X_{-\alpha-\beta}-a^3X_{-\beta}&&X_{-3\alpha-2\beta}\mapsto X_{-3\alpha-2\beta}
\end{aligned}
\eeq}
{\tiny\beq
\begin{aligned}
n_\alpha:\ &X_\beta\mapsto X_{3\alpha+\beta}&&X_{3\alpha+\beta}\mapsto -X_\beta&&X_{\alpha+\beta}\mapsto -X_{2\alpha+\beta}&&X_{2\alpha+\beta}\mapsto X_{\alpha+\beta}&&
X_{-\beta}\mapsto X_{-3\alpha-\beta}\\&X_{-3\alpha-\beta}\mapsto -X_{-\beta}&&X_{-\alpha-\beta}\mapsto -X_{-2\alpha-\beta}&&X_{-2\alpha-\beta}\mapsto X_{-\alpha-\beta}&
&X_{3\alpha+2\beta}\mapsto X_{3\alpha+2\beta}&&X_{-3\alpha-2\beta}\mapsto X_{-3\alpha-2\beta}\end{aligned}
\eeq}
{\tiny\beq
\begin{aligned}
h_\alpha(t):\ &X_\beta\mapsto t^{-3}X_\beta&&X_{\alpha+\beta}\mapsto t^{-1}X_{\alpha+\beta}&&X_{2\alpha+\beta}\mapsto tX_{2\alpha+\beta}&&X_{3\alpha+\beta}\mapsto t^{3}X_{3\alpha+\beta}&&X_{3\alpha+2\beta}\mapsto X_{3\alpha+2\beta}\\
h_\beta(t):\ &X_\beta\mapsto t^{2}X_\beta&&X_{\alpha+\beta}\mapsto tX_{\alpha+\beta}&&X_{2\alpha+\beta}\mapsto X_{2\alpha+\beta}&&X_{3\alpha+\beta}\mapsto t^{-1}X_{3\alpha+\beta}&&X_{3\alpha+2\beta}\mapsto t X_{3\alpha+2\beta}\end{aligned}
\eeq}
\end{subequations}

Let us write $y=y_1 X_{-\beta}+y_2 X_{-\alpha-\beta}+y_3 X_{-2\alpha-\beta}+y_4 X_{-3\alpha-\beta}+y_0X_{3\alpha+2\beta}\in \Lg_{-1}$ so that $\Lg_{-1}=\{(y_0,y_1,\ldots,y_4)\mid y_i\in\bC\}$.
Similarly  let $x=x_1 X_{\beta}+x_2 X_{\alpha+\beta}+x_3 X_{2\alpha+\beta}+x_4 X_{3\alpha+\beta}+x_0X_{-3\alpha-2\beta}\in \Lg_{1}$ so that $\Lg_{1}=\{(x_0,x_1,\ldots,x_4)\mid x_i\in\bC\}$.
One checks that
\begin{align*}
\bC[\Lg_{-1}]^{G_0}&=\bC[y_0^2h(y_1,y_2,y_3,y_4)]
\qquad
\bC[\Lg_{1}]^{G_0}=\bC[x_0^2h(x_1,x_2,x_3,x_4)]
\end{align*}
where 
\begin{align*}
h(y_1,y_2,y_3,y_4)&=y_1^2 y_4^2-3 y_2^2 y_3^2 + 4 y_1 y_3^3 + 4 y_2^3 y_4 - 6 y_1 y_2 y_3 y_4\\&=(y_1y_4-y_2y_3)^2+4(y_1y_3-y_2^2)(y_3^2-y_2y_4)\,.
\end{align*}
So we have 
\begin{align*}
\cN_{-1}=&\{(y_0,y_1,\ldots,y_4)\in\Lg_1\mid h(y_1,y_2,y_3,y_4)=0\}\cup\{(y_0,y_1,\ldots,y_4)\in\Lg_1\mid y_0=0\}.
\end{align*}
\begin{lemma}Let $(G,\theta)=(G_2,3_\rs)$. There are $7$ $G_0$-orbits in $\cN_{-1}$ as follows:

\FloatBarrier
\begin{longtable}[c]{p{4cm}p{1.5cm}p{1.5cm}|p{4cm}p{1.6cm}p{1.6cm}}
\hline
Representative $z$&$\on{dim}\cO_z$&$A_{G_0}(z)$&Representative $x$&$\on{dim}\cO_z$&$A_{G_0}(z)$\\
\hline
$z_1:=X_{-\alpha-\beta}+X_{-2\alpha-\beta}$&$4$&$S_3$&$z_5:=X_{-\beta}$&$2$&$1$\\
$z_2:=X_{-\alpha-\beta}+X_{3\alpha+2\beta}$&$4$&$1$&$z_6:=X_{3\alpha+2\beta}$&$1$&$1$\\
$z_3:=X_{-\beta}+X_{3\alpha+2\beta}$&$3$&$\mu_3$&$0$&$0$&$1$\\
$z_4:=X_{-\alpha-\beta}$&$3$&$1$\\
\hline
\end{longtable}
where $\cO_z=G_0.z$ and $A_{G_0}(z)=Z_{G_0}(z)/Z_{G_0}(z)^0$.
\end{lemma}
\begin{proof}
It is well known that the action of $G_0$ on $V(3\omega)$ (the space of binary cubic forms) has 4 orbits. One checks that
\begin{align*}
\cO_{z_1}=&\{(y_0,y_1,\ldots,y_4)\in\Lg_1\mid y_0=0,h(y_1,y_2,y_3,y_4)\neq 0\}\\
\cO_{z_2}=&\{(y_0,y_1,\ldots,y_4)\in\Lg_1\mid y_0\neq 0,h(y_1,y_2,y_3,y_4)=0, \\&(y_1y_4-y_2y_3,y_1y_3-y_2^2,y_3^2-y_2y_4)\neq(0,0,0)\\
\cO_{z_3}=&\{(y_0,y_1,\ldots,y_4)\in\Lg_1\mid y_0\neq 0,(y_1y_4-y_2y_3,y_1y_3-y_2^2,y_3^2-y_2y_4)=(0,0,0),\\&(y_1,y_2,y_3,y_4)\neq(0,0,0,0)\}\\
\cO_{z_4}=&\{(y_0,y_1,\ldots,y_4)\in\Lg_1\mid y_0=0,h(y_1,y_2,y_3,y_4)= 0, \\&(y_1y_4-y_2y_3,y_1y_3-y_2^2,y_3^2-y_2y_4)\neq(0,0,0)\}\\
\cO_{z_5}=&\{(y_0,y_1,\ldots,y_4)\in\Lg_1\mid y_0= 0,(y_1y_4-y_2y_3,y_1y_3-y_2^2,y_3^2-y_2y_4)=(0,0,0),\\&(y_1,y_2,y_3,y_4)\neq(0,0,0,0)\},\qquad
\cO_{z_6}=\{(y_0,0,0,0,0)\in\Lg_1\mid y_0\neq 0\}.
\end{align*}

Using~\eqref{eqn-adjoint}, one checks that the component groups are as desired. In particular, we have
\begin{align}\label{eqn-centraliser}
Z_{G_0}(z_1)=\langle\gamma_1:=h_\beta(-1)x_\alpha(1),\gamma_2:=n_\alpha h_\beta(-1)\rangle\cong S_3
\end{align}
where $\gamma_1^2=\gamma_2^2=1$ and $\gamma_1\gamma_2\gamma_1=\gamma_2\gamma_1\gamma_2$.
\end{proof}
\begin{remark}
The above calculation can be done using the pre-homogeneous vector space approach. Here we rewrite the representatives using root vectors so that it is easier to use when we perform parabolic induction later.
\end{remark}
Let $\cE_i$, $i=1,2,3$ denote the $G_0$-equivariant local system on $\Lg_1^{rs}$ given by the non-trivial irreducible representations of $\cH_{W}=\cH_{\mu_6,\Phi_1^3\Phi_2\Phi_6}$.
\begin{proposition}\label{prop-g2-3s}
\hfill
\begin{enumerate}
\item The following 5 simple $G_0$-equivarant perverse sheaves on $\cN_{-1}$ are cuspidal:
\beqn
\on{IC}(\cO_{z_1},\cL_{\on{sgn}}),\ \on{IC}(\cO_{z_3},\cL_{i}), i=1,2,\ \on{IC}(\cO_{z_5},\bC),\ \bC_{\{0\}}
\eeqn
where $\cL_{\on{sgn}}$ denotes the $G_0$-equivariant local system on $\cO_{z_1}$ given by the sign representation of $S_3$, and $\cL_{i}$ are the $G_0$-equivariant local system on $\cO_{z_3}$ given by the nontrivial representations of $\mu_3$.
\item The $6$ remaining simple $G_0$-equivarant perverse sheaves on $\cN_{-1}$ are bi-orbital.
\item
We have
\begin{align*}
&\{\fF\on{IC}(\cO_{z_1},\cL_{\on{sgn}}),\,\fF\on{IC}(\cO_{z_3},\cL_{i}), i=1,2\}
=\{\on{IC}(\Lg_1^{rs},\cE_i), i=1,2,3\}\\
&
\fF\on{IC}(\cO_{z_5},\bC)=\on{IC}(\Lg_1^{rs},\cE_4)\text{ where $\on{rank}\cE_4=2$.}
\end{align*}

\end{enumerate}
\end{proposition}
\begin{corollary}
We have $\cH_{\chi_1}\cong\cH_{\mu_3,\Phi_1^3}$.
\end{corollary}
\begin{proof}
By Proposition~\ref{prop-g2-3s}, we obtain a unique cuspidal character sheaf from $\cP_{\chi_1}^\dag$. This implies that $\cH_{\chi_1}$ has a unique irreducible representation, which forces the Hecke relation to be $\Phi_1^3=0$.
\end{proof}
To prove Proposition~\ref{prop-g2-3s}, consider the following $B_0:=B\cap G_0$-stable subspaces of $\Lg_{-1}$:
\begin{align*}
&E_1=\Lg_{-\beta}\oplus\Lg_{-\alpha-\beta}\oplus\Lg_{-2\alpha-\beta}&&E_2=\Lg_{-\beta}\oplus\Lg_{-\alpha-\beta}\oplus\Lg_{3\alpha+2\beta}
&
&E_3=\Lg_{-\beta}\oplus\Lg_{3\alpha+2\beta}\\&E_4=\Lg_{-\beta}\oplus\Lg_{-\alpha-\beta}&
&E_5=\Lg_{3\alpha+2\beta}&&E_6=\Lg_{-\beta}.
\end{align*}
Let $E_i^\perp\subset\Lg_1$ denote the orthogonal space  of $E_i\subset\Lg_{-1}$. Note that $E_i\subset\cN_{-1}$, $i=1,\ldots,6$, and $E_i^\perp\subset \cN_1$, $i=1,\ldots,5$, and $E_6^\perp\not\subset\cN_1$. Let
\begin{align*}
\pi_i\colon &G_0\times^{B_0}E_i\to\cN_{-1}\subset\Lg_{-1},\ (g,y)\mapsto g.y\\
\check\pi_i\colon& G_0\times^{B_0}E_i^\perp\to\Lg_1,\ (g,x)\mapsto g.x.
\end{align*}
We have
\beq\label{eqn-ft}
\fF(\pi_i)_*\bC\cong(\check\pi_i)_*\bC[-]\text{ ($-$ denotes appropriate shift).}
\eeq

It is now easy to check that the sheaves in part (2) of Proposition~\ref{prop-g2-3s} occur in some $\fF(\pi_i)_*\bC$ as a direct summand (up to shift), for $i=1,\ldots,5$. So part (2) follows. For example, let us consider $\pi_1$. One checks that
\beqn
\pi_1^{-1}(z_1)=\{u_0:=B_0,u_1:=n_\alpha B_0,u_2:=x_\alpha(-1)n_\alpha B_0\}
\eeqn
and $\gamma_1:(u_0,u_1,u_2)\mapsto(u_0,u_2,u_1),\ \gamma_2:(u_0,u_1,u_2)\mapsto(u_1,u_0,u_2)$, where $\gamma_1,\gamma_2\in Z_{G_0}(z_5)$, see~\eqref{eqn-centraliser}. Hence $\on{IC}(\cO_{z_1},\bC\oplus\cL)$ is a direct summand of $(\pi_1)_*\bC[-]$, where $\cL$ is the rank $2$ local system on $\cO_{z_1}$ given by the standard representation of $S_3$. Since $E_1^\perp\subset\cN_1$, we conclude, in view of~\eqref{eqn-ft}, that $\fF(\on{IC}(\cO_{z_1},\bC))$ and $\fF(\on{IC}(\cO_{z_1},\cL))$ are supported on nilpotent elements. That is, $\on{IC}(\cO_{z_1},\bC)$ and $\on{IC}(\cO_{z_1},\cL)$ are bi-orbital.

Consider $\pi_6$ and $\check\pi_6$. Let $z=X_{\alpha+\beta}+X_{2\alpha+\beta}+X_{-3\alpha-2\beta}\in\Lg_1^{rs}$.
 One checks that $\check\pi_6^{-1}(z)=\{B_0,n_\alpha B_0,x_\alpha(1)n_\alpha\}$ and 
\begin{align*}
&(\pi_6)_*\bC[-]\cong\on{IC}(\cO_{z_5},\bC)\oplus \bC_{\{0\}}
&&(\check\pi_6)_*\bC[-]\cong\on{IC}(\Lg_1^{rs},\bC\oplus\cE)
\end{align*}
where $\cE$ is a rank $2$ local system. Since $\fF \bC_{\{0\}}\cong\on{IC}(\Lg_1^{rs},\bC)$, it follows that $$\fF\on{IC}(\cO_{z_5},\bC)\cong\on{IC}(\Lg_1^{rs},\cE).$$
Since $\cP_{\chi_0}^\dag=\on{IC}(\Lg_1^{rs},\cH_W)$ gives rise to the cuspidal character sheaves $\on{IC}(\Lg_1^{rs},\bC)$ and $\on{IC}(\Lg_1^{rs},\cE_i)$, $i=1,2,3$ with $\on{rank}(\cE_i)=1$, part (1) and (3) of Proposition~\ref{prop-g2-3s} follow. This completes the proof of Proposition~\ref{prop-g2-3s} and hence that of Theorem~\ref{mainthm} for $(G_2,3_\rs)$.

\end{document}